\documentclass[10pt, leqno]{article}
\usepackage{geometry}           
\usepackage{amsmath, amsthm, amssymb, titlesec, titletoc}

\newtheorem{theorem}{Theorem}[section]
\newtheorem{corollary}[theorem]{Corollary}
\newtheorem{lemma}[theorem]{Lemma}
\newtheorem{proposition}[theorem]{Proposition}
\theoremstyle{definition}
\newtheorem{definition}[theorem]{Definition}
\newtheorem{remark}[theorem]{Remark}

\numberwithin{equation}{section}

\titleformat{\section}[block]{\scshape\filcenter}{\thesection.}{3pt}{}
\titleformat{\subsection}[block]{\scshape\filcenter}{\thesubsection.}{3pt}{}
\titlecontents{section}[0pt]{}{\thecontentslabel. \,}{}{\hfill\contentspage}

\title{\large{\textbf{ROTATIONAL SYMMETRY OF ANCIENT SOLUTIONS TO THE RICCI FLOW IN HIGHER DIMENSIONS}}}
\author{\textsc{\small{Simon Brendle and Keaton Naff}}}
\date{}

\begin{document}

\maketitle

\abstract{We extend the second part of \cite{Bre20} on the uniqueness of ancient $\kappa$-solutions to higher dimensions. In dimensions $n \geq 4$, an ancient $\kappa$-solution is a nonflat, complete, ancient solution of the Ricci flow that is uniformly PIC and weakly PIC2; has bounded curvature; and is $\kappa$-noncollapsed. We show that the only noncompact ancient $\kappa$-solutions up to isometry are a family of shrinking cylinders, a quotient thereof, or the Bryant soliton.}

\section{Introduction}

Recently, there have been significant developments in the understanding of singularity models of the Ricci flow in dimension three. By work of Hamilton \cite{Ham95} and Perelman \cite{Per02}, we have known for some time that singularities of the Ricci flow in dimension three are modeled on ancient $\kappa$-solutions. An ancient $\kappa$-solution is a complete, nonflat, $\kappa$-noncollapsed, ancient solution of Ricci flow which has bounded, nonnegative curvature. Perelman gave a very good qualitative description of these singularity models. In the noncompact setting, Perelman proved that an ancient $\kappa$-solution must either be a quotient of a family of shrinking round cylinders, or the solution must have the structure of a tube attached to a positively curved cap. Perelman conjectured the latter solution must be isometric to the rotationally symmetric Bryant soliton. The first step toward proving this conjecture was accomplished in \cite{Bre13}, where the first author proved that the Bryant soliton is the unique noncompact, nonflat, $\kappa$-noncollapsed, steady gradient Ricci soliton in dimension three. The proof of Perelman's conjecture in its full generality was recently completed by the first author in \cite{Bre20}. 

In the compact setting, there is a related conjecture that an ancient $\kappa$-solution must either be a quotient of a family of shrinking round spheres, or be isometric to a so-called ancient oval, constructed by Perelman in \cite{Per03}.  At very negative times, an ancient oval looks approximately like two Bryant solitons which have been cut along a cross-sectional sphere far from the tip and glued together to form a compact solution. The proof of uniqueness in the compact setting presents several unique challenges not present in the proof in the noncompact setting. In particular, the ancient oval is not a self-similar solution of the Ricci flow and requires a careful analysis of its asymptotics for lack of a closed-form description of the solution. These challenges have been addressed in a pair of papers, \cite{ABDS19} and \cite{BDS21}, by Angenent, Daskalopoulos, Sesum, and the first author. By the combination of \cite{Bre20} and \cite{BDS21}, we now know that, up to isometries and quotients, the only ancient $\kappa$-solutions in dimension three are the shrinking round sphere, the shrinking round cylinder, the ancient oval solution, and the steady Bryant soliton. 

The classification in dimension three is special. The only singularities of the Ricci flow that occur in dimension three are spherical singularities (modeled on $S^3$), neck-pinch singularities (modeled on $S^2 \times \mathbb{R}$), and degenerate neck-pinch singularities (modeled on the Bryant soliton). For the purpose of performing surgery, it is important that the noncompact singularity models have a compact cross-section modeled on $S^2$. In this sense, through dimension reduction, the simple classification of $\kappa$-solutions in dimension three is in part due to the even simpler classification in dimension two: the only ancient $\kappa$-solution in dimension two is the shrinking round sphere (and its quotient). 

Now let us turn our attention to higher dimensions $n \geq 4$. In dimension four, the class of singularities models is not as simple. For example, a singularity of the flow might be modeled upon the product of a three-dimensional Bryant soliton with a line or a bubble-sheet cylinder $S^2 \times \mathbb{R}^2$. In general, singularities modeled on these ancient solutions will occur, unless one imposes a preserved curvature condition on the initial data to rule them out. In his fundamental work \cite{Ham97}, Hamilton showed that positive isotropic curvature (PIC) is the right curvature assumption on the initial metric that allows for only spherical, neck-pinch, and degenerate neck-pinch singularities to occur. In particular, Hamilton developed an analogue of the very important Hamilton-Ivey pinching condition for initial data that is PIC in dimension four. By combining the work of Perelman and Hamilton, Chen and Zhu in \cite{CZ06} showed singularities models for Ricci flows of PIC initial data in dimension four have the same qualitative description as singularity models in dimension three. 

There is a conjecture that a similar singularity phenomenon holds in all dimensions $n \geq 4$ when the initial data is PIC. The first author confirmed this conjecture in dimensions $n \geq 12$ in \cite{Bre19}. A key new ingredient in \cite{Bre19} is a pinching result for PIC initial data when $n \geq 12$. The result shows that singularity models for PIC initial data must be uniformly PIC and weakly PIC2. See Section 2 to recall the definitions. It is expected that there should be a similar pinching estimate for the missing dimensions $5 \leq n \leq 11$. In any case, together with Perelman's noncollapsing result, one concludes that an appropriate definition of an ancient $\kappa$-solution in the PIC setting for $n \geq 4$ is the following:
 
 \begin{definition}\label{kappa_solution}
Suppose $n \geq 4$. An $n$-dimensional ancient $\kappa$-solution is an ancient, complete, nonflat solution of the Ricci flow that is uniformly PIC and weakly PIC2; has bounded curvature; and is $\kappa$-noncollapsed on all scales.
\end{definition}

The purpose of this work is to confirm that the main uniqueness result of \cite{Bre20} holds in higher dimensions for $\kappa$-solutions in the sense above. The proof of the conjecture in \cite{Bre20} is accomplished in two parts. In the first part, the first author proves the Bryant soliton is unique among the class of rotationally symmetric metrics. 

\begin{theorem}[\cite{Bre20}]
Assume that $(M, g(t))$ is a three-dimensional ancient $\kappa$-solution which is noncompact and has positive sectional curvature. If $(M, g(t))$ is rotationally symmetric, then $(M, g(t))$ is isometric to the Bryant soliton up to scaling. 
\end{theorem}

This part of \cite{Bre20} has recently been extended to higher dimensions by Li and Zhang in \cite{LZ20}. 

\begin{corollary}[\cite{LZ20}]\label{li_zhang_pic}
Assume $n \geq 4$ and that $(M, g(t))$ is an $n$-dimensional ancient $\kappa$-solution, which is noncompact and strictly PIC2. If $(M, g(t))$ is rotationally symmetric, then $(M, g(t))$ is isometric to the Bryant soliton up to scaling. 
\end{corollary}

Technically, these authors proved the theorem above under the assumption that the $\kappa$-solution has positive curvature operator. In the rotationally symmetric setting, however, a curvature operator that is PIC2 is automatically positive, so their result clearly implies uniqueness under the weaker assumption (see Lemma \ref{li-zhang-pic}).

In the second part of \cite{Bre20}, the first author proves the $\kappa$-solution must be rotationally symmetric.

\begin{theorem}[\cite{Bre20}]
Assume that $(M, g(t))$ is a three-dimensional ancient $\kappa$-solution which is noncompact and has positive sectional curvature. Then $(M, g(t))$ is rotationally symmetric. 
\end{theorem}

Extending this theorem to higher dimensions is the main result of this paper:

\begin{theorem}\label{main}
Assume $n \geq 4$ and that $(M, g(t))$ is an $n$-dimensional ancient $\kappa$-solution which is noncompact and strictly PIC2. Then $(M, g(t))$ is rotationally symmetric. 
\end{theorem}

Dropping the strictly PIC2 assumption and combining the theorems above, we draw the following corollary. 

\begin{corollary} 
Assume $n \geq 4$ and that $(M, g(t))$ is an $n$-dimensional ancient $\kappa$-solution which is noncompact. Then $(M, g(t))$ is isometric to either a family of shrinking cylinders, a quotient thereof, or to the Bryant soliton. 
\end{corollary} 

The proof of Theorem 1.4 is rather robust and much of it extends to higher dimensions without change or with minor modifications. We have tried to follow the proof in \cite{Bre20} closely so that readers already familiar with this work will recognize the arguments here. For the convenience of this reader, we have enumerated in Section 2 some of the minor differences between the setting in dimension three and the setting here. For readability and rigor, we have opted to include full proofs of the important results in \cite{Bre20}. Namely, we include the spectral analysis for the parabolic Lichnerowicz equation on the cylinder, the proof of the Neck Improvement Theorem, and the proof of the rotational symmetry of the ancient solution. These proofs occupy Sections 3, 4, and 5 respectively. For less significant results, where the proofs in \cite{Bre20} really do only require notational modification, we have omitted the details.

In the remainder of Section 2, we recall relevant definitions, as well as several of the preliminary results from \cite{Bre20} that carry over without modification to higher dimensions. Many of the results in \cite{Bre20} rely on Perelman's celebrated structure theorem for noncompact ancient $\kappa$-solutions in dimension three. We will need the analogous of this and some additional results for ancient $\kappa$-solutions in higher dimensions. These are included in Appendix A. We will also need a variant of the important Anderson-Chow estimate for solutions of the parabolic Lichnerowicz equation in higher dimensions. We will explain this variant in Appendix B. Finally, to serve as a reference for the reader, we have also included an adaptation of Shi's interior estimates for tensors (used often in \cite{Bre20} and here) in Appendix C. 

We would like to conclude this introduction by directing the reader to some very interesting recent results in the mean curvature flow that are related to the results discussed above. If the second fundamental form of a hypersurface in $\mathbb{R}^{n+1}$ is two-convex, then the induced metric on the hypersurface has  positive isotropic curvature. In particular, one expects the singularity behavior of the mean curvature flow of two-convex hypersurfaces to be similar to the singularity behavior of the Ricci flow of metrics with positive isotropic curvature. Indeed, by work of Angenent, Choi, Daskalopoulos, Sesum, and the first author in \cite{ADS19},\cite{ADS20}, \cite{BC19}, \cite{BC20},  the only ancient, nonflat, noncollapsed, uniformly two-convex solutions of the mean curvature flow are the shrinking round sphere, the shrinking round cylinder, the ancient ovals (constructed by White \cite{Whi03} and later by Haslhofer and Hershkovits \cite{HH16}), and the translating bowl soliton. The Neck Improvement Theorem in \cite{Bre20} and here, as well as the asymptotic analysis of the ancient oval in the Ricci flow \cite{ABDS19} are based on the analogous results established first for the mean curvature flow.  \\

\textbf{Acknowledgements.} This project was supported by the National Science Foundation under grant DMS-1806190 and by the Simons Foundation.


\section{Preliminaries}

\subsection{Definitions}

We begin by recalling the notions of uniformly PIC, weakly PIC2, and strictly PIC2. 

\begin{definition}\label{pic}
Suppose $n \geq 4$ and $(M, g)$ is a Riemannian manifold of dimension $n \geq 4$. 

\begin{itemize}
\item We say that $(M,g)$ is uniformly PIC if there exists a real number $\alpha > 0$ with the property that $R(\varphi,\bar{\varphi}) \geq \alpha \, |\text{\rm Rm}| \, |\varphi|^2$ for all complex two-vectors of the form $\varphi = (e_1+ie_2) \wedge (e_3+ie_4)$, where $\{e_1,e_2,e_3,e_4\}$ is an orthonormal four-frame. 

\item We say that $(M,g)$ is weakly PIC2 if $R(\varphi,\bar{\varphi}) \geq 0$ for all complex two-vectors of the form $\varphi = (e_1+i\mu e_2) \wedge (e_3+i\lambda e_4)$, where $\{e_1,e_2,e_3,e_4\}$ is an orthonormal four-frame and $\lambda,\mu \in [0,1]$. If the inequality is always strict, we say that $(M,g)$ is strictly PIC2. 
\end{itemize}
\end{definition}

\begin{remark}\label{uniformly-PIC-dim4}
When $n \geq 5$, the PIC condition implies that $|\mathrm{Rm}| \leq CR$ for some universal constant $C$, where $R$ denotes the scalar curvature. This is because the PIC condition implies (1) $\mathrm{Ric} \leq \frac{1}{2}Rg$ for $n \geq 5$ (see Lemma A.5 in \cite{Bre19}) and (2) $\mathrm{Ric} = 0 \iff \mathrm{Rm} = 0$ for $n \geq 4$ (see Proposition 7.3 in \cite{Bre10}). So if $n\geq 5$, it is equivalent to assume $R(\varphi, \bar{\varphi}) \geq \alpha R |\varphi|^2$, as is sometimes done, in the definition of uniformly PIC above. For $n = 4$, the PIC condition does not imply the scalar curvature controls the norm of the Riemann tensor. In fact, it is more natural to phrase Definition \ref{pic} in dimension four using notation introduced by Hamilton in \cite{Ham97}. In dimension four, the space of bivectors $\Lambda^2\mathbb{R}^4 \cong \mathrm{so}(4)$ naturally decomposes into self-dual and anti-self-dual subspaces $\Lambda_+ \oplus \Lambda_- \cong \mathrm{so}(3) \oplus \mathrm{so}(3)$. Under this decomposition, the curvature operator $R : \Lambda^2 \to \Lambda^2$ can be expressed as a block matrix 
\[
R = \begin{bmatrix} A & B \\ B^t & C \end{bmatrix}.
\]
Let $a_1 \leq a_2 \leq a_3$ and $c_1 \leq c_2 \leq c_3$ denote the eigenvalues of the $A$ and $C$ matrices respectively. Let $0 \leq b_1 \leq b_2 \leq b_3$ denote the eigenvalues of $\sqrt{B^t B}$. In four dimensions, $(M^4, g)$ is uniformly PIC if and only if there exists a real number $\alpha > 0$ with the property that
\[
\min\{a_1 + a_2, c_1 + c_2 \} \geq \alpha \max\{a_3, b_3, c_3 \}
\]
everywhere on $M$. 
\end{remark}

In \cite{LZ20}, the authors use a slightly different definition for $\kappa$-solutions than Definition \ref{kappa_solution}. Namely, the authors require the solution have nonnegative curvature operator in place of the uniformly PIC and weakly PIC2 assumptions. However, if $(M, g)$ is rotationally symmetric and strictly PIC2, then it must also have positive curvature operator. 

\begin{lemma}\label{li-zhang-pic}
Suppose that $n \geq 4$ and $(M, g)$ is a $n$-dimensional, rotationally symmetric, complete Riemannian manifold whose curvature operator is strictly PIC2. Then $(M, g)$ has positive curvature operator. 
\end{lemma}
\begin{proof}
The curvature operator of a rotationally symmetric warped product $g = dz^2 + \phi(z)^2 g_{S^{n-1}}$ is diagonalizable. In particular, the Weyl tensor vanishes and algebraically $R_{ijkl} = \delta_{ik} S_{jl} + \delta_{jl} S_{ik} - \delta_{il} S_{jk} - \delta_{jk} S_{il}$ where the $(0, 2)$-tensor $S$ is the Schouten tensor. The strictly PIC2 condition implies that $(M,g)$ has positive sectional curvature, which in the rotationally symmetric case implies the curvature operator is positive. 
\end{proof}

\subsection{The Parabolic Lichnerowicz Equation}

Recall that if $h$ is a symmetric $(0, 2)$-tensor on a Riemannian manifold $(M, g)$, then the Lichnerowicz Laplacian $\Delta_{L, g}$ is defined 
\[
(\Delta_{L, g}h)_{ik} = (\Delta h)_{ik} + 2 R_{ijkl}h^{jl} - \mathrm{Ric}_{il}h^l_k - \mathrm{Ric}_{kl}h^l_i. 
\]
We restate three results from Section 5 of \cite{Bre20} on the parabolic Lichnerowicz equations that we will use later. 

\begin{lemma}\label{linearization_of_ricci}
Let $g$ be a Riemannian metric on a manifold $M$ and let $V$ be a vector field. Define $h:= \mathcal L_V(g)$ and $Z := \mathrm{div}(h) - \frac{1}{2} \nabla \mathrm{tr}(h)$. Note $Z = \Delta V + \mathrm{Ric}(V)$. Then
\[
\mathcal L_{V}(\mathrm{Ric}) = -\frac{1}{2} \Delta_{L, g} h + \frac{1}{2} \mathcal L_{Z}(g)
\]
where $\mathrm{Ric}$ is viewed as a $(0,2)$-tensor. 
\end{lemma}

\begin{theorem}\label{pde_for_lie_derivative}
Suppose that $g(t)$ is a solution to the Ricci flow on a manifold $M$. Moreover, suppose that $V(t)$ is a family of vector fields satisfying $\frac{\partial}{\partial t} V(t) = \Delta_{g(t)} V(t) + \mathrm{Ric}_{g(t)}(V(t))$. Then the Lie derivative $h(t) := \mathcal L_{V(t)}(g(t))$ satisfies the parabolic Lichnerowicz equation 
\[
\frac{\partial}{\partial t} h(t) = \Delta_{L, g(t)} h(t).
\]
\end{theorem}

\begin{proposition}\label{norm_pde}
Suppose that $g(t)$ is a solution to the Ricci flow on a manifold $M$. Moreover, suppose that $V(t)$ is a family of vector fields satisfying $\frac{\partial}{\partial t} V = \Delta V +\mathrm{Ric}(V)+Q$. Then
\[
\frac{\partial}{\partial t} |V| \leq \Delta |V|+|Q|.
\]
\end{proposition}

The parabolic Lichnerowicz equation is the linearization of the Ricci-DeTurck flow with respect to an evolving family of metrics. The following definition and two propositions makes precise what this means. 

\begin{definition}
Suppose $\bar g(t)$ and $\tilde g(t)$ are smooth families of metrics on a manifold $M$. We say $\tilde g(t)$ is a solution of the Ricci-DeTurck flow with respect to the evolving metrics $\bar g(t)$ if 
\[
\frac{\partial}{\partial t} \tilde g(t) = -2\mathrm{Ric}_{\tilde g(t)} - \mathcal{L}_{\xi_t}(\tilde g(t))
\]
where $\xi_t := \Delta_{\tilde g(t), \bar g(t)} \mathrm{Id}$ and $\Delta_{\tilde g(t), \bar g(t)}\mathrm{Id}$ denotes the map Laplacian of $\mathrm{Id} :(M, \tilde g(t)) \to (M, \bar g(t))$. 
\end{definition}

\begin{proposition}\label{ricci-to-ricci-deturck}
Let $\bar g(t)$ be a smooth family of background metrics on a manifold $M$. Suppose $g(t)$ is a solution to the Ricci flow on $M$. Assume $\Phi_t$ solves the harmonic map heat flow $\frac{\partial}{\partial t} \Phi_t = \Delta_{g(t), \bar g(t)} \Phi_t$ with initial condition $\Phi_0 = \mathrm{Id}$ and, if necessary, Dirichlet boundary condition $\Phi_t|_{\partial M} = \mathrm{Id}$. The family of metrics $\tilde g(t)$ defined by the condition $g(t) = \Phi_t^\ast \tilde g(t)$ is a solution to the Ricci-DeTurck flow with respect to the family of background metrics $\bar g(t)$. 
\end{proposition}
Observe that if $\bar g(t)$ is a solution to the Ricci flow, then $\bar g(t)$ is also a solution to Ricci-DeTurck flow with respect to the evolving metrics $\bar g(t)$. This follows immediately because $\Delta_{\bar g(t),\bar g(t)}\mathrm{Id} = 0$ and hence $\Phi_t \equiv \mathrm{Id}$ is a solution of the harmonic map heat flow.

\begin{proposition}\label{linearization-of-evolving-ricci-deturck}
Let $\bar g(t)$ be a smooth family of background metrics on a manifold $M$. Suppose $\tilde g_s(t)$, $s \in (-\varepsilon, \varepsilon)$, is smooth family of solutions on $M$ to the Ricci-DeTurck flow with respect to the family of background metrics $\bar g(t)$. Assume $\tilde g_{0}(t) = \bar g(t)$ and let $h(t) = \frac{d}{ds} \tilde g_s(t) \big|_{s =0}$ and $\xi_{s, t} = \Delta_{\tilde g_s(t), \bar g(t)}\mathrm{Id}$. Then 
\[ \frac{d}{ds}\Big(-2\mathrm{Ric}_{\tilde g_s(t)} - \mathcal{L}_{\xi_{s, t}}(\tilde g_s(t))\Big)\Big|_{s = 0}= \Delta_{L, \bar g(t)} h(t). \]
Suppose $\Phi_{s, t}$ is a family of evolving diffeomorphisms of $M$ for $s \in (-\varepsilon, \varepsilon)$ with $\Phi_{0, t} \equiv \mathrm{Id}$. Let $V(t) = \frac{d}{ds} \Phi_{s, t} \big|_{s = 0}$. Then 
\[ \frac{d}{ds} (\Delta_{\bar g(t), \bar g(t)} \Phi_{s, t} )\big|_{s = 0} = \Delta_{\bar{g}(t)} V(t) + \mathrm{Ric}_{\bar{g}(t)}(V(t)).\]
\end{proposition}

\subsection{Gluing Almost Killing Vector Fields}

In this subsection, we restate the results from Section 7 of \cite{Bre20} for higher dimensions. The proofs of the these results in dimension three work verbatim in higher dimension, after making notational modifications to the statements and arguments. For this reason, we will not repeat the proofs here. Throughout this section $C:= C(n)$ always denotes a universal constant depending only upon the dimension. We let the metric $\bar g = g_{S^{n-1}} + dz \otimes dz$ denote the standard round metric on the cylinder of constant sectional curvature $1$ and thus of radius $1$. The value of $\varepsilon_0$ in the following two results is determined by contradiction arguments used to prove the results of Section 7 in \cite{Bre20}, as well as Hamilton's proof of the existence of a canonical CMC foliation once $g$ is sufficiently close to $\bar g$ in \cite{Ham95}.

\begin{proposition}\label{uniqueness_of_symmetry_fields}
If $\varepsilon_0$ is sufficiently small, depending only upon the dimension $n$, then the following holds. Let $g$ be a Riemannian metric on the cylinder $S^{n-1} \times [-20, 20]$ such that $\lVert g- \bar g\rVert_{C^{10}} \leq \varepsilon_0$ and let $\varepsilon \leq \varepsilon_0$. Let $\bar x$ be a point on the center slice $S^{n-1} \times \{0\}$. Let $\Sigma$ denote the leaf of Hamilton's CMC foliation with respect to $g$ that passes through $\bar x$ and let $\nu$ denote its unit normal vector. Suppose that $\mathcal U := \{ U^{(a)} : 1 \leq a \leq {n \choose 2} \}$ is a family of vector fields with the following properties:
\begin{enumerate}
\item[$\bullet$] $\sup_{B_g(\bar x, 12)} \sum_{a =1}^{n \choose 2} |\mathcal L_{U^{(a)}}(g)|^2 + |D(\mathcal L_{U^{(a)}}(g))|^2 \leq \varepsilon^2$;
\item[$\bullet$] $\sup_{\Sigma} \sum_{a =1}^{{n \choose 2}} |\langle U^{(a)}, \nu \rangle |^2 \leq \varepsilon^2$;
\item[$\bullet$] $\sum_{a, b =1}^{{n \choose 2}} \big|\delta_{ab} - \mathrm{area}_g(\Sigma)^{-\frac{n+1}{n-1}} \int_{\Sigma} \langle U^{(a)}, U^{(b)} \rangle d\mu_g\big|^2 \leq \varepsilon^2$. 
\end{enumerate}
Suppose that $\tilde{\mathcal U} := \{ \tilde U^{(a)} : 1 \leq a \leq {n\choose 2} \}$ is a second family of vector fields satisfying the same three properties above (with $U^{(a)}$ replaced by $\tilde U^{(a)}$). 
Then there exists an ${n \choose 2} \times {n \choose 2}$ matrix $\omega \in O({n \choose 2})$ such that 
\[
\sup_{B_{g}(\bar x, 9)} \sum_{a = 1}^{{n \choose 2}} \Big| \sum_{b =1}^{{n \choose 2}} \omega_{ab} U^{(b)} - \tilde U^{(a)} \Big|^2 \leq C \varepsilon^2.
\]
\end{proposition}

\begin{corollary}\label{gluing_vectors}
Let $\bar g$, $g$, $\bar x$, $\mathcal U$, $\tilde {\mathcal U}$, $\Sigma$, and $\nu$ satisfy the hypotheses of Proposition \ref{uniqueness_of_symmetry_fields}. We can find a suitable smooth cutoff function $\eta$ such that $\eta \equiv 1$ on $S^{n-1} \times [-20, -1]$ and $\eta \equiv 0$ on $S^{n-1} \times [1, 20]$ and an ${n \choose 2} \times {n \choose 2}$ matrix $\omega \in O({n \choose 2})$ with the property that the vector fields $V^{(a)} := \eta \sum_{b =1}^{{n \choose 2}} \omega_{ab} U^{(b)} + (1-\eta)\tilde U^{(a)}$ satisfy $\sum_{a =1}^{{n \choose 2}} |\mathcal L_{V^{(a)}}(g)|^2 + |D (\mathcal L_{V^{(a)}}(g))|^2 \leq C \varepsilon^2$ in the transition region $S^{n-1} \times [-1, 1]$. 
\end{corollary}

Note in particular, that the vector fields $V^{(a)}$ satisfy the better estimate $\sum_{a =1}^{{n \choose 2}} |\mathcal L_{V^{(a)}}(g)|^2 + |D (\mathcal L_{V^{(a)}}(g))|^2 \leq \varepsilon^2$ outside of the transition region.

\subsection{Differences in Higher Dimensions and Dimension Three}

For the convenience of the reader already familiar with the proofs in \cite{Bre20}, we would like to give some specific direction of where to look for differences between the proof in dimension three and the proof here. For the most part, these differences are minor. The most significant difference is the replacement of the Anderson-Chow estimate, the last item below. 

\begin{enumerate}
\item In Proposition \ref{mode_analysis}, there is a non-rotationally symmetric solution of the parabolic Lichnerowicz equation over the cylinder that (in norm) remains bounded in dimension three, but grows in higher dimensions. For symmetry to improve along the flow, the presence of a non-symmetric growing mode of the linearization could be trouble. Fortunately, since this solution is still the Lie derivative of the metric along a vector field, we can remove it by hand in Step 6 of the proof of the Neck Improvement Theorem (although in higher dimensions the conformal killing vector field $\xi$ must now be time-dependent). We discuss the presence of the non-decaying solutions in Remark 3.2. 
\item In dimension three, the Lie algebra structure of rotational vector fields is used to eliminate the other non-decaying solutions of the parabolic Lichnerowicz equation over the cylinder in the proof of the Neck Improvement Theorem. To take advantage of this structure in higher dimensions, we use a simple lemma (Lemma \ref{lie_algebra}) to bound the structure constants. This slightly modifies the vector fields we define in Steps 7 and 10 of the proof of the Neck Improvement Theorem. 
\item The proof of Proposition 5.7 (which corresponds to Proposition 9.7 in \cite{Bre20}) requires that we verify an ``asymptotically cylindrical'' condition used to prove the uniqueness of the Bryant soliton among self-similar solutions in \cite{Bre14} in higher dimensions. We do this in Proposition A.10. Also, for our results to apply in dimension four, we need to confirm that the uniformly PIC assumption in Definition \ref{uniformly-PIC-dim4} implies the restricted PIC condition of \cite{CZ06}. We do this in Proposition A.2.
\item The Anderson-Chow estimate \cite{AC05} is an important result that allows one to control solutions of the parabolic Lichnerowicz equation in dimension three. However this estimate is special to three dimensions. In this paper, we use an alternative weighted version of the estimate which is valid in all dimensions. We give a proof of this in Proposition B.1. For its application, see Step 3 in the proof of Proposition \ref{main_prop}. 
\end{enumerate}


\section{Analysis of the Parabolic Lichnerowicz Equation on Shrinking Cylinders} 

In this section, we extend the analysis of the parabolic Lichnerowicz equation on the round cylinder to higher dimensions following Section 6 of \cite{Bre20}. Let $(S^{n-1} \times \mathbb{R}, \bar g(t))$ be the standard solution of the Ricci flow on the round cylinder given by 
\[
\bar g(t) = (-2(n-2)t)g_{S^{n-1}} + dz \otimes dz
\]
for $t < 0$. Here $g_{S^{n-1}}$ is the round metric on $S^{n-1}$ with constant sectional curvature 1. Note that $R_{\bar g(t)} \equiv (n-1)(-2t)^{-1}$ and the radius of the neck is $(-2(n-2)t)^{\frac{1}{2}}$ at time $t$. Let us define
\[
t_n := -\frac{1}{2(n-2)}.
\]
At this time, the scalar curvature of the cylinder is $R_{\bar g(t_n)} \equiv (n-1)(n-2)$ and the radius of the neck is $1$, which is a convenient scale for comparing metric balls and subcylinders. Throughout this section, $C$ always denotes a universal constant depending only upon the dimension, but which may change from line to line. 

\begin{proposition}\label{mode_analysis}
Suppose $L$ is a large positive constant. Let $h(t)$ be a one-parameter family of symmetric $(0,2)$-tensors on the cylinder that is defined in the region $\{|z| \leq \frac{1}{2}L, - \frac{1}{2}L \leq t \leq t_n\}$ and satisfies the parabolic Lichnerowicz equation 
\[
\frac{\partial}{\partial t} h(t) = \Delta_{L, \bar g(t)} h(t).
\]
 Assume that $|h(t)|_{\bar g(t)} \leq 1$ in the region $\{|z| \leq \frac{1}{2}L, - \frac{1}{2}L \leq t \leq -\frac{1}{4}L\}$ and that $|h(t)|_{\bar g(t)} \leq L^{10}$ for all $\{|z| \leq \frac{1}{2}L, - \frac{1}{4}L \leq t \leq t_n\}$. On each slice $S^{n-1} \times \{z\}$, we can decompose $h(t)$ into a sum 
\[
h(t) = \omega (z, t) g_{S^{n-1}} + \chi(z, t) + dz \otimes \sigma(z, t) + \sigma(z, t) \otimes dz + \beta(z, t) dz \otimes dz, 
\]
where $\omega(z, t)$ and $\beta(z, t)$ are scalar functions on $S^{n-1}$, $ \sigma(z, t)$ is a one-form on $S^{n-1}$, and $\chi(z, t)$ is a tracefree symmetric $(0, 2)$-tensor on $S^{n-1}$.  Let $\bar \omega(z, t)$ and $\bar \beta(z, t)$ be the rotationally invariant functions obtained from $\omega(z, t)$ and $\beta(z, t)$ by averaging over the slices $S^{n-1} \times \{z\}$
\[
\int_{S^{n-1} \times \{z\}} (\omega (z, t) - \bar \omega (z, t)) \,\mathrm{dvol}_{S^{n-1}} = \int_{S^{n-1} \times \{z\}} (\beta (z, t) - \bar \beta (z, t))  \,\mathrm{dvol}_{S^{n-1}}= 0 
\]
for  $t \in [-1000, t_n]$ and $z \in [-1000, 1000]$. Then there exists a function $\psi : S^{n-1} \to \mathbb{R}$ (independent of $t$ and $z$) that lies in the span of the first spherical harmonics on $S^{n-1}$, and 
\[
\big| h(t) - \bar \omega (z, t) g_{S^{n-1}} - \bar \beta(z, t) dz \otimes dz - (-t)^{\frac{n-1}{2(n-2)}} \psi g_{S^{n-1}}\big|_{\bar g(t)} \leq C L^{-\frac{1}{2(n-2)}} 
\]
in the region $\{|z| \leq 1000, -1000 \leq t \leq t_n \}$. 
\end{proposition}

\begin{proof}

The parabolic Lichnerowicz equation for $h(t)$ is equivalent to the following system of equations for $\omega(z,t)$, $\chi(z,t)$, $\sigma(z,t)$, and $\beta(z,t)$:
\begin{align*}
\frac{\partial}{\partial t} \omega(z,t) &= \frac{\partial^2}{\partial z^2} \omega(z,t) + \frac{1}{(-2(n-2)t)} \Delta_{S^{n-1}} \omega(z,t),\\
\frac{\partial}{\partial t} \chi(z,t) &= \frac{\partial^2}{\partial z^2} \chi(z,t) + \frac{1}{(-2(n-2)t)}\big( \Delta_{S^{n-1}} \chi(z,t) - 2(n-1)\chi(z,t)\big), \\
\frac{\partial}{\partial t} \sigma(z,t) &= \frac{\partial^2}{\partial z^2} \sigma(z,t) + \frac{1}{(-2(n-2)t)}\big( \Delta_{S^{n-1}} \sigma(z,t) - (n-2)\sigma(z,t)\big), \\
\frac{\partial}{\partial t} \beta(z,t) &= \frac{\partial^2}{\partial z^2} \beta(z,t) + \frac{1}{(-2(n-2)t)} \Delta_{S^{n-1}} \beta(z,t). 
\end{align*}
Since $|h(t)|_{\bar g(t)} \leq 1$ in the region $\{|z| \leq \frac{1}{2}L, - \frac{1}{2}L \leq t \leq -\frac{1}{4}L\}$, we have the estimate
\[
\frac{|\omega(z,t)|}{(-t)} + \frac{|\chi(z,t)|_{g_{S^{n-1}}}}{(-t)} +\frac{|\sigma(z,t)|_{g_{S^{n-1}}}}{(-t)^{\frac{1}{2}}}  + |\beta(z,t)|  \leq C
\]
in the region $\{|z| \leq \frac{1}{2}L, - \frac{1}{2}L \leq t \leq -\frac{1}{4}L\}$. 
Similarly, since $|h(t)|_{\bar g(t)} \leq L^{10}$ in the region $\{|z| \leq \frac{1}{2}L, - \frac{1}{4}L \leq t \leq t_n\}$, we have the estimate
\[
\frac{|\omega(z,t)|}{(-t)} + \frac{|\chi(z,t)|_{g_{S^{n-1}}}}{(-t)} +\frac{|\sigma(z,t)|_{g_{S^{n-1}}}}{(-t)^{\frac{1}{2}}}  + |\beta(z,t)|  \leq CL^{10}
\]
in the region $\{|z| \leq \frac{1}{2}L, - \frac{1}{4}L \leq t \leq t_n \}$. Note that $(-t)^{-1} \leq 2(n-2)$ in this range. 
The powers of $(-t)$ in the first estimate are important and will determine whether specific modes decay. The method of estimation in \cite{Bre20} is not sensitive to the power of $L$ in the second estimate. Therefore, as in \cite{Bre20}, we may use the coarse estimate $L^{10} \leq  CL^{20} (-t)^{-10}$ to introduce additional factors of $(-t)^{-1}$ when needed below. In the following steps, we will decompose each of $\chi(z,t)$, $\sigma(z,t)$, $\beta(z,t)$, and $\omega(z,t)$ into sums over the modes of $\Delta_{S^{n-1}}$ acting on traceless symmetric $(0,2)$-tensors, one-forms, and functions respectively. \\

\textit{Step 1:} We begin by analyzing the equation for $\chi(z,t)$. In this step, unless otherwise noted, $|\cdot|$ and $\langle \cdot\,, \cdot \rangle$ denote the norm and inner product with respect to $g_{S^{n-1}}$. Let $S_j$, $j =1,2, \dots$, denote an orthonormal basis of eigenfunctions for the Laplacian acting on the bundle of tracefree $(0,2)$-tensors on $S^{n-1}$, so that $\Delta_{S^{n-1}} S_j = - \nu_j S_j$. Note that $\nu_j > 0$. Indeed, if $\Delta_{S^{n-1}} S = 0$, then $S$ must be both parallel and traceless, and therefore $S \equiv 0$ in view of the holonomy of $S^{n-1}$. Moreover, $\nu_j \sim j^{\frac{2}{n-1}}$ as $j \to \infty$ (see \cite{BGV04}, Corollary 2.43). We assume these eigenfunctions are normalized so that $\int_{S^{n-1}} |S_j|^2 = 1$ for each $j$. Using the Sobolev estimate $\sup_{S^{n-1}}|S_j| \leq C \lVert S_j\rVert_{H^{n-1}}$ and the equation $\Delta_{S^{n-1}} S_j = - \nu_j S_j$, we can deduce $\sup_{S^{n-1}} |S_j| \leq C \nu_j^{\frac{n-1}{2}} \leq C \nu_j^n$.

Let us write 
\[
\chi(z,t) = \sum_{j=1}^{\infty} \chi_j(z, t) S_j,
\]
where 
\[
\chi_j(z, t) = \int_{S^{n-1}} \langle \chi(z, t), S_j \rangle \, \mathrm{dvol}_{S^{n-1}}. 
\]
Note that $|\chi_j(z,t)| \leq C \sup_{S^{n-1}} |\chi(z, t)|$. Define 
\[
\bar \nu_j := \frac{\nu_j + 2}{2(n-2)}.
\]
 Then $\bar \nu_j \geq \frac{1}{n-2} > 0$, $\bar \nu_j + 1 = \frac{\nu_j + 2(n-1)}{2(n-2)}$, and $\bar \nu_j \sim j^{\frac{2}{n-1}}$ as $j \to \infty$.  The equation for $\chi(z,t)$ implies 
\[
\frac{\partial}{\partial t} \chi_j (z, t) = \frac{\partial^2}{\partial z^2} \chi_j(z, t) - \frac{\bar \nu_j + 1}{(-t)} \chi_j(z,t). 
\]
Hence the function $\hat \chi_j(z, t) := (-t)^{-\bar \nu_j - 1} \chi_j(z, t)$ satisfies 
\[
\frac{\partial }{\partial t} \hat \chi_j(z, t) =\frac{\partial^2}{\partial z^2} \hat \chi_j(z, t). 
\]
From our estimates for $\chi(z,t)$, we obtain 
\[
|\hat \chi_j(z, t)| = (-t)^{-\bar \nu_j} \frac{|\chi_j(z, t)|}{(-t)} \leq C(-t)^{-\bar \nu_j}
\]
for $(z, t) \in [-\frac{1}{2}L, \frac{1}{2}L] \times [-\frac{1}{2}L, - \frac{1}{4}L]$, and 
\[
 |\hat \chi_j(z, t)| = (-t)^{-\bar \nu_j} \frac{|\chi_j(z, t)|}{(-t)} \leq CL^{20}(-t)^{- \bar \nu_j-1}
\]
for $(z, t) \in [-\frac{1}{2}L, \frac{1}{2}L] \times [-\frac{1}{4}L, t_n]$. Now suppose $(z, t) \in [-1000, 1000] \times [-1000, t_n]$. Combining these estimates with the representation formula for solution of the one-dimensional heat equation in the rectangle $[-\frac{1}{4}L, \frac{1}{4}L] \times [-\frac{1}{4}L, t_n]$ (see Appendix \ref{parabolic_equations}), we obtain
\[
|\hat \chi_j(z, t)| \leq C \Big(\frac{1}{4}L\Big)^{-\bar \nu_j }+ C L^{21} \int_{-\frac{1}{4}L}^t e^{-\frac{L^2}{100(t-s)}} (t-s)^{-\frac{3}{2}} (-s)^{- \bar \nu_j-1} \, ds. 
\]
In the interval $(-\frac{1}{4}L, t)$, we have
\[
e^{-\frac{L^2}{200(t-s)}} (t-s)^{-\frac{3}{2}}  \leq Ce^{-\frac{L}{100}} L^{-\frac{3}{2}}  \leq C L^{-1}, 
\]
and so
\[
|\hat \chi_j(z, t)| \leq C \Big(\frac{1}{4}L\Big)^{-\bar \nu_j}+ C L^{20} \int_{-\frac{1}{4}L}^t e^{-\frac{L^2}{200(t-s)}} (-s)^{- \bar \nu_j-1} \, ds. 
\]
For $\chi_j$ this gives
\[
|\chi_j(z, t)| \leq C \Big(\frac{L}{4(-t)}\Big)^{-\bar \nu_j}+ C L^{20} \int_{-\frac{1}{4}L}^t e^{-\frac{L^2}{200(t-s)}} \Big(\frac{s}{t}\Big)^{- \bar \nu_j-1} \, ds. 
\]
for $(z, t) \in [-1000, 1000] \times [-1000, t_n]$. The first term decays exponentially as $j \to \infty$. If we naively estimate the exponential in second term by $Ce^{-\frac{L}{100}}$ and then integrate, we could bound the second term by $C L^{20} e^{-\frac{L}{100}} \frac{1}{\bar \nu_j}$, but this is not even summable given the growth rate of $\bar \nu_j$. We need to be a bit more careful near $t$ to avoid losing exponential decay. As observed in \cite{Bre20}, the following estimate works  
\begin{align*}
\int_{-\frac{1}{4}L}^t e^{-\frac{L^2}{200(t-s)}} \Big(\frac{s}{t}\Big)^{- \bar \nu_j-1} \, ds &= \int_{-\frac{1}{4}L}^{(1+\frac{1}{\sqrt{\bar \nu_j}}) t}  e^{-\frac{L^2}{200(t-s)}} \Big(\frac{s}{t}\Big)^{- \bar \nu_j-1} \, ds + \int_{(1+\frac{1}{\sqrt{\bar \nu_j}})t}^{t}  e^{-\frac{L^2}{200(t-s)}} \Big(\frac{s}{t}\Big)^{- \bar \nu_j-1} \, ds \\
& \leq C \Big(e^{-\frac{L}{100}} (1+\frac{1}{\sqrt{\bar \nu_j}})^{-\bar \nu_j} + e^{-\frac{L^2\sqrt{\bar \nu_j}}{200(-t)}} \Big). 
\end{align*}
Since $\bar \nu_j \sim j^{\frac{2}{n-1}}$, both terms decay exponentially as $j \to \infty$.  
In summary, 
\[
|\chi_j(z, t)| \leq C \Big(\frac{L}{4(-t)}\Big)^{-\bar \nu_j}+ C L^{20}e^{-\frac{L}{100}} (1+\frac{1}{\sqrt{\bar \nu_j}})^{-\bar \nu_j }  +  CL^{20}e^{-\frac{L^2\sqrt{\bar \nu_j}}{200(-t)}} 
\]
for $(z, t) \in [-1000, 1000] \times [-1000, t_n]$. Recalling that $\bar \nu_j \sim \nu_j$ and $\bar \nu_j \geq \frac{1}{n-2}$, we sum over all $j$ to obtain
\[
|\chi(z, t)|_{g_{S^{n-1}}} \leq \Big|\sum_{j=1}^{\infty} \chi_j(z, t) S_j \Big| \leq C \sum_{j =1}^\infty \bar \nu_j^n|\chi_j(z, t)| \leq CL^{-\frac{1}{n-2}}
\]
in the region $\{|z| \leq 1000, -1000 \leq t \leq t_n \}$. \\

\textit{Step 2:} Next we analyze the equation for $\sigma(z,t)$. In this step, unless otherwise noted, $|\cdot|$ and $\langle \cdot\,, \cdot \rangle$ denote the norm and inner product with respect to $g_{S^{n-1}}$. Let $Q_j$, $j = 1, 2, \dots$, denote an orthonormal basis of eigenfunctions for the Laplacian acting on the bundle of one-forms of $S^{n-1}$, so that $\Delta_{S^{n-1}} Q_j = - \mu_j Q_j$. Note by Proposition A.1 in \cite{Bre14}, $\mu_j \geq 1$. As before, after normalizing $\int_{S^{n-1}} |Q_j|^2 = 1$, we get $\sup_{S^{n-1}} |Q_j| \leq C \mu_j^n$ for each $j$. Moreover, $\mu_j \sim j^{\frac{2}{n-1}}$ as $j \to \infty$. Let us write 
\[
\sigma(z,t) = \sum_{j =1}^\infty \sigma_j(z, t) Q_j,
\]
where 
\[
\sigma_j(z,t) = \int_{S^{n-1}} \langle \sigma(z, t), Q_j \rangle \, \mathrm{dvol}_{S^{n-1}}.
\]
Note that $|\sigma_j(z,t)| \leq C \sup_{S^{n-1}} |\sigma(z, t)|$. Define 
\[
\bar \mu_j := \frac{\mu_j}{2(n-2)}.
\]
Then $\bar \mu_j \geq \frac{1}{2(n-2)} > 0$, $\bar \mu_j + \frac{1}{2} = \frac{\mu_j +n -2}{2(n-2)}$, and $\bar \mu_j \sim j^{\frac{2}{n-1}}$. The equation for $\sigma(z,t)$ implies 
\[
\frac{\partial}{\partial t} \sigma_j(z,t) = \frac{\partial^2}{\partial z^2} \sigma_j(z, t) - \frac{\bar \mu_j + \frac{1}{2}}{(-t)} \sigma_j(z,t).
\]
Hence the function $\hat \sigma_j(z, t) := (-t)^{-\bar \mu_j - \frac{1}{2}} \sigma_j(z,t)$ satisfies 
\[
\frac{\partial}{\partial t} \hat \sigma_j(z,t) = \frac{\partial^2}{\partial z^2} \hat \sigma_j(z, t). 
\]
From the estimates for $\sigma(z,t)$, we obtain
\[
|\hat \sigma_j(z, t)| = (-t)^{-\bar \mu_j} \frac{|\sigma_j(z, t)|}{(-t)^{\frac{1}{2}}} \leq C(-t)^{-\bar \mu_j} 
\]
for $(z, t) \in [-\frac{1}{2}L, \frac{1}{2}L] \times [-\frac{1}{2}L, - \frac{1}{4}L]$, and 
\[
 |\hat \sigma(z, t)| = (-t)^{-\bar \mu_j} \frac{|\sigma_j(z, t)|}{(-t)^{\frac{1}{2}}} \leq CL^{20}(-t)^{- \bar \mu_j-1} 
\]
for $(z, t) \in [-\frac{1}{2}L, \frac{1}{2}L] \times [-\frac{1}{4}L, t_n]$. Now suppose $(z, t) \in [-1000, 1000] \times [-1000, t_n]$. Using the representation formula for the one-dimensional heat equation in the rectangle $[-\frac{1}{4}L, \frac{1}{4}L] \times [-\frac{1}{4}L, t_n]$, we obtain 
\[
|\hat \sigma_j(z, t)| \leq C \Big(\frac{1}{4}L\Big)^{-\bar \mu_j }+ C L^{21} \int_{-\frac{1}{4}L}^t e^{-\frac{L^2}{100(t-s)}} (t-s)^{-\frac{3}{2}} (-s)^{- \bar \mu_j-1} \, ds.
\]
Repeating the argument from the previous step yields 
\[
|\sigma_j(z, t)| \leq C \Big(\frac{L}{4(-t)}\Big)^{-\bar \mu_j}+ C L^{20}e^{-\frac{L}{100}} (1+(\bar \mu_j)^{-\frac{1}{2}})^{-\bar \mu_j}  +  CL^{20}e^{-\frac{L^2\sqrt{\bar \mu_j}}{200(-t)}}
\]
for $(z, t) \in [-1000, 1000] \times [-1000, t_n]$. Since $\bar \mu_j \geq \frac{1}{2(n-2)}$, summation over all $j$ gives the estimate 
\[
|\sigma(z, t)|_{g_{S^{n-1}}} \leq \Big|\sum_{j=1}^{\infty} \sigma_j(z, t) S_j \Big| \leq C \sum_{j =1}^\infty \bar \mu_j^n|\sigma_j(z, t)| \leq CL^{-\frac{1}{2(n-2)}}
\]
in the region $\{|z| \leq 1000, -1000 \leq t \leq t_n \}$.\\

\textit{Step 3:} Next we analyze the equation for $\beta(z,t)$. Let $Y_j$, $j = 0, 1, 2, \dots$, denote an orthonormal basis of eigenfunctions for the Laplacian acting on scalar functions of $S^{n-1}$, so that $\Delta_{S^{n-1}} Y_j = - \lambda_j Y_j$. Recall that $\lambda_0 = 0$, $\lambda_1 = \lambda_n = n-1$, and $\lambda_{n+1} = 2n$. As before, assuming that $\int_{S^{n-1}} Y_j^2 = 1$, we get $\sup_{S^{n-1}} |Y_j| \leq C \lambda_j^n$ for each $j \geq 1$. Moreover, $\lambda_j \sim j^{\frac{2}{n-1}}$ as $j \to \infty$. Let us write 
\[
\beta(z,t) = \sum_{j =0}^\infty \beta_j(z, t) Y_j,
\]
 where 
\[
\beta_j(z,t) = \int_{S^{n-1}}  \beta(z, t) Y_j \, \mathrm{dvol}_{S^{n-1}}
\]
Note that $|\beta_j(z,t)| \leq C\sup_{S^{n-1}} |\beta(z, t)|$. In the following, let us restrict our attention to the case $j\geq 1$. Define 
\[
\bar \lambda_j := \frac{\lambda_j}{2(n-2)}.
\]
Then $\bar \lambda_j \geq \frac{n-1}{2(n-2)} > 0$ and $\bar \lambda_j \sim j^{\frac{2}{n-1}}$. The equation for $\beta(z,t)$ implies 
\[
\frac{\partial}{\partial t} \beta_j(z,t) = \frac{\partial^2}{\partial z^2} \beta_j(z, t) - \frac{\bar \lambda_j}{(-t)} \beta_j(z,t).
\]
Hence the function $\hat \beta_j(z, t) := (-t)^{-\bar \lambda_j} \beta_j(z,t)$ satisfies 
\[
\frac{\partial}{\partial t} \hat \beta_j(z,t) = \frac{\partial^2}{\partial z^2} \hat \beta_j(z, t). 
\]
From our estimates for $\beta(z,t)$, we obtain
\[
|\hat \beta_j(z, t)| = (-t)^{-\bar \lambda_j} |\beta_j(z,t)|  \leq C (-t)^{-\bar \lambda_j}
\]
for $(z, t) \in [-\frac{1}{2}L, \frac{1}{2}L] \times [-\frac{1}{2}L, - \frac{1}{4}L]$, and 
\[
|\hat \beta_j(z, t)| = (-t)^{-\bar \lambda_j} |\beta_j(z,t)| \leq CL^{20}(-t)^{- \bar \lambda_j-1}
\]
for $(z, t) \in [-\frac{1}{2}L, \frac{1}{2}L] \times [-\frac{1}{4}L, t_n]$. Now suppose $(z, t) \in [-1000, 1000] \times [-1000, t_n]$. Using the representation formula for the one-dimensional heat equation in the rectangle $[-\frac{1}{4}L, \frac{1}{4}L] \times [-\frac{1}{4}L, t_n]$, we obtain 
\[
|\hat \beta_j(z, t)| \leq C \Big(\frac{1}{4}L\Big)^{-\bar \lambda_j }+ C L^{21} \int_{-\frac{1}{4}L}^t e^{-\frac{L^2}{100(t-s)}} (t-s)^{-\frac{3}{2}} (-s)^{- \bar \lambda_j-1} \, ds. 
\]
Repeating the argument from Step 1 yields
\[
|\beta_j(z, t)| \leq C \Big(\frac{L}{4(-t)}\Big)^{-\bar \lambda_j}+ C L^{20}e^{-\frac{L}{100}} (1+(\bar \lambda_j)^{-\frac{1}{2}})^{-\bar \lambda_j }  +  CL^{20}e^{-\frac{L^2\sqrt{\bar \lambda_j}}{200(-t)}},
\]
for $(z, t) \in [-1000, 1000] \times [-1000, t_n]$. Since $\bar \lambda_j \geq \frac{n-1}{2(n-2)}$ for $j \geq 1$, summation over all $j \geq 1$ gives the estimate 
\[
|\beta(z, t) - \bar \beta(z, t)|  \leq \Big|\sum_{j=1}^{\infty} \beta_j(z, t) S_j \Big| \leq C \sum_{j =1}^\infty \bar \lambda_j^n|\beta_j(z, t)| \leq CL^{-\frac{n-1}{2(n-2)}}
\]
in the region $\{|z| \leq 1000, -1000 \leq t \leq t_n \}$.\\

\textit{Step 4:} Finally, we analyze the equation for $\omega(z,t)$. As above, let $Y_j$, $j = 0,1,2, \dots, $ denote an orthonormal basis of eigenfunctions for the Laplacian acting on scalar functions of $S^{n-1}$, so that $\Delta_{S^{n-1}} Y_j = - \lambda_j Y_j$. Let us write 
\[
\omega(z,t) = \sum_{j=0}^{\infty} \omega_j(z, t) Y_j,
\]
where 
\[
\omega_j(z, t) = \int_{S^{n-1}} \omega(z,t) Y_j \, \mathrm{dvol}_{S^{n-1}}.
\]
Note that $|\omega_j(z,t)| \leq C\sup_{S^{n-1}} |\omega(z, t)|$. Let $\bar \lambda_j = \frac{\lambda_j}{2(n-2)}$ as in the previous step. Note that if $1 \leq j \leq n$, then $\bar \lambda_j - 1 = \frac{3 - n}{2(n-2)} \leq 0$. When $j \geq n+1$, then $\bar \lambda_j -1 \geq \frac{2}{n-2} > 0$. Since $\omega(z,t)$ and $\beta(z,t)$ satisfy the same equation, if we let $\hat \omega_j(z,t) = (-t)^{-\bar \lambda_j} \omega_j(z, t)$, then 
\[
\frac{\partial}{\partial t} \hat \omega_j(z,t) = \frac{\partial^2}{\partial z^2} \hat \omega_j(z, t).
\]
From our estimates $\omega(z,t)$, we obtain
\[
|\hat \omega_j(z, t)| = (-t)^{-\bar \lambda_j+1} \frac{|\omega_j(z,t)|}{(-t)}  \leq C (-t)^{-\bar \lambda_j +1} 
\]
for $(z, t) \in [-\frac{1}{2}L, \frac{1}{2}L] \times [-\frac{1}{2}L, - \frac{1}{4}L]$, and 
\[
|\hat \omega_j(z, t)| = (-t)^{-\bar \lambda_j +1} \frac{|\omega_j(z,t)|}{(-t)} \leq CL^{20}(-t)^{- \bar \lambda_j-1} 
\]
for $(z, t) \in [-\frac{1}{2}L, \frac{1}{2}L] \times [-\frac{1}{4}L, t_n]$.  Notice the first estimate is not suitable when $1 \leq j \leq n$, since in this case $-\bar \lambda_j + 1 > 0$. 

First, we consider the case $j \geq n+1$ so that $\bar \lambda_j -1 \geq \frac{2}{n-2}$. Arguing as we have before, we obtain
\[
|\omega_j(z, t)| \leq C \Big(\frac{L}{4(-t)}\Big)^{-\bar \lambda_j + 1}+ C L^{20}e^{-\frac{L}{100}} (1+(\bar \lambda_j)^{-\frac{1}{2}})^{-\bar \lambda_j}  +  CL^{20}e^{-\frac{L^2\sqrt{\bar \lambda_j}}{200(-t)}}
\]
for $(z, t) \in [-1000, 1000] \times [-1000, t_n]$. Summing over $j \geq n+1$, we get 
\[
\Big|\sum_{j = n+1}^{\infty} \omega_j(z,t) Y_j \Big| \leq C \sum_{j =n+1}^{\infty} \lambda_j^n |\omega_j(z, t)| \leq C L^{-\frac{2}{n-2}}. 
\]

When $j =1, \dots n$, then $\bar \lambda_j = \frac{n-1}{2(n-2)}$. In this case, we have estimates
\[
|\hat \omega_j(z, t)| \leq C (-t)^{\frac{n-3}{2(n-2)}}
\]
for $(z, t) \in [-\frac{1}{2}L, \frac{1}{2}L] \times [-\frac{1}{2}L, - \frac{1}{4}L]$, and 
\[
|\hat \omega_j(z, t)| \leq CL^{20}(-t)^{-\frac{n-1}{2(n-2)} -1}
\]
for $(z, t) \in [-\frac{1}{2}L, \frac{1}{2}L] \times [-\frac{1}{4}L, t_n]$. Note that $\frac{n-3}{2(n-2)} - \frac{1}{2} = -\frac{1}{2(n-2)}$ and $-\frac{n-1}{2(n-2)} - 1 = -\frac{3}{2} - \frac{1}{2(n-2)}$. Thus, by standard interior estimates for linear parabolic equations, we obtain 
\[
\Big|\frac{\partial}{\partial z}\hat \omega_j(z, t)\Big| \leq C (-t)^{-\frac{1}{2(n-2)}}
\]
for $(z, t) \in [-\frac{1}{4}L, \frac{1}{4}L] \times [-\frac{1}{2}L, - \frac{1}{4}L]$, and 
\[
\Big|\frac{\partial}{\partial z} \hat \omega_j(z, t)\Big| \leq CL^{20}(-t)^{-\frac{3}{2}}
\]
for $(z, t) \in [-\frac{1}{4}L, \frac{1}{4}L] \times [-\frac{1}{4}L, t_n]$. Since $\frac{\partial}{\partial z}\hat \omega_j(z, t)$ satisfies the one-dimensional heat equation, we can use the representation formula to obtain the estimate
\[
\Big|\frac{\partial}{\partial z}\hat \omega_j(z, t)\Big| \leq C L^{-\frac{1}{2(n-2)}}
\]
for $(z, t) \in [-2000, 2000] \times [-2000, t_n]$ and for $j =1, \dots, n$. Using interior estimates again, we get 
\[
\Big|\frac{\partial}{\partial t} \hat \omega_j(z, t)\Big| = \Big|\frac{\partial^2}{\partial z^2}\hat \omega_j(z, t)\Big| \leq C L^{-\frac{1}{2(n-2)}}
\]
for $(z, t) \in [-1000, 1000] \times [-1000, t_n]$ and for $j =1, \dots, n$. In view of our estimates for $\frac{\partial }{\partial t} \hat \omega_j(z, t)$ and $\frac{\partial }{\partial z} \hat \omega_j(z, t)$, it follows that we can find constants $q_j$, for $j =1 ,\dots, n$, such that 
\[
\Big|\hat \omega_j(z, t) - q_j \Big| \leq C L^{-\frac{1}{2(n-2)}}
\]
for $(z, t) \in [-1000,1000] \times [-1000, t_n]$. Hence
\[
\Big|\omega_j(z, t) - (-t)^{\frac{n-1}{2(n-2)}}q_j \Big| \leq C L^{-\frac{1}{2(n-2)}}
\]
for $(z, t) \in [-1000, 1000] \times [-1000, t_n]$ and $j =1, \dots, n$. 

To summarize this step, we conclude that
\[
\Big|\omega(z, t) - \bar \omega(z,t) - (-t)^{\frac{n-1}{2(n-2)}}\sum_{j=1}^n q_j Y_j \Big| = \Big|\sum_{j =1}^{\infty} \omega_j(z,t) Y_j - (-t)^{\frac{n-1}{2(n-2)}}\sum_{j=1}^n q_j Y_j \Big| \leq C L^{-\frac{1}{2(n-2)}} 
\]
in the region $\{|z| \leq 1000, -1000 \leq t \leq t_n \}$. \\

\textit{Conclusion:} Define $\psi := \sum_{j=1}^n q_j Y_j$. Putting the estimates in all the previous steps together, we have shown that 
\begin{align*}
\big| h(t)& - \bar \omega (z,t) g_{S^{n-1}} - \bar \beta(z, t) dz \otimes dz - (-t)^{\frac{n-1}{2(n-2)}} \psi g_{S^{n-1}}\big|_{\bar g(t)}\\
& \leq C\big| \omega(z,t) - \bar \omega(z,t) - (-t)^{\frac{n-1}{2(n-2)}}\psi\big| + C|\chi(z, t)|_{g_{S^{n-1}}}+ C|\sigma(z,t)|_{g_{S^{n-1}}} + C|\beta(z,t)- \bar \beta(z,t)\big|  \\
& \leq C L^{-\frac{1}{2(n-2)}}
\end{align*}  
in the region $\{|z| \leq 1000, -1000 \leq t \leq t_n \}$. This completes the proof of the theorem. 
\end{proof}

\begin{remark}
Note that since 
\[
|\bar \omega(t) g_{S^{n-1}}|_{\bar g(t)} \sim 1, \quad \text{and} \quad |\bar \beta(t)\, dz \otimes dz |_{\bar g(t)} \sim 1, 
\]
these solutions of the parabolic Lichnerowicz equation correspond to rotationally symmetric sums of non-decaying modes of the operator on the cylinder. Meanwhile, 
\[
|(-t)^{\frac{n-1}{2(n-2)}}\psi g_{S^{n-1}}|_{\bar g(t)} \sim (-t)^{-\frac{n-3}{2(n-2)}}
\]
is actually growing as $t$ increases when $n > 3$. 

The presence of these non-decaying solutions of the parabolic Lichnerowicz equation can be understood as follows. Recall $\bar g(t) = (-2(n-2) t)g_{S^{n-1}} + dz \otimes dz$ denotes the standard solution to the Ricci flow on the cylinder for $t \in (-\infty, 0)$. Consider a large smooth bounded domain $\Omega \subset S^{n-1} \times \mathbb{R}$ and a large real number $-T$. Suppose $g(t)$ is a smooth solution of the Ricci flow defined on $\Omega$ for $t \in [T, t_n]$ which is very close to the standard metric $\bar g(t)$ on $\Omega$ for $t \in [T, t_n]$. Solve the harmonic map heat flow $\frac{\partial}{\partial t} \Phi(t) = \Delta_{g(t), \bar g(t)} \Phi(t)$ on $\Omega \times (T, t_n)$ with initial condition $\Phi_{T} = \mathrm{Id}$ and boundary condition $\Phi_{t} |_{\partial \Omega} = \mathrm{Id}$ for $t \in [T, t_n]$. Note that $\Phi(t)$ must remain very close to the identity. Then, define $\tilde g(t)$ by the identity $g(t) = \Phi_t^\ast \tilde g(t)$. By Proposition \ref{ricci-to-ricci-deturck}, $\tilde g(t)$ is a solution of the Ricci-DeTurck flow with respect to the evolving background metrics $\bar g(t)$. By Proposition \ref{linearization-of-evolving-ricci-deturck}, the difference $\tilde g(t)- \bar g(t)$ should be very close to solving the parabolic Lichnerowicz equation on the cylinder. 

Now there are two ways the metrics $g(t)$ may differ from the metrics $\bar g(t)$. Firstly, they may differ geometrically. This is the case if we pullback the Bryant soliton solution or the ancient oval solution in their tubular regions to the cylinder. Of course the Bryant soliton and ancient oval are rotationally symmetric, and in this case, the difference between $\tilde g(t)$ and $\bar g(t)$ gives rise to the rotationally symmetric $\bar{\omega}$-solution of the parabolic Lichnerowicz equation. Alternatively, we could let $g(t)$ be a small translation of $\bar g(t)$ in time. This would also give rise to the $\bar \omega$-solution. The second possibility is for $g(t)$ to differ from $\bar g(t)$ by a diffeomorphism. In this case, you may have $g(t) = \varphi^\ast \bar g(t)$ for some $\varphi$ close to the identity and so $\Phi_t$ solves $\frac{\partial}{\partial t}\Phi_t = \Delta_{\varphi^\ast \bar g(t), \bar g(t)} \Phi_t$ or equivalently $\frac{\partial}{\partial t}(\Phi_t \circ \varphi^{-1}) = \Delta_{\bar g(t), \bar g(t)} (\Phi_t\circ \varphi^{-1})$. It turns out the identity map is not a stable critical point of the harmonic map heat flow on the cylinder in higher dimensions. By the second part of  Proposition \ref{linearization-of-evolving-ricci-deturck}, the linearization of the map Laplacian on the cylinder is 
\[
X \mapsto \frac{1}{(-2(n-2)t)} (\Delta_{S^{n-1}} + (n-2)) X\big|_{S^{n-1}}+ \frac{\partial^2}{\partial z^2} X.
\]
The vector field $\frac{\partial}{\partial z}$ is clearly in the kernel of this operator. This gives rise to a perturbation of the family of diffeomorphisms which corresponds to a ``stretching" in the $\mathbb{R}$-factor of the cylinder. In higher dimensions $n > 3$, the operator $\Delta_{S^{n-1}} +(n-2)$ has one positive eigenvalue which is $n-3$. It turns out the space of eigenvectors for this eigenvalue is spanned by the gradients of the first spherical harmonics. These considerations point to why the $\bar \beta$-solution is steady and $\psi$-solution grows. 
\end{remark}


\section{The Neck Improvement Theorem}

Recall that $t_n = -\frac{1}{2(n-2)}$ and $\bar g(t) = (-2(n-2)t)g_{S^{n-1}} + dz \otimes dz$. 

\begin{definition}\label{epsilon_neck}
Let $(M, g(t))$ be an $n$-dimensional solution to the Ricci flow and let $(\bar x, \bar t)$ be a spacetime point with $R(\bar x, \bar t) = (n-2)(n-1)r^{-2}$. Define
\[
\hat g(t) = r^{-2} g(r^2 (t-t_n) + \bar t). 
\]
Given $\varepsilon > 0$, we say that $(\bar x, \bar t)$ lies at the center of an evolving $\varepsilon$-neck if the rescaled solution $\hat g(t)$ in the parabolic neighborhood $B_{\hat g(t_n)}(\bar x, \varepsilon^{-1}) \times [t_n -\varepsilon^{-1}  , t_n]$ is $\varepsilon$-close in $C^{[\varepsilon^{-1}]}$ to the standard metric $\bar g(t)$ of a family of shrinking cylinders in a corresponding parabolic neighborhood. 
\end{definition}

By modeling neck regions on $\bar g(t)$ at time $t = t_n$, we have the desirable property that metric balls are close to subcylinders. In particular, for the model metric $\bar g(t_n)$ and a point $x \in S^{n-1} \times \{0\}$, we have 
\[
S^{n-1} \times [-L+2, L-2] \subset B_{\bar g(t_n)}(x, L) \subset S^{n-1} \times [-L, L].
\] 

\begin{definition}[Neck Symmetry]\label{neck_symmetry}
Let $(M, g(t))$ be an $n$-dimensional solution to the Ricci flow and let $(\bar x, \bar t)$ be a spacetime point with $R(\bar x, \bar t) = (n-1)(n-2)r^{-2}$. Assume that $(\bar x, \bar t)$ lies at the center of an evolving $\varepsilon_0$-neck for some  small positive real number $\varepsilon_0$. We say $(\bar x, \bar t)$ is $\varepsilon$-symmetric if there exists a smooth, time-independent family of vector fields $\mathcal U = \{U^{(a)} : 1 \leq a \leq {n \choose 2}\}$ defined on the closed ball $\bar B_{g(\bar t)}(\bar x, 100r)$ with the following properties: 
\begin{enumerate}
\item[$\bullet$] In $B_{g(\bar t)}(\bar x, 100 r) \times [\bar t - 100 r^2, \bar t]$, we have the estimate
\[
 \sum_{l = 0}^2 \sum_{a = 1}^{{n \choose 2}} r^{2l} \big|D^l (\mathcal L_{U^{(a)}}(g(t)))|^2 \leq \varepsilon^2.
\]
\item[$\bullet$] If $t \in [\bar t - 100 r^2, \bar t]$ and $\nu$ denotes the unit normal vector to Hamilton's CMC foliation of the $\varepsilon_0$-neck at time $t$, then in $B_{g(\bar t)}(\bar x, 100 r)$, we have the estimate
\[
\sum_{a = 1}^{{n \choose 2}} r^{-2} |\langle U^{(a)}, \nu \rangle|^2 \leq \varepsilon^2. 
\]
\item[$\bullet$] If $t \in [\bar t -100 r^2, \bar t]$ and $\Sigma \subset B_{g(\bar t)}(\bar x, 100 r)$ is a leaf of Hamilton's CMC foliation of the $\varepsilon_0$-neck at time $t$, then 
\[
\sum_{a, b = 1}^{{n \choose 2}} \bigg| \delta_{ab} - \mathrm{area}_{g(t)}(\Sigma)^{-\frac{n+1}{n-1}} \int_\Sigma \langle U^{(a)}, U^{(b)} \rangle_{g(t)} \, d\mu_{g(t)} \bigg|^2 \leq \varepsilon^2. 
\]
\end{enumerate}
\end{definition}

We recall the following lemma from \cite{Bre20}. The proof in higher dimensions using Definition \ref{neck_symmetry} is the same. 

\begin{lemma}
If $(\bar x, \bar t)$ is a spacetime point that is $\varepsilon$-symmetric and $(\tilde x, \tilde t)$ is a point sufficiently close to $(\bar x, \bar t)$, then $(\tilde x, \tilde t)$ is $2\varepsilon$-symmetric. 
\end{lemma}

In the following lemma, we use the gluing result Corollary \ref{gluing_vectors} to construct a family of vector fields defined on a large metric ball that satisfy slightly weakened versions of the estimates of Definition \ref{neck_symmetry}. These estimates hold on an interval of time proportional to the curvature scale. 

\begin{lemma}\label{big_existence}
If $L$ is sufficiently large (depending upon $n$) and $\varepsilon_0$ is sufficiently small (depending upon $n$ and $L$), then the following holds. Let $(M, g(t))$ be a solution to the Ricci flow in dimension $n$ and let $(x_0, t_n)$ be a spacetime point that lies at the center of an evolving $\varepsilon_0$-neck and satisfies $R(x_0, t_n) = (n-1)(n-2)$. Assume that every spacetime point in the parabolic neighborhood $B_{g(t_n)}(x_0, L) \times [t_n -L , t_n)$ is $\varepsilon$-symmetric for some positive real number $\varepsilon \leq \varepsilon_0$. Then, for any $\bar t \in [-\frac{L}{5n}, t_n]$, we can find a time-independent family of vector fields $\mathcal U = \{ U^{(a)} : 1 \leq a \leq {n \choose 2}\}$ defined on $B_{g(t_n)}(x_0, \frac{127L}{128})$ with the following properties: 
\begin{enumerate}
\item[$\bullet$] In $B_{g(t_n)}(x_0, \frac{127L}{128}) \times [5n \bar t, \bar t]$, we have the estimate
\[
\sum_{a = 1}^{{n \choose 2}} \big|\mathcal L_{U^{(a)}}(g(t))|^2 + (-t) \big|D (\mathcal L_{U^{(a)}}(g(t)))|^2 \leq C\varepsilon^2.
\]
\item[$\bullet$] If $t \in [5n \bar t, \bar t]$ and $\nu$ is the unit normal vector to Hamilton's CMC foliation of the $\varepsilon_0$-neck at time $t$, then in $B_{g(t_n)}(x_0, \frac{127L}{128})$ we have the estimate
\[
\sum_{a = 1}^{{n \choose 2}} (-t)^{-1} |\langle U^{(a)}, \nu \rangle|^2 \leq C\varepsilon^2. 
\]
\item[$\bullet$] If $t \in [5n \bar t, \bar t]$ and $\Sigma \subset B_{g(t_n)}(x_0,\frac{127L}{128})$ is a leaf of Hamilton's CMC foliation of the $\varepsilon_0$-neck at time $t$, then 
\[
\sum_{a, b = 1}^{{n \choose 2}} \bigg| \delta_{ab} - \mathrm{area}_{g(t)}(\Sigma)^{-\frac{n+1}{n-1}} \int_\Sigma \langle U^{(a)}, U^{(b)} \rangle_{g(t)} \, d\mu_{g(t)} \bigg|^2 \leq C \varepsilon^2. 
\]
\end{enumerate}
Moreover, the family of vector fields $\mathcal U$ is $C(L)\varepsilon_0$-close to a standard family of rotational vector fields on the cylinder in the $C^{2}$-norm.
\end{lemma}

\begin{proof}
The proof is analogous to the proof of Lemma 8.4 in \cite{Bre20}. 
\end{proof}

In the next lemma, we extend the uniqueness result Proposition \ref{uniqueness_of_symmetry_fields} to the families of rotational vector fields constructed in the previous lemma. The result shows such families of vector fields are unique up to a rotation in $O({n \choose 2})$. It follows that if we allow our family of vector fields to be time-dependent, then we can construct a family of vector fields satisfying the estimates in Lemma \ref{big_existence} for a large interval of time.

\begin{lemma}\label{big_uniqueness}
If $L$ is sufficiently large (depending upon $n$) and $\varepsilon_0$ is sufficiently small (depending upon $n$ and $L$), then the following holds. Let $(M, g(t))$ be a solution to the Ricci flow in dimension $n$ and let $(x_0, t_n)$ be a spacetime point that lies at the center of an evolving $\varepsilon_0$-neck and satisfies $R(x_0, t_n) = (n-1)(n-2)$. Consider a time $\bar t \in [-L, t_n]$ a positive real number $\varepsilon \leq \varepsilon_0$. Suppose that $\mathcal U = \{ U^{(a)} : 1 \leq a \leq {n \choose 2}\}$ is a time-independent family of vector fields defined on $B_{g(t_n)}(x_0, \frac{127L}{128})$ with the following properties: 
\begin{enumerate}
\item[$\bullet$] In $B_{g(t_n)}(x_0, \frac{127L}{128})$, we have the estimate
\[
\sum_{a = 1}^{{n \choose 2}} \big|\mathcal L_{U^{(a)}}(g(\bar t))|^2 + (-\bar t) \big|D (\mathcal L_{U^{(a)}}(g(\bar t)))|^2 \leq \varepsilon^2.
\]
\item[$\bullet$] If $\nu$ is the unit normal vector to Hamilton's CMC foliation of the $\varepsilon_0$-neck at time $\bar t$, then in $B_{g(t_n)}(x_0, \frac{127L}{128})$ we have the estimate
\[
\sum_{a = 1}^{{n \choose 2}} (-\bar t)^{-1} |\langle U^{(a)}, \nu \rangle|^2 \leq \varepsilon^2. 
\]
\item[$\bullet$] If $\Sigma \subset B_{g(t_n)}(x_0, \frac{127L}{128})$ is a leaf of Hamilton's CMC foliation of the $\varepsilon_0$-neck at time $\bar t$, then 
\[
\sum_{a, b = 1}^{{n \choose 2}} \bigg| \delta_{ab} - \mathrm{area}_{g(\bar t)}(\Sigma)^{-\frac{n+1}{n-1}} \int_\Sigma \langle U^{(a)}, U^{(b)} \rangle_{g(\bar t)} \, d\mu_{g(\bar t)} \bigg|^2 \leq \varepsilon^2. 
\]
\end{enumerate}
Moreover, suppose that $\tilde {\mathcal U} = \{ \tilde U^{(a)} : 1 \leq a \leq {n \choose 2}\}$ is a second time-independent family of vector fields defined on $B_{g(t_n)}(x_0,\frac{127L}{128})$ satisfying the same three properties above (with $U^{(a)}$ replaced by $\tilde U^{(a)}$). Then there exists an ${n \choose 2} \times {n \choose 2}$ matrix $\omega \in O({n \choose 2})$ such that in $B_{g(t_n)}(x_0, \frac{31L}{32})$
\[
(-\bar t)^{-1}\sum_{a = 1}^{{n \choose 2}} \Big| \sum_{b =1}^{{n \choose 2}} \omega_{ab} U^{(b)} - \tilde U^{(a)} \Big|_{g(\bar{t})}^2 \leq CL^2 \varepsilon^2.
\]
\end{lemma}
\begin{proof}
The proof is analogous to the proof of Lemma 8.5 in \cite{Bre20}. 
\end{proof}

Finally, in two steps of the proof of the Neck Improvement Theorem, we will want to take advantage of the Lie algebra structure of a standard family of rotational vector fields on a neck. For that, we will use the following lemma. 

\begin{lemma}\label{lie_algebra}
Let $\{\sigma^{(a)} : 1 \leq a \leq {n \choose 2}\}$ be any orthonormal basis of the Lie algebra $so(n)$. We can find constants $\{ k_{abc} : 1 \leq a,b,c \leq {n \choose 2}\}$ with the property that $|k_{abc}| \leq C$ for every $a, b,c$ and 
\[
\sigma^{(a)} = \sum_{b, c =1}^{{n \choose 2}} k_{abc} \,[\sigma^{(b)}, \sigma^{(c)}]. 
\]
\end{lemma}
\begin{proof}
Consider the basis of $so(n)$ given by antisymmetric matrices $E_{ij}$, for $1 \leq i < j \leq n$, which have $1$ in the $(i, j)$-entry and have $-1$ in the $(j, i)$-entry. It is easy to see that $[E_{ij}, E_{jk}] = E_{ik}$ and $[E_{ij}, E_{ik}] = -E_{jk}$. Let $\{ \tilde \sigma^{(a)} : 1 \leq a \leq {n \choose 2}\}$ be any ordering of this basis. Then clearly there exist constants $\tilde k_{abc}$ such that $|\tilde k_{abc}| \leq 1$ and $\tilde \sigma^{(a)} = \sum_{b, c = 1}^{{n \choose 2}} \tilde k_{abc} [\tilde \sigma^{(b)}, \tilde \sigma^{(c)}]$. For any other basis, we can find an ${n\choose 2} \times {n \choose 2} $ matrix $\omega \in O({n \choose 2})$ such that $\tilde \sigma^{(a)} = \sum_{b =1}^{{n \choose 2}} \omega_{ab} \sigma^{(b)}$. Note that since $\sum_{b =1}^{{n \choose 2}} \omega_{ab} \omega_{cb} = \delta_{ac}$,  we have $|\omega_{ab}| \leq C$. Moreover, 
\[
\sigma^{(a)} = \sum_{d =1}^{{n \choose 2}} \omega_{da} \tilde \sigma^{(d)} = \sum_{d, e, f =1}^{{n \choose 2}} \omega_{da} \tilde k_{def}  [\tilde \sigma^{(e)}, \tilde \sigma^{(f)}] = \sum_{b,c =1}^{{n \choose 2}}\Big(\sum_{d, e, f =1}^{{n \choose 2}} \omega_{da} \omega_{eb} \omega_{fc} \tilde k_{def} \Big) [\sigma^{(b)}, \sigma^{(c)}].
\]
Thus we can take 
\[
k_{abc} := \sum_{d, e, f =1}^{{n \choose 2}} \omega_{da} \omega_{eb} \omega_{fc} \tilde k_{def}.
\]
\end{proof}

The following is the main result of this section. 

\begin{theorem}[Neck Improvement Theorem]\label{neck_improvement}
There exists a large constant $L$ (depending only upon $n$) and a small constant $\varepsilon_1$ (depending only upon $L$ and $n$) with the following property. Let $(M, g(t))$ be a solution of the Ricci flow in dimension $n$, and let $(x_0, t_0)$ be a spacetime point that lies at the center of an evolving $\varepsilon_1$-neck and satisfies $R(x_0, t_0) = (n-1)(n-2)r^{-2}$. Moreover, suppose that every point in the parabolic neighborhood $B_{g(t_0)}(x_0, L r) \times [t_0 - L r^2, t_0)$ is $\varepsilon$-symmetric, where $\varepsilon \leq \varepsilon_1$. Then $(x_0, t_0)$ is $\frac{\varepsilon}{2}$-symmetric. 
\end{theorem}

\begin{proof}
After a translation in time and a parabolic rescaling, we may assume that $t_0 = t_n$ and $R(x_0, t_n) = (n-1)(n-2)$. Throughout the proof we will assume that $L$ is sufficiently large and $\varepsilon_1$ is sufficiently small depending upon $L$.  Recall that we have a height function $z : B_{g(t_n)}(x_0, \varepsilon_1^{-1}) \to \mathbb{R}$, which we assume is normalized so that $z(x_0) = 0$. We will let $\bar g(t) = (-2(n-2)t)g_{S^{n-1}} + dz \otimes dz$ denote evolving metric on the exact cylinder, as well as its image under pullback to the parabolic neighborhood $B_{g(t_n)}(x_0, \varepsilon_1^{-1}) \times [t_n-\varepsilon_1^{-1} , t_n]$. We have the estimate $\sum_{l = 0}^{1000} |D^l (\bar g(t) - g(t))| \leq C(L)\varepsilon_1$ on the neck region. Here and throughout the proof, the covariant derivative $D$ and norm $|\cdot |$ are taken with respect to $g(t)$.  However, note that $(1- C(L)\varepsilon_1)|D^lT| \leq |\bar D^l T|_{\bar g(t)} \leq (1+ C(L)\varepsilon_1)|D^lT|$ for any tensor $T$ and $0 \leq l \leq 500 $.  \\

\textit{Step 1:} Using Lemma \ref{big_existence} and Lemma \ref{big_uniqueness} (in particular, by choosing a suitable time-dependent element $\omega_{ab}(t) \in O({n \choose 2})$), we can construct a time-dependent family of vector fields $\mathcal U = \{U^{(a)} : 1 \leq a \leq {n \choose 2 }\}$ defined on $B_{g(t_n)}(x_0, \frac{15L}{16}) \times [-L, t_n]$ with the following properties: 
\begin{enumerate}
\item[$\bullet$] On $B_{g(t_n)}(x_0, \frac{15L}{16})$ 
\[
\begin{cases}
\frac{\partial}{\partial t} U^{(a)} = 0 & \text{ for }  t \in \Big[-L, -\frac{1}{4}L\Big]; \\
\Big|\frac{\partial}{\partial t} U^{(a)}\Big| \leq CL (-t)^{-\frac{1}{2}} \varepsilon & \text{ for }  t \in \Big[-\frac{1}{4}L, t_n\Big]. 
\end{cases}
\]
\item[$\bullet$] On $B_{g(t_n)}(x_0, \frac{15L}{16}) \times [-L, t_n]$, 
\[
\big|\mathcal L_{U^{(a)}}(g(t))| + (- t)^{\frac{1}{2}} \big|D (\mathcal L_{U^{(a)}}(g(t)))| \leq C\varepsilon.
\]
\end{enumerate}
 Recall that we have the identity 
\[
\Delta U^{(a)} + \mathrm{Ric}(U^{(a)}) = \mathrm{div}(\mathcal L_{U^{(a)}}(g)) - \frac{1}{2} \nabla \,\mathrm{tr}(\mathcal L_{U^{(a)}}(g)). 
\]
Consequently, the vector fields $U^{(a)}$ satisfy
\[
|\Delta U^{(a)} + \mathrm{Ric}(U^{(a)})| \leq C |D(\mathcal L_{U^{(a)}}(g(t)))| \leq C(-t)^{-\frac{1}{2}} \varepsilon
\]
 on $B_{g(t_n)}(x_0, \frac{15L}{16}) \times [-L, t_n]$. 
 
There is a correspondence between orthonormal bases of $so(n)$ and (time-independent) families of vector fields $\mathcal U_{\mathrm{cyl}} := \{ U^{(a)}_{\mathrm{cyl}} : 1 \leq a \leq {n \choose 2}\}$ defined on the cylinder satisfying $\mathcal L_{U^{(a)}_{\mathrm{cyl}}} \bar g(t) = 0$, $dz( U^{(a)}_{\mathrm{cyl}} )= 0$, and
\[
\mathrm{area}_{\bar g(t)}(S^{n-1})^{-\frac{n+1}{n-1}} \int_{S^{n-1}} \langle U_{\mathrm{cyl}}^{(a)}, U_{\mathrm{cyl}}^{(b)} \rangle_{\bar g(t)} = \delta_{ab}.
\] 
Our assumptions imply that the family of vector fields $\mathcal U$ satisfies these three identities up to errors of bounded by $C(L)\varepsilon_1$. It follows that we can find a standard family of rotational vector fields $U^{(a)}_{\mathrm{cyl}}$ on the standard cylinder such that the vector field $U^{(a)}$ is $C(L)\varepsilon_1$-close in the $C^{2}$-norm to the vector field $ U^{(a)}_{\mathrm{cyl}}$. \\

\textit{Step 2:} Let $V^{(a)}$ be a solution of the PDE 
\[
\frac{\partial}{\partial t} V^{(a)} = \Delta V^{(a)} + \mathrm{Ric}(V^{(a)})
\]
in the region $\{|z| \leq \frac{7}{8}L, \, -L \leq t \leq t_n\}$ with Dirichlet boundary conditions $V^{(a)} = U^{(a)}$ on the parabolic boundary $\{|z| \leq \frac{7}{8}L , \, t = -L \} \cup \{|z| = \frac{7}{8}L ,\, -L \leq t \leq t_n\}$. Consider the difference $V^{(a)} - U^{(a)}$. In the region $\{|z| \leq  \frac{7}{8}L, \, -L \leq t \leq -\frac{1}{4}L\}$, we have
\begin{align*}
\big| \frac{\partial}{\partial t} &(V^{(a)} - U^{(a)}) - \Delta (V^{(a)} - U^{(a)})  + \mathrm{Ric}(V^{(a)} - U^{(a)}) \big| \\
& = |\Delta U^{(a)} + \mathrm{Ric}(U^{(a)})| \leq C(-t)^{-\frac{1}{2}} \varepsilon.
\end{align*}
In the region $\{|z| \leq  \frac{7}{8}L, \, -\frac{1}{4}L \leq t \leq t_n\}$, we have 
\begin{align*}
\big| \frac{\partial}{\partial t} &(V^{(a)} - U^{(a)}) - \Delta (V^{(a)} - U^{(a)})  + \mathrm{Ric}(V^{(a)} - U^{(a)}) \big| \\
& \leq \big|\frac{\partial}{\partial t} U^{(a)}\big| +   |\Delta U^{(a)} + \mathrm{Ric}(U^{(a)})| \leq CL(-t)^{-\frac{1}{2}} \varepsilon. 
\end{align*}
Hence by Proposition \ref{norm_pde}, we have 
\[
\frac{\partial}{\partial t} |V^{(a)} - U^{(a)}| \leq \Delta |V^{(a)} - U^{(a)}| + C(-t)^{-\frac{1}{2}} \varepsilon
\]
in the region $\{|z| \leq  \frac{7}{8}L, \, -L \leq t \leq -\frac{1}{4}L\}$ and 
\[
\frac{\partial}{\partial t} |V^{(a)} - U^{(a)}| \leq \Delta |V^{(a)} - U^{(a)}| + CL (-t)^{-\frac{1}{2}} \varepsilon
\]
in the region $\{|z| \leq \frac{7}{8}L, \, -\frac{1}{4}L \leq t \leq t_n\}$. By the maximum principle, 
\[
|V^{(a)} - U^{(a)}| \leq CL^{\frac{1}{2}}\varepsilon
\]
in the region $\{|z| \leq \frac{7}{8}L, \, -L \leq t \leq -\frac{1}{4}L\}$ and 
\[
|V^{(a)} - U^{(a)}| \leq CL^2 \varepsilon
\]
in the region $\{|z| \leq \frac{7}{8}L, \, -\frac{1}{4}L \leq t \leq t_n\}$. By standard interior estimates for linear parabolic equations (see Appendix C),
\[
|D(V^{(a)} - U^{(a)})| \leq C\varepsilon
\]
in the region $\{|z| \leq \frac{3}{4}L, \, -\frac{3}{4}L \leq t \leq -\frac{1}{4}L\}$ and 
\[
|D(V^{(a)} - U^{(a)})| \leq CL^2 \varepsilon  
\]
in the region $\{|z| \leq \frac{3}{4}L, \, -\frac{1}{4}L \leq t \leq t_n\}$. In particular, in the region $\{|z| \leq \frac{3}{4}L, \, -\frac{1}{4}L \leq t \leq t_n\}$ the vector field $V^{(a)}$ is $C(L)\varepsilon_1$-close to the vector field $U^{(a)}_{\mathrm{cyl}}$ in $C^1$-norm. Consequently, in the region $\{|z| \leq 1000 , \, -1000  \leq t \leq t_n\}$, the vector fields $V^{(a)}$ are $C(L)\varepsilon_1$-close to the standard rotation vector fields in $C^{100}$-norm. \\

\textit{Step 3:} We now consider $h^{(a)}(t) := \mathcal L_{V^{(a)}}(g(t))$. Since $V^{(a)}$ satisfies the PDE $\frac{\partial}{\partial t} V^{(a)} = \Delta V^{(a)} + \mathrm{Ric}(V^{(a)})$, by Theorem \ref{pde_for_lie_derivative}, its Lie derivative satisfies 
\[
\frac{\partial}{\partial t} h^{(a)}(t)  = \Delta_{L, g(t)} h^{(a)}(t).
\]
The estimates we have obtained for $V^{(a)}- U^{(a)}$ imply that 
\[
|h^{(a)}| \leq |\mathcal L_{U^{(a)}}(g)| + C|D(V^{(a)} - U^{(a)})| \leq C \varepsilon
\]
in the region $\{|z| \leq \frac{3}{4}L, \, -\frac{3}{4}L \leq t \leq -\frac{1}{4}L\}$ and 
\[
|h^{(a)}| \leq |\mathcal L_{U^{(a)}}(g)| + C|D(V^{(a)} - U^{(a)})| \leq C L^2\varepsilon
\]
in the region $\{|z| \leq \frac{3}{4}L, \, -\frac{1}{4}L \leq t \leq t_n\}$. By standard interior estimates for linear parabolic equations, 
\[
\sum_{l =0}^{100} |D^l h^{(a)}| \leq C(L) \varepsilon
\]
in the region $\{|z| \leq \frac{1}{2}L, \, -\frac{1}{2}L \leq t \leq t_n\}$. \\

\textit{Step 4:} Recall that we have $\sum_{l = 0}^{100} |D^l (\bar g(t) - g(t))| \leq C(L)\varepsilon_1$. Let $\bar h^{(a)}$ be the solution of the equation
\[
\frac{\partial}{\partial t} \bar h^{(a)}(t)  = \Delta_{L, \bar g(t)} \bar h^{(a)}(t).
\]
in the region $\{|z| \leq \frac{1}{2}L, \, -\frac{1}{2}L \leq t \leq t_n\}$ with Dirichlet boundary condition $\bar h^{(a)} = h^{(a)}$ on the parabolic boundary $\{|z| \leq \frac{1}{2}L, \, t = -\frac{1}{2}L \} \cup \{|z| = \frac{1}{2}L, \, -\frac{1}{2}L \leq t \leq t_n\}$. Then 
\[
\frac{\partial}{\partial t} (\bar h^{(a)}(t) - h^{(a)}(t))  - \Delta_{L, \bar g(t)} (\bar h^{(a)}(t) - h^{(a)}(t)) = E^{(a)}(t)
\]
where the error term is defined by $E^{(a)}(t) := \Delta_{L, \bar g(t)} h^{(a)}(t) - \Delta_{L, g(t)} h^{(a)}(t)$. Since $\sum_{l = 0}^{100} |D^l (\bar g(t) - g(t))| \leq C(L)\varepsilon_1$ and $\sum_{l =0}^{100} |D^l h^{(a)}| \leq C(L) \varepsilon$, we have $\sum_{l =0}^{90} |D^l E^{(a)}(t)| \leq C(L)\varepsilon_1 \varepsilon$ in the region $\{|z| \leq \frac{1}{2}L, \, -\frac{1}{2}L \leq t \leq t_n\}$. Hence by the maximum principle, we have 
\[
|\bar h^{(a)} - h^{(a)}| \leq C(L)\varepsilon_1\varepsilon
\]
in the region $\{|z| \leq \frac{1}{2}L, \, -\frac{1}{2}L \leq t \leq t_n\}$. Using the higher order derivative estimates for the error, standard interior estimates for linear parabolic equations imply 
\[
\sum_{l=0}^{80} |D^l(\bar h^{(a)} - h^{(a)})| \leq C(L)\varepsilon_1 \varepsilon
\]
in the region $\{|z| \leq 1000, \, -1000 \leq t \leq t_n\}$. \\

\textit{Step 5:} In the next step, we use Proposition \ref{mode_analysis} to analyze $\bar h^{(a)}$. By our estimates for $|h^{(a)}|$ and our estimate for $|\bar h^{(a)} - h^{(a)}|$ in Step 4, we have 
\[
|\bar h^{(a)}| \leq C\varepsilon + C(L)\varepsilon_1\varepsilon
\]
in the region $\{|z| \leq \frac{1}{2}L, -\frac{1}{2}L \leq t \leq - \frac{1}{4}L\}$ and 
\[
|\bar h^{(a)}| \leq CL^2\varepsilon + C(L)\varepsilon_1\varepsilon
\]
in the region $\{|z| \leq \frac{1}{2}L, -\frac{1}{4}L \leq t \leq t_n\}$. This implies $\bar h^{(a)}$ satisfies the hypotheses of Proposition \ref{mode_analysis}. It follows that for each $a \in \{1, \dots, {n \choose 2}\}$, we can find a function $\psi^{(a)} : S^{n-1} \to \mathbb{R}$ (independent of $z$ and $t$), as well as rotationally invariant functions $\bar \omega^{(a)}(z, t)$ and $\bar \beta^{(a)}(z, t)$ (depending only about the height $z$ and time $t$) with the following properties:
\begin{enumerate}
\item[$\bullet$] $\psi^{(a)}$ is in the span of the first spherical harmonics on $S^{n-1}$.
\item[$\bullet$] $\bar \omega^{(a)}(z, t)$ and $\bar \beta^{(a)}(z, t)$ are solutions of the one-dimensional heat equation in the region $\{|z| \leq \frac{1}{2}L, -\frac{1}{4}L \leq t \leq t_n\}$.
\item[$\bullet$] In the region $\{|z| \leq 1000, \, -1000 \leq t \leq t_n\}$, we have the estimate 
\[
\big| \bar h^{(a)}(t) - \bar \omega^{(a)}(z, t) g_{S^{n-1}} - \bar \beta^{(a)}(z, t) dz \otimes dz - (-t)^{\frac{n-1}{2(n-2)}} \psi^{(a)} \, g_{S^{n-1}} \big| \leq CL^{-\frac{1}{2(n-2)}} \varepsilon + C(L)\varepsilon_1\varepsilon. 
\]
\end{enumerate}
Since $h^{(a)}$ and its derivatives are bounded by $C(L)\varepsilon$, the functions $\bar \omega^{(a)}, \bar \beta^{(a)}$, and $\psi^{(a)}$ and their derivatives are bounded by $C(L)\varepsilon$. Note that $\bar h^{(a)}(t) - \bar \omega^{(a)}(z, t) g_{S^{n-1}} - \bar \beta^{(a)}(z, t) dz \otimes dz - (-t)^{\frac{n-1}{2(n-2)}} \psi^{(a)} \, g_{S^{n-1}}$ solves the parabolic Lichnerowicz equation with respect to $\bar g(t)$ in the region $\{|z| \leq \frac{1}{2}L, -\frac{1}{4}L \leq t \leq t_n\}$. By standard interior estimates for linear parabolic equations, we obtain 
\begin{align*}
\sum_{l = 0}^{80}& \big|D^l\big(\bar h^{(a)}(t) - \bar \omega^{(a)}(z, t) g_{S^{n-1}} - \bar \beta^{(a)}(z, t) dz \otimes dz - (-t)^{\frac{n-1}{2(n-2)}}\psi^{(a)} \, g_{S^{n-1}} \big)\big|\\
& \leq CL^{-\frac{1}{2(n-2)}} \varepsilon + C(L)\varepsilon_1\varepsilon
\end{align*}
 in the region $\{|z| \leq 800, -400  \leq t \leq t_n\}$. Combining this with the last estimate of Step 4, we obtain 
 \begin{align*}
\sum_{l = 0}^{80}& \big|D^l\big(h^{(a)}(t) - \bar \omega^{(a)}(z, t) g_{S^{n-1}} - \bar \beta^{(a)}(z, t) dz \otimes dz - (-t)^{\frac{n-1}{2(n-2)}}\psi^{(a)} \, g_{S^{n-1}} \big)\big|\\
& \leq CL^{-\frac{1}{2(n-2)}} \varepsilon + C(L)\varepsilon_1\varepsilon
\end{align*}
 in the region $\{|z| \leq 800, -400  \leq t \leq t_n\}$. \\
 
 \textit{Step 6:}  In this step and several following, we modify the vector fields $V^{(a)}$ to remove the non-decaying solutions of the parabolic Lichnerowicz from the estimate of the previous step. We will also need to address the time-dependence of the family of vector fields.
 
 We first turn our attention to the term $k^{(a)}(t):= (-t)^{\frac{n-1}{2(n-2)}} \psi^{(a)} \, g_{S^{n-1}}$. Let us define a time-dependent vector field on $S^{n-1}$ by 
 \[
 \xi^{(a)} := -\frac{1}{4(n-2)}(-t)^{-\frac{n-3}{2(n-2)}}\nabla_{S^{n-1}} \psi^{(a)}.
 \]
 As the gradient of a first spherical harmonic, the vector field $\xi^{(a)}$ is a conformal killing vector field on $S^{n-1}$. Hence 
 \[
 \mathcal L_{\xi^{(a)}} (g_{S^{n-1}})  = \frac{2}{n-1} \mathrm{div}_{S^{n-1}}(\xi^{(a)}) g_{S^{n-1}} = \frac{1}{2(n-2)}(-t)^{-\frac{n-3}{2(n-2)}}\psi^{(a)} g_{S^{n-1}}.
 \]
 Consequently, $\mathcal L_{\xi^{(a)}}(\bar g(t)) = 2(n-2)(-t) \mathcal L_{\xi^{(a)}} (g_{S^{n-1}}) =k^{(a)}(t)$.  
 
We now define a family of vector fields $\mathcal W = \{ W^{(a)} : 1 \leq a \leq {n \choose 2}\}$ by $W^{(a)} := V^{(a)} - \xi^{(a)}$. The final estimates in Step 3 and 4 imply $\sum_{l =0}^{80} |D^l \psi^{(a)}| \leq C(L)\varepsilon$ and in particular $\sum_{l =0}^{79} |D^l \xi^{(a)}| \leq C(L)\varepsilon$. Consequently, in the region $\{|z| \leq 1000 , \, -1000 \leq t \leq t_n\}$, the vector fields $W^{(a)}$ are $C(L)\varepsilon_1$-close to the standard rotation vector fields $U^{(a)}_{\mathrm{cyl}}$ on the cylinder in the $C^{75}$-norm. Putting the identity
 \[
 \mathcal L_{W^{(a)}}(g(t)) = h^{(a)}(t) - k^{(a)}(t) - \mathcal L_{\xi^{(a)}} (g(t) - \bar g(t))
 \]
together with the estimates in Step 5, the estimates for $\xi^{(a)}$, and the estimates for $g(t) - \bar g(t)$, we obtain 
\[
  \sum_{l=0}^{60} \big|D^l\big( \mathcal L_{W^{(a)}}(g(t)) - \bar \omega^{(a)}(z, t) g_{S^{n-1}} - \bar \beta^{(a)}(z, t) dz \otimes dz\big) \big|  \leq CL^{-\frac{1}{2(n-2)}} \varepsilon + C(L)\varepsilon_1\varepsilon
\]
 in the region $\{|z| \leq 800, -400 \leq t \leq t_n\}$. 

We next analyze the time derivative of $W^{(a)}$. Recalling that $ \Delta V^{(a)} + \mathrm{Ric}(V^{(a)}) = \mathrm{div} \,h^{(a)} - \frac{1}{2} \nabla( \mathrm{tr} \,h^{(a)})$, we compute 
\[
\frac{\partial}{\partial t} W^{(a)} = \frac{\partial}{\partial t} V^{(a)} + \frac{\partial}{\partial t} \xi^{(a)} = \mathrm{div} \,h^{(a)} - \frac{1}{2} \nabla( \mathrm{tr} \,h^{(a)}) + \frac{n-3}{(-2(n-2)t)} \xi^{(a)}. 
\]
 Using the Bochner formula, we obtain that
\begin{align*}
\mathrm{div}_{\bar g(t)}\big(k^{(a)}(t)\big) - \frac{1}{2} \nabla_{\bar g(t)} \big( \mathrm{tr}_{\bar g(t)} k^{(a)}(t)\big)& = \Delta_{\bar g(t)} \xi^{(a)} + \mathrm{Ric}_{\bar g(t)}(\xi^{(a)})\\
& = \frac{1}{(-2(n-2)t)} \Big(\Delta_{S^{n-1}}\xi^{(a)} + \mathrm{Ric}_{S^{n-1}}(\xi^{(a)})\Big) \\
& = \frac{n-3}{(-2(n-2)t)}\xi^{(a)}. 
\end{align*}
We also compute that
\begin{align*}
\mathrm{div}_{\bar g(t)} \big(\bar \omega^{(a)} g_{S^{n-1}} + \bar \beta^{(a)} dz \otimes dz \big) &= \frac{\partial \bar \beta^{(a)}}{\partial z} \frac{\partial}{\partial z}, \\
- \frac{1}{2} \nabla_{\bar g(t)} \,\mathrm{tr}_{\bar g(t)}\big(\bar \omega^{(a)} g_{S^{n-1}} + \bar \beta^{(a)} dz \otimes dz\big) & = - \bigg(\frac{n-1}{(-4(n-2)t)} \frac{\partial \bar \omega^{(a)}}{\partial z}  - \frac{1}{2} \frac{\partial \bar \beta^{(a)}}{\partial z}\bigg)\frac{\partial}{\partial z}.
\end{align*}
By our estimates for $h$ and $\bar h - h$, replacing $\bar g(t)$ by $g(t)$ in these identities introduces an error bounded by $C(L)\varepsilon_1\varepsilon$.  Therefore, we obtain that 
\begin{align*}
\sum_{l=0}^{60} &\Big| D^l\Big(\mathrm{div}\, h^{(a)} - \frac{1}{2} \nabla( \mathrm{tr}\,h^{(a)}) + \frac{n-3}{(-2(n-2)t)} \xi^{(a)}- \Big(\frac{1}{2} \frac{\partial \bar \beta^{(a)}}{\partial z}- \frac{n-1}{(-4(n-2)t)} \frac{\partial \bar \omega^{(a)}}{\partial z}\Big) \frac{\partial}{\partial z}\Big)\Big|\\
& \leq CL^{-\frac{1}{2(n-2)}} \varepsilon + C(L)\varepsilon_1\varepsilon
\end{align*}
in the region $\{|z| \leq 800, -400  \leq t \leq t_n\}$. Hence
\[
\sum_{l=0}^{60} \Big|D^l\Big( \frac{\partial}{\partial t} W^{(a)} - \Big(\frac{1}{2} \frac{\partial \bar \beta^{(a)}}{\partial z} - \frac{n-1}{(-4(n-2)t)} \frac{\partial \bar \omega^{(a)}}{\partial z}\Big) \frac{\partial}{\partial z}\Big)\Big| \leq CL^{-\frac{1}{2(n-2)}} \varepsilon + C(L)\varepsilon_1\varepsilon
\]
in the region  $\{|z| \leq 800, -400  \leq t \leq t_n\}$. \\

\textit{Step 7:} In the next step, we use that the rotationally symmetric solutions $\bar \omega^{(a)}(z, t) g_{S^{n-1}}$ and $\bar \beta^{(a)}(z, t) dz \otimes dz$ vanish identically under Lie derivatives by $U^{(a)}_{\mathrm{cyl}}$. Recall that $\mathcal{U}_{\mathrm{cyl}}= \{ U_{\mathrm{cyl}}^{(a)} : 1 \leq a \leq {n \choose 2}\}$ corresponds to some orthonormal basis of $so(n)$. Thus by Lemma \ref{lie_algebra}, we can find constants $k_{abc}$ satisfying $|k_{abc}| \leq C$ and such that $U_{\mathrm{cyl}}^{(a)} = \sum_{b, c =1}^{{n \choose 2}} k_{abc} [U^{(b)}_{\mathrm{cyl}}, U^{(c)}_{\mathrm{cyl}}]$. We define a family of vector fields $\mathcal X := \{X^{(a)} : 1 \leq a \leq {n \choose 2}\}$ by 
\[
X^{(a)} = \sum_{b, c =1}^{{n \choose 2}} k_{abc} [W^{(b)}, W^{(c)}].
\]
By construction, in the region $\{|z| \leq 1000, -1000  \leq t \leq t_n\}$, the vector fields $X^{(a)}$ agree with the standard rotational vector fields $U^{(a)}_{\mathrm{cyl}}$ up to errors of order $C(L)\varepsilon_1$. 

Now the vector fields $X^{(a)}$ satisfy 
\[
\mathcal L_{X^{(a)}}(g(t)) = \sum_{b,c =1}^{{n \choose 2}} k_{abc} \Big( \mathcal L_{W^{(b)}} \big(\mathcal L_{W^{(c)}}(g(t))\big) -  \mathcal L_{W^{(c)}} \big(\mathcal L_{W^{(b)}}(g(t)) \big)\Big).
\]
Moreover, 
\begin{align*}
\mathcal L_{W^{(b)}} &\big(\mathcal L_{W^{(c)}}(g(t))\big) -  \mathcal L_{W^{(c)}} \big(\mathcal L_{W^{(b)}}(g(t)) \big)\\
&  = \mathcal L_{W^{(b)}} \big(\mathcal L_{W^{(c)}}(g(t)) - \bar \omega^{(c)}(z,t)g_{S^{n-1}} - \bar \beta^{(c)}(z,t)dz \otimes dz\big) \\
& \quad  -  \mathcal L_{W^{(c)}} \big(\mathcal L_{W^{(b)}}(g(t)) - \bar \omega^{(b)}(z,t)g_{S^{n-1}} - \bar \beta^{(b)} (z,t)dz \otimes dz\big)\\
& \quad +\mathcal L_{W^{(b)}} \big(\bar \omega^{(c)}(z,t) g_{S^{n-1}} + \bar \beta^{(c)}(z,t) dz \otimes dz \big)\\
& \quad -\mathcal L_{W^{(c)}}  \big(\bar \omega^{(b)}(z,t) g_{S^{n-1}} + \bar \beta^{(b)}(z,t) dz \otimes dz\big)
\end{align*}
The vector fields $W^{(a)}$ are $C(L)\varepsilon_1$-close to the vector fields $U^{(a)}_{\mathrm{cyl}}$ while the functions $\bar \omega^{( b)}, \bar \beta^{(b)}$ (and their derivatives) are bounded by $C(L)\varepsilon$. In particular, for the third and fourth terms above, we have the estimate
\[
\sum_{l =0}^{40} \Big|D^l \Big(\mathcal L_{W^{(b)}} \big( \bar \omega^{(c)}(z, t) g_{S^{n-1}} + \bar \beta^{(c)}(z, t) dz \otimes dz\big)\Big)\Big| \leq C(L)\varepsilon_1\varepsilon 
\]
in the region  $\{|z| \leq 800, -400  \leq t \leq t_n\}$. We estimated the remaining two terms in the previous step. Since $|k_{abc}| \leq C$, we conclude 
\[
\sum_{l=0}^{40} \big|D^l \big( \mathcal L_{X^{(a)}}(g(t))\big)\big| \leq CL^{-\frac{1}{2(n-2)}} \varepsilon + C(L)\varepsilon_1\varepsilon
\] 
in the region $\{|z| \leq 800, -400  \leq t \leq t_n\}$.

We next consider the time derivative of the vector fields. We have
\[
\frac{\partial}{\partial t} X^{(a)} =  \sum_{b, c =1}^{{n \choose 2}} k_{abc} \Big( \Big[\frac{\partial}{\partial t} W^{(b)}, W^{(c)} \Big] + \Big[ W^{(b)}, \frac{\partial}{\partial t} W^{(c)}\Big]\Big). 
\]
Moreover, 
\begin{align*}
\Big[\frac{\partial}{\partial t} W^{(b)}&, W^{(c)} \Big] + \Big[ W^{(b)}, \frac{\partial}{\partial t} W^{(c)}\Big]\\
& = \Big[ \frac{\partial}{\partial t} W^{(b)} - \Big(\frac{1}{2} \frac{\partial \bar \beta^{(b)}}{\partial z} - \frac{n-1}{(-4(n-2)t)} \frac{\partial \bar \omega^{(b)}}{\partial z}\Big) \frac{\partial}{\partial z} \; , \;  W^{(c)} \Big]  \\
& \quad + \Big[ W^{(b)} \;, \;  \frac{\partial}{\partial t} W^{(c)}-\Big(\frac{1}{2} \frac{\partial \bar \beta^{(c)}}{\partial z} - \frac{n-1}{(-4(n-2)t)} \frac{\partial \bar \omega^{(c)}}{\partial z}\Big) \frac{\partial}{\partial z} \Big] \\
& \quad + \Big[ \Big(\frac{1}{2} \frac{\partial \bar \beta^{(b)}}{\partial z}- \frac{n-1}{(-4(n-2)t)} \frac{\partial \bar \omega^{(b)}}{\partial z}\Big) \frac{\partial}{\partial z}\; , \; W^{(c)} \Big] \\
& \quad +\Big[W^{(b)} \; , \; \Big(\frac{1}{2} \frac{\partial \bar \beta^{(c)}}{\partial z} - \frac{n-1}{(-4(n-2)t)} \frac{\partial \bar \omega^{(c)}}{\partial z}\Big) \frac{\partial}{\partial z}\Big].
\end{align*}
The third and fourth therms vanish identically if $W^{(a)}$ is replaced by $U^{(a)}_{\mathrm{cyl}}$. Therefore
\[
\sum_{l = 0}^{40} \Big| D^l \Big[W^{(c)} , \; \Big(\frac{1}{2} \frac{\partial \bar \beta^{(b)}}{\partial z} - \frac{n-1}{(-4(n-2)t)} \frac{\partial \bar \omega^{(b)}}{\partial z}\Big) \frac{\partial}{\partial z} \Big] \Big| \leq C(L)\varepsilon_1\varepsilon
\]
in the region $\{|z| \leq 800, -400  \leq t \leq t_n\}$. We estimated the remaining two terms in our expression for $\frac{\partial}{\partial t} X^{(a)}$ in the previous step. Since again $|k_{abc}| \leq C$, we conclude 
\[
\sum_{l=0}^{40} \big|D^l \big( \frac{\partial}{\partial t} X^{(a)} \big)\big| \leq CL^{-\frac{1}{2(n-2)}} \varepsilon + C(L)\varepsilon_1\varepsilon
\]
in the region $\{|z| \leq 800, -400  \leq t \leq t_n\}$. \\

\textit{Step 8:} We define a family of vector fields $\mathcal Y:= \{Y^{(a)} : 1 \leq a \leq {n \choose 2}\}$ by setting  $Y^{(a)} := X^{(a)}$ at time $t_n$. Since the standard rotational vector fields on the cylinder are time-independent, the vector fields $Y^{(a)}$ agree with the standard rotational vector fields $U^{(a)}_{\mathrm{cyl}}$ on the cylinder in the region $\{|z| \leq 1000, -1000 \leq t \leq t_n\}$ up to errors of order $C(L)\varepsilon_1$. 

The time derivative estimates in Step 7 imply that 
\[
\sum_{l=0}^{40} \big|D^l \big( Y^{(a)} - X^{(a)} \big)\big| \leq CL^{-\frac{1}{2(n-2)}} \varepsilon + C(L)\varepsilon_1\varepsilon
\]
in the region $\{|z| \leq 800, -400  \leq t \leq t_n\}$. Consequently, as we have seen before, this implies 
\[
\sum_{l=0}^{30} \big|D^l\big(\mathcal L_{Y^{(a)}}(g(t)) \big)\big| \leq CL^{-\frac{1}{2(n-2)}} \varepsilon + C(L)\varepsilon_1\varepsilon
\]
in the region $\{|z| \leq 800, -400  \leq t \leq t_n\}$. \\

\textit{Step 9:}  In Steps 1 through 8, we have constructed a family of time-independent vector fields $Y^{(a)}$ that satisfy the first criterion of Definition \ref{neck_symmetry} needed for $(x_0, t_n)$ to be $\big(CL^{-\frac{1}{2(n-2)}} \varepsilon + C(L)\varepsilon_1\varepsilon\big)$-symmetric. We now turn out attention to the remaining criteria in Definition \ref{neck_symmetry}.

Fix a time $t \in [-200, t_n]$ and let us write $g = g(t)$. Let $\Sigma_s$ denote the leaves of Hamilton's CMC foliation of the evolving $\varepsilon_1$-neck centered on $(x_0, t_n)$. This foliation depends on the time $t$, but we suppress this dependence from our notation. Let $\nu$ denote the unit normal to the foliation $\Sigma_s$ and, for each $s$, let $v : \Sigma_s \to \mathbb{R}$ denote the lapse function associate to the foliation. (If our foliation is parametrized by a normal embedding $N : S^{n-1} \times [s_1, s_2] \to \Omega$ with $\Sigma_s = N(S^{n-1} \times \{s\})$,  then $v = \big|\frac{\partial N}{\partial s}\big|$ and $v \nu = \frac{\partial N}{\partial s}$. If our foliation is expressed as the level sets of a submersion $f : \Omega \to [s_1, s_2]$ with $\Sigma_s = f^{-1}(s)$, then $v = |\nabla f|^{-1}$ and $v^{-1} \nu = \nabla f$). We can assume the foliation is parametrized so that $x_0 \in \Sigma_0$ and $\int_{\Sigma_s} v = 1$ for all $s$. The latter identity and comparison with an exact cylinder implies that
\[
\sum_{l =1}^{40} \big| D^l \big( v - \mathrm{area}_{\bar g(t)}(\Sigma_s)^{-1} \big)\big| \leq C(L)\varepsilon_1. 
\]
Since $\Sigma_s$ is a CMC surface for each $s$, the function $v$ satisfies the Jacobi equation 
\[
\Delta_{\Sigma_s} v + (|A|^2 + \mathrm{Ric}(\nu, \nu))v = \mathrm{constant} 
\]
on $\Sigma_s$, where $A$ denotes the second fundamental form of $\Sigma_s$ in $(M, g(t))$. As in three dimensions, the Jacobi operator $\Delta_{\Sigma_s} + (|A|^2 + \mathrm{Ric}(\nu, \nu))$ is a small perturbation of $\Delta_{\Sigma_s}$. Hence for each $s$, the operator $\Delta_{\Sigma_s}  + (|A|^2 + \mathrm{Ric}(\nu, \nu))$ is an invertible operator from the space $\{f \in C^{2, \frac{1}{2}} (\Sigma_s) : \int_{\Sigma_s} f = 0\}$ to the space $\{f \in C^{\frac{1}{2}} (\Sigma_s) : \int_{\Sigma_s} fv = 0\}$ and we have a universal bound for the norm of the inverse of the operator. 

We now restrict our attention to leaves $\Sigma_s$ that are contained in the region $\{|z| \leq 700\}$. Let us define a function $F^{(a)} : \Sigma_s \to \mathbb{R}$ by $F^{(a)} := \langle Y^{(a)}, \nu\rangle$. The quantity
\[
 \Delta_{\Sigma_s} F^{(a)} + (|A|^2 + \mathrm{Ric}(\nu, \nu))F^{(a)} =: H^{(a)}
\]
can be expressed in terms of $\mathcal L_{Y^{(a)}}(g)$ and first derivatives of $\mathcal L_{Y^{(a)}}(g)$. To see this, let $\phi_\tau$ denote the flow of the vector field $Y^{(a)}$. Then the mean curvature of $\phi_{\tau}(\Sigma_s)$ with respect to $g$ is the same as the mean curvature of $\Sigma_s$ with respect to $\phi_{\tau}^\ast g$. Differentiating both sides with respect to $\tau$ and using that $\Sigma_s$ has constant mean curvature implies the claim. As a consequence, the estimates for $\mathcal L_{Y^{(a)}}(g(t))$ in Step 8 imply 
\[
\sum_{l =0}^{20} |D^l H^{(a)}| \leq CL^{-\frac{1}{2(n-2)}} \varepsilon + C(L)\varepsilon_1\varepsilon
\]
in $\{|z| \leq 600\}$. We define $G^{(a)}(s) := \int_{\Sigma_s} F^{(a)}$ and $\tilde F^{(a)} := F^{(a)} - G^{(a)}(s) \nu$. Then $\int_{\Sigma_s} \tilde F^{(a)} =0$ and 
\[
\Delta_{\Sigma_s} \tilde F^{(a)} + (|A|^2 + \mathrm{Ric}(\nu, \nu))\tilde F^{(a)} = H^{(a)} -\int_{\Sigma_s} H^{(a)} v
\]
on $\Sigma_s$.  Using the estimate for $H^{(a)}$ and the universal bound for the inverse of the Jacobi operator, we conclude that $\sum_{l = 0}^{10} |D^l \tilde F^{(a)}| \leq C L^{-\frac{1}{2(n-2)}} \varepsilon + C(L) \varepsilon_1 \varepsilon$ in the region $\{|z| \leq 600\}$. By our estimates for $v$, we can equivalently write 
\[
\sum_{l = 0}^{10} \big| D^l\big(v^{-1} \langle Y^{(a)}, \nu \rangle - G^{(a)}(s) \big)\big| \leq C L^{-\frac{1}{2(n-2)}} \varepsilon + C(L) \varepsilon_1 \varepsilon
\]
in the region $\{|z| \leq 600\}$. 

If $\Omega_s$ denotes the region bounded by $\Sigma_s$ and $\Sigma_0$, then the divergence theorem gives
\[
G^{(a)}(s) - G^{(a)}(0) = \int_{\Sigma_s} \langle Y^{(a)}, \nu \rangle - \int_{\Sigma_0} \langle Y^{(a)}, \nu \rangle = \int_{\Omega_s} \mathrm{div}\, Y^{(a)}. 
\]
Hence, taking the derivative with respect to $s$, we get 
\[
\frac{d}{ds} G^{(a)}(s) = \pm \int_{\Sigma_s} v\, \mathrm{div}\, Y^{(a)}. 
\]
Since $\mathrm{div}\, Y^{(a)} = \frac{1}{2} \mathrm{tr}( \mathcal L_{Y^{(a)}}(g(t)))$, our estimates in the  previous step imply 
\[
|G^{(a)}(s) - G^{(a)}(0)| \leq CL^{-\frac{1}{2(n-2)}} + C(L)\varepsilon_1 \varepsilon
\]
and 
\[
\sum_{l =1}^{10} \Big| \frac{d^l}{ds^l} G^{(a)}(s) \Big| \leq CL^{-\frac{1}{2(n-2)}} + C(L)\varepsilon_1 \varepsilon 
\]
for $s$ such that $\Sigma_s$ intersects the region $\{|z| \leq 600\}$. Putting the previous three estimates together, noting that $G^{(a)}(0)$ is constant, we conclude that
\[
\sum_{l = 1}^{10} \big| D^l\big(v^{-1} \langle Y^{(a)}, \nu \rangle) \big)\big| \leq C L^{-\frac{1}{2(n-2)}} \varepsilon + C(L) \varepsilon_1 \varepsilon
\]
in the region $\{|z| \leq 600\}$.  \\

\textit{Step 10:} Finally, as in Step 7, we define a family of vector fields $\mathcal Z := \{Z^{(a)} : 1 \leq a \leq {n \choose 2}\}$ by 
\[
Z^{(a)} := \sum_{b, c =1}^{{n \choose 2}} k_{abc} [Y^{(b)}, Y^{(c)}],
\]
where the constant coefficients $k_{abc}$ are the same as in Step 7. The vector fields $Z^{(a)}$ are time-independent. As before, these vector fields agree with the standard rotation vector fields $U^{(a)}_{\mathrm{cyl}}$ on the cylinder up to errors bounded by $C(L)\varepsilon_1$ in the region $\{|z| \leq 800, -400 \leq t \leq t_n\}$. 
Since
\[
\mathcal L_{Z^{(a)}}(g) = \sum_{b,c =1}^{{n \choose 2}}  k_{abc} \big(\mathcal L_{Y^{(b)}} (\mathcal L_{Y^{(c)}}(g))-  \mathcal L_{Y^{(c)}} (\mathcal L_{Y^{(b)}}(g))\big),
\]
our estimates in Step 8 imply 
\[
\sum_{l=0}^{20} \big|D^l\big(\mathcal L_{Z^{(a)}}(g(t)) \big)\big| \leq CL^{-\frac{1}{2(n-2)}} \varepsilon + C(L)\varepsilon_1\varepsilon 
\]
in the region $\{|z| \leq 500, -200 \leq t \leq t_n\}$. 

Next, let us show the vector fields $Z^{(a)}$ are nearly tangent to the foliation. Fix a time $t \in [-200, t_n]$, write $g = g(t)$, and let $v$ and $\nu$ denote the lapse function and the unit normal vector fields to Hamilton's CMC foliation of the $\varepsilon_1$-neck at time $t$. It follows from the definition of the lapse function that the vector field $T:= v^{-1}\nu$ is the gradient of a function. Thus
\[
\langle [Y^{(b)}, Y^{(c)}], T \rangle = \big\langle Y^{(b)}, \nabla (\langle Y^{(c)}, T \rangle) \big\rangle - \big\langle Y^{(c)}, \nabla (\langle Y^{(b)}, T\rangle) \big\rangle. 
\]
The last estimate of Step 9 implies that $|\nabla (\langle Y^{(a)}, T \rangle )| \leq CL^{-\frac{1}{2(n-2)}} \varepsilon + C(L)\varepsilon_1\varepsilon$ and similarly for higher order derivatives of this quantity. Since these estimates hold for any time $t \in [-200, t_n]$ and $Z^{(a)}$ is a linear combination of such terms (with $|k_{abc}| \leq C$), we obtain 
\[
\sum_{l =0}^{8} \big| D^l ( \langle Z^{(a)}, T \rangle) \big| \leq CL^{-\frac{1}{2(n-2)}} \varepsilon + C(L)\varepsilon_1\varepsilon
\]
in the region $\{|z| \leq 500, -200 \leq t \leq t_n \}$ 

The $(0,2)$-tensor $\langle \nabla_{\cdot }\; T, \cdot\; \rangle$ is the Hessian of a function and hence symmetric. From this, we obtain that
\begin{align*}
 \langle \nabla (|T|^2), Z^{(a)} \rangle &= 2\langle \nabla_{Z^{(a)}} T, T \rangle \\
 &= 2 \langle \nabla_T T, Z^{(a)} \rangle + 2 \langle T, \nabla_T Z^{(a)}\rangle -2 \langle \nabla_T Z^{(a)}, T \rangle \\
 & = 2\langle \nabla (\langle Z^{(a)}, T \rangle ), T\rangle -(\mathcal L_{Z^{(a)}} g)(T, T) . 
\end{align*}
Similarly, 
\[
g([T, Z^{(a)}], \, \cdot \,) = (\mathcal L_{Z^{(a)}}(g))(T, \cdot ) - d(\langle Z^{(a)}, T \rangle ).
\]
Thus from our estimates for $\mathcal L_{Z^{(a)}}(g)$ and $\langle Z^{(a)}, T\rangle$, we obtain 
\[
\sum_{l=0}^{6} |D^l(\langle \nabla |T|^2 , Z^{(a)} \rangle) | \leq CL^{-\frac{1}{2(n-2)}} \varepsilon + C(L)\varepsilon_1\varepsilon
\]
and 
\[
\sum_{l=0}^{6} |D^l([T, Z^{(a)}]) | \leq CL^{-\frac{1}{2(n-2)}} \varepsilon + C(L)\varepsilon_1\varepsilon
\]
in the region $\{|z| \leq 500, -200 \leq t \leq t_n \}$.

In summary, recalling our estimates for $v$ and the definition of $T$, we have shown in this step that 
\begin{align*} 
\sum_{l=0}^{20} \big|D^l\big(\mathcal L_{Z^{(a)}}(g(t)) \big)\big| &\leq CL^{-\frac{1}{2(n-2)}} \varepsilon + C(L)\varepsilon_1\varepsilon, \\
\sum_{l =0}^{8} \big| D^l ( \langle Z^{(a)}, \nu \rangle) \big| &\leq CL^{-\frac{1}{2(n-2)}} \varepsilon + C(L)\varepsilon_1\varepsilon,\\
\sum_{l=0}^{6} |D^l(\langle \nabla v , Z^{(a)} \rangle) | &\leq CL^{-\frac{1}{2(n-2)}} \varepsilon + C(L)\varepsilon_1\varepsilon, \\
\sum_{l=0}^{6} |D^l([\nu, Z^{(a)}]) | &\leq CL^{-\frac{1}{2(n-2)}} \varepsilon + C(L)\varepsilon_1\varepsilon, \\
\sum_{l=0}^{6} |D^l([v\nu, Z^{(a)}]) | &\leq CL^{-\frac{1}{2(n-2)}} \varepsilon + C(L)\varepsilon_1\varepsilon
\end{align*}
in the region $\{|z| \leq 500, -200 \leq t \leq t_n \}$. In particular, this shows the family vector fields $\mathcal Z$ satisfies the first two criteria of Definition \ref{neck_symmetry} needed for $(x_0, t_n)$ to be $(CL^{-\frac{1}{2(n-2)}} \varepsilon + C(L)\varepsilon_1\varepsilon)$-symmetric. Additionally, we see that third estimate above implies
\[
\sup_{\Sigma} \big| v - \mathrm{area}_{g(t)}(\Sigma)^{-1}\big| \leq CL^{-\frac{1}{2(n-2)}} \varepsilon + C(L)\varepsilon_1\varepsilon,
\]
for any $t \in [-200, t_n]$ and any leaf $\Sigma$ of Hamilton's CMC foliation at time $t$ that is contained in $\{|z| \leq 400\}$. 
 \\

\textit{Step 11:} In the final four steps of the proof, we work to establish the third and final criterion in Definition \ref{neck_symmetry} need for $(x_0, t_n)$ to be $(CL^{-\frac{1}{2(n-2)}} \varepsilon + C(L)\varepsilon_1\varepsilon)$-symmetric. 

In this step, we use the vector fields $\mathcal Z$ to deduce estimates for the Ricci curvature and second fundamental in the region $\{|z| \leq 400, -200 \leq t \leq t_n\}$. To that end, let us fix a point $(\bar x, \bar t) \in \{|z| \leq 400, -200 \leq t \leq t_n\}$ and let $e_1, \dots, e_{n-1}$ be an orthonormal basis for the tangent space at $\bar x$ of the leaf of Hamilton's CMC foliation passing through $\bar x$ at time $\bar t$. The family of vector fields $\mathcal Z$ is close to the standard family of rotational vector fields $\mathcal U_{\mathrm{cyl}}$ on the cylinder. It follows that we can find a vector $\lambda = (\lambda_1, \dots, \lambda_{{n \choose 2}}) \in \mathbb{R}^{{n \choose 2}}$ such that at the point $(\bar x, \bar t)$
\[
\sum_{a = 1}^{{n \choose 2}} \lambda_a \langle Z^{(a)}, e_i \rangle = 0
\]
for $i \in \{1, \dots, n-1\}$,  
\[
\sum_{a = 1}^{{n \choose 2}} \lambda_a \langle D_{e_1} Z^{(a)}, e_2 \rangle = 1, 
\] 
\[
\sum_{a = 1}^{{n \choose 2}} \lambda_a \langle D_{e_1} Z^{(a)}, e_j \rangle = \sum_{a = 1}^{{n \choose 2}} \lambda_a \langle D_{e_2} Z^{(a)}, e_j \rangle = 0,
\]
for $j \in \{3, \dots, n-1\}$ and 
\[
\sum_{a = 1}^{{n \choose 2}} \lambda_a \langle D_{e_j} Z^{(a)}, e_k\rangle = 0,
\]
for $3 \leq k < j \leq n-1$. Having specified ${n \choose 2}$ equations, we have $|\lambda| \leq C$. 

Using the estimate for $\langle Z^{(a)}, \nu \rangle$ in Step 10, the first identity above implies that $\big|\sum_{a = 1}^{{n \choose 2}} \lambda_a Z^{(a)}\big| \leq CL^{-\frac{1}{2(n-2)}} \varepsilon + C(L)\varepsilon_1\varepsilon$ at the point $(\bar x, \bar t)$. This estimate also implies that $| \langle D_{e_i} Z^{(a)}, \nu \rangle + \langle Z^{(a)}, D_{e_i} \nu \rangle| \leq CL^{-\frac{1}{2(n-2)}} \varepsilon + C(L)\varepsilon_1\varepsilon$ for $i \in \{1, 2\}$. Putting these estimates together implies $\big| \sum_{a = 1}^{{n \choose 2}} \lambda_a \langle D_{e_i} Z^{(a)}, \nu \rangle \big| \leq CL^{-\frac{1}{2(n-2)}} \varepsilon + C(L)\varepsilon_1\varepsilon$ for $i \in \{1, 2 \}$ at the point $(\bar x, \bar t)$. Next, using the estimate for $\mathcal L_{Z^{(a)}}(g)$ in Step 10, we obtain 
\[
\Big|\sum_{a = 1}^{{n \choose 2}} \lambda_a \langle D_{e_1} Z^{(a)}, e_1 \rangle \Big| + \Big|\sum_{a = 1}^{{n \choose 2}} \lambda_a \langle D_{e_2} Z^{(a)}, e_2 \rangle \Big| \leq CL^{-\frac{1}{2(n-2)}} \varepsilon + C(L)\varepsilon_1\varepsilon, 
\]
and
\[
\Big|\sum_{a = 1}^{{n \choose 2}} \lambda_a \langle D_{e_2} Z^{(a)}, e_1 \rangle + 1 \Big|  \leq CL^{-\frac{1}{2(n-2)}} \varepsilon + C(L)\varepsilon_1\varepsilon. 
\]
We can similarly estimate all other combinations $e_j, e_k$ for $3 \leq j \leq n-1$.  The estimates for $\mathcal L_{Z^{(a)}}(g)$ imply that the Ricci tensor satisfies $|\mathcal L_{Z^{(a)}} \mathrm{Ric}| \leq CL^{-\frac{1}{2(n-2)}} \varepsilon + C(L)\varepsilon_1\varepsilon$ for $a \in \{1, \dots, {n \choose 2}\}$. By a straightforward calculation, 
\begin{align*}
(\mathcal L_{Z^{(a)}}\mathrm{Ric})(e_1, e_2) &= (D_{Z^{(a)}} \mathrm{Ric})(e_1, e_2) + \mathrm{Ric}( D_{e_1} Z^{(a)}, e_2) + \mathrm{Ric}(e_1, D_{e_2} Z^{(a)}) \\
& = (D_{Z^{(a)}} \mathrm{Ric})(e_1, e_2) + (\langle D_{e_1} Z^{(a)}, e_1 \rangle +\langle  D_{e_2} Z^{(a)}, e_2 \rangle )\mathrm{Ric}(e_1, e_2)  \\
& \quad +\langle D_{e_1} Z^{(a)}, e_2 \rangle  \mathrm{Ric}(e_2, e_2) + \langle D_{e_2} Z^{(a)}, e_1 \rangle \mathrm{Ric}(e_1, e_1) \\ 
& \quad + \sum_{j =3}^n \big(\langle D_{e_1} Z^{(a)}, e_j \rangle \mathrm{Ric}(e_j, e_2) + \langle D_{e_2} Z^{(a)}, e_j \rangle \mathrm{Ric}(e_1, e_j) \big) \\
& \quad + \langle D_{e_1} Z^{(a)}, \nu \rangle \mathrm{Ric}(\nu, e_2) + \langle D_{e_2} Z^{(a)}, \nu \rangle \mathrm{Ric}(e_1, \nu).
\end{align*}
If we multiply this identity by $\lambda_a$ and sum over $1 \leq a \leq {n \choose 2}$, we conclude that 
\[
|\mathrm{Ric}(e_2, e_2) - \mathrm{Ric}(e_1, e_1) | \leq CL^{-\frac{1}{2(n-2)}} \varepsilon + C(L)\varepsilon_1\varepsilon
\]
at the point $(\bar x, \bar t)$. Since $e_1$ and $e_2$ were arbitrary, we conclude $|\mathrm{Ric}(e_i, e_j) - \frac{1}{n-1} \mathrm{tr}_{\Sigma}(\mathrm{Ric}) \, \delta_{ij}| \leq  CL^{-\frac{1}{2(n-2)}} \varepsilon + C(L)\varepsilon_1\varepsilon$ at the point $(\bar x, \bar t)$, where $\mathrm{tr}_{\Sigma}(\mathrm{Ric}) = \sum_{i =1}^{n-1} \mathrm{Ric}(e_i, e_i)$. 

We obtain an estimate for the second fundamental form in a similar fashion. Let $A$ denote the second fundamental form to Hamilton's CMC foliation. We think of $A$ as a $(0, 2)$-tensor that vanishes in the normal direction. The estimates in Step 10 imply $|\mathcal L_{Z^{(a)}} A | \leq CL^{-\frac{1}{2(n-2)}} \varepsilon + C(L)\varepsilon_1\varepsilon$ for $a \in \{1, \dots, {n \choose 2}\}$. We again compute that 
\begin{align*}
(\mathcal L_{Z^{(a)}}A)(e_1, e_2) &= (D_{Z^{(a)}} A)(e_1, e_2) + A( D_{e_1} Z^{(a)}, e_2) + A(e_1, D_{e_2} Z^{(a)}) \\
& = (D_{Z^{(a)}} A)(e_1, e_2) + (\langle D_{e_1} Z^{(a)}, e_1 \rangle +\langle  D_{e_2} Z^{(a)}, e_2 \rangle )A(e_1, e_2)  \\
& \quad +\langle D_{e_1} Z^{(a)}, e_2 \rangle  A(e_2, e_2) + \langle D_{e_2} Z^{(a)}, e_1 \rangle A(e_1, e_1) \\ 
& \quad + \sum_{j =3}^n \big(\langle D_{e_1} Z^{(a)}, e_j \rangle A(e_j, e_2) + \langle D_{e_2} Z^{(a)}, e_j \rangle A(e_1, e_j) \big).
\end{align*}
If we multiply the identity by $\lambda_a$ and sum over $1 \leq a \leq {n \choose 2}$, we conclude that 
\[
|A(e_2, e_2) - A(e_1, e_1) | \leq CL^{-\frac{1}{2(n-2)}} \varepsilon + C(L)\varepsilon_1\varepsilon
\]
at the point $(\bar x, \bar t)$. Again, since $e_1$ and $e_2$ are arbitrary, we conclude $|A(e_i, e_j) - \frac{1}{n-1} H\, \delta_{ij}| \leq  CL^{-\frac{1}{2(n-2)}} \varepsilon + C(L)\varepsilon_1\varepsilon$ at the point $(\bar x, \bar t)$, where $H$ is the mean curvature of Hamilton's CMC foliation.  
 
Finally, by the estimates in Step 10, if $t \in [-200, t_n]$ and $\Sigma \subset \{|z| \leq 400\}$ is a leaf of Hamiton's CMC foliation at time $t$, then 
\[
\inf_\rho \sup_{\Sigma} \big|\frac{1}{n-1} \mathrm{tr}_{\Sigma}(\mathrm{Ric}) - \rho \big| \leq  CL^{-\frac{1}{2(n-2)}} \varepsilon + C(L)\varepsilon_1\varepsilon. 
\]
To summarize, if $t \in [-200, t_n]$ and $\Sigma$ is a leaf of Hamilton's CMC foliation that intersects $\{|z| \leq 400\}$, then we have
\[
\inf_{\rho} \sup_{\Sigma} \big|(\mathrm{Ric} - \rho g)|_{T\Sigma}\big| \leq CL^{-\frac{1}{2(n-2)}} \varepsilon + C(L)\varepsilon_1\varepsilon,
\]
and 
\[
\sup_{\Sigma} \big|(A - \frac{1}{n-1} H g)|_{T\Sigma}\big| \leq CL^{-\frac{1}{2(n-2)}} \varepsilon + C(L)\varepsilon_1\varepsilon,
\]
where $H$ denotes the mean curvature of $\Sigma$, which is constant. \\

\textit{Step 12:} The estimates established in Step 10 show that for each time $t \in [-200, t_n]$, the family of vector fields $\mathcal Z$ are tangential to Hamilton's CMC foliation at time $t$ up to errors of order $CL^{-\frac{1}{2(n-2)}} \varepsilon + C(L)\varepsilon_1\varepsilon$. Let $x \in \{|z| \leq 400\}$ and let $\Sigma_{x, t}$ be the leaf of Hamilton's CMC foliation passing through $x$ at time $t$. Because the family of vector fields $\mathcal Z$ are close to a standard family of rotational vector fields on the cylinder, for each $x \in \{|z| \leq 400\}$, we can find a subset of vector fields $Z^{(a_1)}, \dots, Z^{(a_{n-1})} \in \mathcal Z$ (here the $a_i$ depend upon $x$, but not upon $t$) satisfying two properties. First, $|Z^{(a_1)}(x) \wedge \cdots \wedge Z^{(a_{n-1})}(x)| \geq C^{-1}$. This is a uniform lower bound for the linear independence of the vector fields at $x$. Second, the spaces $T_x\Sigma_{x,t}$ and $\mathrm{span}\{ Z^{(a_1)}(x), \cdots, Z^{(a_{n-1})}(x)\}$ agree as subspaces of $T_xM$ up to errors of order $CL^{-\frac{1}{2(n-2)}} \varepsilon + C(L)\varepsilon_1\varepsilon$. Since the family of vector fields $\mathcal Z$ is time-independent, this means the spaces $T_x\Sigma_{x,t}$ and $T_x \Sigma_{x, t_n}$ are $(CL^{-\frac{1}{2(n-2)}} \varepsilon + C(L)\varepsilon_1\varepsilon)$-close as subspaces of $T_xM$. We can express both the foliation at time $t$ and the foliation at time $t_n$ as the level sets of suitable functions defined on $\{|z| \leq 500\}$. These remarks imply the functions must be close in $C^1$-norm. We conclude that every leaf in of Hamilton's CMC foliation at time $t$ which intersects the region $\{|z| \leq 400\}$ is $(CL^{-\frac{1}{2(n-2)}} \varepsilon + C(L)\varepsilon_1\varepsilon)$-close in the $C^1$-norm to a leaf of Hamilton's CMC foliation at time $t_n$ (as submanifolds of $M$).  \\

\textit{Step 13:} Again, let us fix a time $t \in [-200, t_n]$ and let $\Sigma_s$ denote Hamilton's CMC foliation at time $t$. Let $\nu$ and $v$ denote the associated normal vector field and the lapse function of this foliation. Using our estimates for the second fundamental form in the previous step, we show the quantity $\mathrm{area}_{g(t)}(\Sigma_s)^{-\frac{n+1}{n-1}} \int_{\Sigma_s} \langle Z^{(a)}, Z^{(b)} \rangle_{g(t)} \,d\mu_{g(t)}$ is nearly constant as a function of $s$. Flowing by the vector field $v \nu = \frac{\partial}{\partial s}$ for time $s$ moves $\Sigma_{s_0}$ into $\Sigma_{s_0 + s}$. By standard evolution equations for a flow with velocity $v \nu$, we have $\frac{d }{d s} d\mu_{g(t)} = H v \, d\mu_{g(t)}$ and consequently $\frac{d }{d s} \mathrm{area}_{g(t)}(\Sigma_s) = \int_{\Sigma_s} H v \, d\mu_{g(t)} = H$, where $H$ denote the mean curvature of $\Sigma_s$ with respect to $g(t)$. Moreover, we can compute that 
\begin{align*}
\frac{d}{d s} \Big(\int_{\Sigma_s} \langle Z^{(a)}, Z^{(b)} \rangle_{g(t)} \, d\mu_{g(t)}\Big) & = \int_{\Sigma_s}Hv \langle Z^{(a)}, Z^{(b)} \rangle_{g(t)} \, d\mu_{g(t)} + \int_{\Sigma_s}\big(\mathcal L_{v\nu}(g)\big)( Z^{(a)}, Z^{(b)}) \, d\mu_{g(t)}  \\
& \quad + \int_{\Sigma_s}\langle [v \nu, Z^{(a)}], Z^{(b)} \rangle_{g(t)} \, d\mu_{g(t)} + \int_{\Sigma_s}\langle Z^{(a)}, [v\nu, Z^{(b)}] \rangle_{g(t)} \, d\mu_{g(t)}.
\end{align*}
Since $\mathcal L_{v\nu}(g) = \frac{\partial}{\partial s} g(t) = 2Av$, the estimate $|A - \frac{1}{n-1} H g| \leq CL^{-\frac{1}{2(n-2)}} \varepsilon + C(L)\varepsilon_1\varepsilon$ implies 
\[
|\big(\mathcal L_{v\nu}(g)\big)( Z^{(a)}, Z^{(b)}) -\frac{2}{n-1} H \langle Z^{(a)}, Z^{(b)}\rangle_{g(t)} | \leq CL^{-\frac{1}{2(n-2)}} \varepsilon + C(L)\varepsilon_1\varepsilon.
\]
By our estimate for $[v \nu, Z^{(a)}]$ established in Step 10, 
\[
\big|\langle [v \nu, Z^{(a)}], Z^{(b)} \rangle_{g(t)}\big| + \big|\langle Z^{(a)}, [v\nu, Z^{(b)}] \rangle_{g(t)}\big| \leq CL^{-\frac{1}{2(n-2)}}\varepsilon + C(L)\varepsilon_1\varepsilon.
\] 
Plugging these estimates into the identity above, we obtain 
\[
\Big|\frac{d}{d s} \Big(\int_{\Sigma_s} \langle Z^{(a)}, Z^{(b)} \rangle_{g(t)} \, d\mu_{g(t)}\Big) - \frac{n+1}{n-1} \int_{\Sigma_s} H v \langle Z^{(a)}, Z^{(b)} \rangle_{g(t)} \, d\mu_{g(t)}\Big|\leq CL^{-\frac{1}{2(n-2)}} \varepsilon + C(L)\varepsilon_1\varepsilon.
\]
By our estimate for $v$, this gives 
\[
\Big|\frac{d}{d s} \Big(\int_{\Sigma_s} \langle Z^{(a)}, Z^{(b)} \rangle_{g(t)} \, d\mu_{g(t)}\Big) -  \frac{\frac{n+1}{n-1}H}{\mathrm{area}_{g(t)}(\Sigma_s)} \int_{\Sigma_s} \langle Z^{(a)}, Z^{(b)} \rangle_{g(t)} \, d\mu_{g(t)}\Big|\leq CL^{-\frac{1}{2(n-2)}} \varepsilon + C(L)\varepsilon_1\varepsilon.
\]
Recalling the evolution equation for $\mathrm{area}_{g(t)}(\Sigma_s)$, we conclude that
\[
\Big|\frac{d}{ds} \Big( \mathrm{area}_{g(t)}(\Sigma_s)^{-\frac{n+1}{n-1}} \int_{\Sigma_s} \langle Z^{(a)}, Z^{(b)} \rangle_{g(t)} \, d\mu_{g(t)} \Big)\Big| \leq CL^{-\frac{1}{2(n-2)}} \varepsilon + C(L)\varepsilon_1\varepsilon.
\]

\textit{Step 14:} By the estimate of the previous step, we deduce that there exists a symmetric ${n \choose 2} \times {n \choose 2}$ matrix $Q_{ab}$ (independent of $\Sigma$) such that 
\[
\Big|Q_{ab} - \mathrm{area}_{g(t_n)}(\Sigma)^{-\frac{n+1}{n-1}} \int_{\Sigma} \langle Z^{(a)}, Z^{(b)} \rangle_{g(t_n)} \, d\mu_{g(t_n)} \Big| \leq CL^{-\frac{1}{2(n-2)}} \varepsilon + C(L)\varepsilon_1\varepsilon
\]
for every leaf $\Sigma \subset \{|z| \leq 300\}$ of Hamilton's CMC foliation at time $t_n$. Since the family of vector fields $\mathcal Z$ are close the a standard rotational family of vector fields on the cylinder, the eigenvalues of the matrix $Q_{ab}$ lie in an interval $[\frac{1}{C}, C]$ for some fixed constant $C$. 

By an argument similar to the one given in the previous step, the estimate for the Ricci tensor established in Step 11 implies
\[
\Big|\frac{d}{dt} \Big( \mathrm{area}_{g(t)}(\Sigma)^{-\frac{n+1}{n-1}} \int_{\Sigma} \langle Z^{(a)}, Z^{(b)} \rangle_{g(t)} \, d\mu_{g(t)} \Big)\Big| \leq CL^{-\frac{1}{2(n-2)}} \varepsilon + C(L)\varepsilon_1\varepsilon,
\]
for every leaf $\Sigma \subset \{|z| \leq 300\}$ of Hamilton's CMC foliation at time $t_n$. It follows that 
\[
\Big|Q_{ab} - \mathrm{area}_{g(t)}(\Sigma)^{-\frac{n+1}{n-1}} \int_{\Sigma} \langle Z^{(a)}, Z^{(b)} \rangle_{g(t)} \, d\mu_{g(t)} \Big| \leq CL^{-\frac{1}{2(n-2)}} \varepsilon + C(L)\varepsilon_1\varepsilon
\]
for every leaf $\Sigma \subset \{|z| \leq 300\}$ of Hamilton's CMC foliation at time $t_n$.

By Step 12, every leaf of Hamilton's CMC foliation of the $\varepsilon_1$-neck at time $t$ which is contained in $\{|z| \leq 200\}$ is $(CL^{-\frac{1}{2(n-2)}} \varepsilon + C(L)\varepsilon_1\varepsilon)$-close in the $C^1$-norm to a leaf of Hamilton's CMC foliation of the $\varepsilon_1$-neck at time $t_n$. This finally implies that 
\[
\Big|Q_{ab} - \mathrm{area}_{g(t)}(\Sigma)^{-\frac{n+1}{n-1}} \int_{\Sigma} \langle Z^{(a)}, Z^{(b)} \rangle_{g(t)} \, d\mu_{g(t)} \Big)\Big| \leq CL^{-\frac{1}{2(n-2)}} \varepsilon + C(L)\varepsilon_1\varepsilon
\]
whenever $t \in [-200, t_n]$ and $\Sigma \subset \{|z| \leq 200\}$ is a leaf of Hamilton's CMC foliation of the $\varepsilon_1$-neck at time $t$. 

The vector fields $\sum_{b =1}^{{n\choose 2}} (Q^{-\frac{1}{2}})_{ab} Z^{(a)}$ now satisfy all three conditions of Definition \ref{neck_symmetry}. We conclude the point $(x_0, t_n)$ is $(CL^{-\frac{1}{2(n-2)}} \varepsilon + C(L)\varepsilon_1\varepsilon)$-symmetric. Hence, if we  first take $L$ sufficiently large and then take $\varepsilon_1$ sufficiently small (depending upon $L$), then $(x_0, t_n)$ is $\frac{\varepsilon}{2}$-symmetric. This completes the proof of Theorem \ref{neck_improvement}.
\end{proof}


\section{Rotational Symmetry of Ancient $\kappa$-Solutions}

In this section, we give a proof of Theorem \ref{main}. Throughout this section we assume $n \geq 4$ and $(M, g(t))$, $t \in (-\infty, 0]$ is an $n$-dimensional ancient $\kappa$-solution which is noncompact, $\alpha$-uniformly PIC, and strictly PIC2. In this section, unless otherwise indicated $C$ always denotes a universal constant, which may depend upon $n$, $\kappa$, or $\alpha$. For each $t$, we denote by $R_{\max}(t)$ the supremum of the scalar curvature of $(M, g(t))$. By the extension of Hamilton's Harnack inequality to PIC2 solutions (Theorem \ref{harnack_inequality}), $R_{\max}(t)$ is nondecreasing in $t$. By the extension of Perelman's pointwise derivative estimates to $\kappa$-solutions in higher dimensions (Theorem \ref{pointwise_derivative_estimates}), $t \mapsto R_{\max}(t)^{-1}$ is uniformly Lipschitz continuous. Let us define $r_{\max}(t)$ by the identity $R_{\max}(t) =  r_{\max}(t)^{-2}$. For points $x \in M$ that lie at the center of a neck, we define the curvature scale $r_{\mathrm{neck}}(x, t)$ by the identity $R(x, t) = (n-2)(n-1) r_{\mathrm{neck}}(x,t)^{-2}$. This is the scale we have used in Definition \ref{epsilon_neck} and Definition \ref{neck_symmetry}. Note that 
\[
r_{\max}(t)^2 \leq r_{\mathrm{neck}}(x, t)^2. 
\]

We begin by choosing a large constant $L$ and a small constant $\varepsilon_1 \leq \varepsilon(n)$ such that the gluing results of Section 4 can be applied to any $\varepsilon_1$-neck and so that the conclusion of the Neck Improvement Theorem holds.  For each spacetime point $(x, t)$, we let $\lambda_1(x, t)$ denote the smallest eigenvalue of the Ricci tensor at $(x, t)$. 

\begin{proposition}\label{neck_detection}
Given $\varepsilon_1$, we can find a small constant $\theta:=\theta(n, \varepsilon_1)$ with the following property. Suppose that $(\bar x, \bar t)$ is a spacetime point satisfying $\lambda_1(\bar x, \bar t) \leq \theta R(\bar x, \bar t)$. Then $(\bar x, \bar t)$ lies at the center of an evolving $\varepsilon_1$-neck $N$. Moreover, if $x$ lies outside the compact domain bounded by the leaf of Hamilton's CMC foliation passing through $\bar x$ at time $\bar t$, then $(x, \bar t)$ lies at the center of an evolving $\varepsilon_1$-neck. 
\end{proposition}
\begin{proof}
As in the proof of Proposition 9.1 in \cite{Bre20}, this result only depends upon the structure results, Theorem \ref{structure_thm} and Corollary \ref{structure_cor} (cf. Theorem A.2 and Corollary A.3 in \cite{Bre20}) and Lemma \ref{neck_detection_lemma}, the Neck Detection Lemma. Thus, we can follow the proof of Proposition 9.1 in \cite{Bre20} verbatim. 
\end{proof}

For the remainder of this section, let us fix a choice of $\theta$ such that the proposition above holds. We now define a notion of $\varepsilon$-symmetry for the cap of our manifold, following Definition 9.2 in \cite{Bre20}.

\begin{definition}[Cap Symmetry]\label{cap_symmetry}
We will say the flow is $\varepsilon$-symmetric at time $\bar t$ if there exists a compact domain $D \subset M$ and a family of time-independent vector fields $\mathcal U = \{U^{(a)} : 1 \leq a \leq {n \choose 2}\}$ which are defined on an open subset containing $D$ such that the following statements hold: 
\begin{enumerate}
\item[$\bullet$] There exists a point $x \in \partial D$ such that $\lambda_1(x, \bar t) < \theta R(\bar x, \bar t)$. 
\item[$\bullet$] For each $x \in D$, we have $\lambda_1(x, \bar t) > \frac{1}{2} \theta R(x, \bar t)$. 
\item[$\bullet$] The boundary $\partial D$ is a leaf of Hamilton's CMC foliation at time $\bar t$. 
\item[$\bullet$] For each $x \in M \setminus D$, the point $(x, \bar t)$ is $\varepsilon$-symmetric in the sense of Definition \ref{neck_symmetry}.
\item[$\bullet$]  In $D \times  [\bar t - r_{\max}(\bar t)^{2}, \bar t]$, we have the estimate
\[
 \sum_{l = 0}^2 \sum_{a = 1}^{{n \choose 2}} r_{\max}(\bar t)^{2l} \big|D^l (\mathcal L_{U^{(a)}}(g(t)))|^2 \leq \varepsilon^2
\]
\item[$\bullet$] If $\Sigma \subset D$ is a leaf of Hamilton's CMC foliation of $(M, g(\bar t))$ that has distance at most $50\,r_{\mathrm{neck}}(\partial D)$ from $\partial D$, then 
\[
\sup_{\Sigma} \sum_{a = 1}^{{n \choose 2}} r_{\max}(\bar t)^{-2}|\langle U^{(a)}, \nu \rangle|^2 \leq \varepsilon^2, 
\]
where $\nu$ is the unit normal vector to $\Sigma$ in $(M, g(\bar t))$ and $r_{\mathrm{neck}}(\partial D)$ is defined by the identity $\mathrm{area}_{g(\bar t)}(\partial D) = \mathrm{area}_{g_{S^{n-1}}}(S^{n-1}) r_{\mathrm{neck}}(\partial D)^{n-1}$. 
\item[$\bullet$] If $\Sigma \subset D$ is a leaf of Hamilton's CMC foliation of $(M, g(\bar t))$ that has distance at most $50 \,r_{\mathrm{neck}}(\partial D)$ from $\partial D$, then 
\[
\sum_{a, b = 1}^{{n \choose 2}} \bigg| \delta_{ab} - \mathrm{area}_{g(\bar t)}(\Sigma)^{-\frac{n+1}{n-1}} \int_\Sigma \langle U^{(a)}, U^{(b)} \rangle_{g(\bar t)} \, d\mu_{g(\bar t)} \bigg|^2 \leq \varepsilon^2, 
\]
where $r_{\mathrm{neck}}(\partial D)$ is defined by the identity $\mathrm{area}_{g(\bar t)}(\partial D) = \mathrm{area}_{g_{S^{n-1}}}(S^{n-1}) r_{\mathrm{neck}}(\partial D)^{n-1}$. 
\end{enumerate}
\end{definition}

\begin{remark}
Note that by Proposition \ref{neck_detection}, for each $x \in M \setminus D$, the point $(x, \bar t)$ lies at the center of an evolving $\varepsilon_1$-neck. 
\end{remark}

\begin{remark}
Since $D \subset \{x \in M : \lambda_1(x, \bar t) > \frac{1}{2} \theta R(x, \bar t)\}$, Corollary \ref{structure_cor} implies $\mathrm{diam}_{g(\bar t)}(D) \leq C \,r_{\max}(\bar t)$ and $\frac{1}{C}R_{\max}(t) \leq R(x, t) \leq R_{\max}(t)$ for all $x \in D$. Here $C$ is a positive constant that depends upon $\theta$. 
\end{remark}

\begin{lemma}
Suppose that the flow is $\varepsilon$-symmetric at time $\bar t$. If $\tilde t$ is sufficiently close to $\bar t$, then the flow is $2\varepsilon$-symmetric at time $\bar t$. 
\end{lemma}
\begin{proof}
The proof is analogous to the proof of Lemma 9.5 in \cite{Bre20}. 
\end{proof}

\begin{lemma}\label{symmetric}
Let us fix a time $\bar t$. Suppose that, for each $\varepsilon > 0$, the solution $(M, g(\bar t))$ is $\varepsilon$-symmetric. Then $(M, g(\bar t))$ is rotationally symmetric. 
\end{lemma}
\begin{proof}
The proof is analogous to the proof of Lemma 9.6 in \cite{Bre20}. 
\end{proof}

We now continue with the proof of the main theorem. In the following proposition, we use the uniqueness of the Bryant soliton in higher dimensions \cite{Bre14} and Hamilton's rigidity result for the Harnack inequality (adapted to our setting in \cite{Bre09}) to produce a sequence of times $\hat t_k$ along which rescalings of the solution converge to the Bryant soliton. 

\begin{proposition}\label{bryant_by_harnack}
There exists a sequence of times $\hat t_k \to -\infty$ and a sequence of points $\hat p_k \in M$ with the following property. If we define rescaled solutions of the Ricci flow by 
\[
\hat g_k(t) := R_{\max}(\hat t_k) g(\hat t_k + t\, R_{\max}(\hat{t}_k)^{-1}),
\]
then the solutions $(M, \hat g_k(t), \hat p_k)$ converge (subsequentially) to the Bryant soliton in the pointed Cheeger-Gromov sense. Moreover, the points $\hat p_k$ converge to the tip of the Bryant soliton, and $\frac{R(\hat p_k, \hat t_k)}{R_{\max}(\hat t_k)} \to 1$ as $k \to \infty$. 
\end{proposition}
\begin{proof}
The proof begins precisely as in the proof of Proposition 9.7 in \cite{Bre20}. The solution $(M, g(t))$ must be a Type II ancient solution (see Proposition \ref{type_ii}) and therefore, we can extract a Type II blow-up limit as Hamilton does in Section 16 of \cite{Ham95}. Namely, we choose points $(\hat p_k, \hat t_k) \in M \times (-k, 0)$ with the property that 
\[
\sup_{(x, t) \in M \times (-k ,0)} \Big(1 +\frac{t}{k}\Big) (-t) R(x, t) \leq \Big(1 +\frac{1}{k}\Big)\Big(1 + \frac{\hat t_k}{k}\Big)(-\hat t_k) R(\hat p_k, \hat t_k).
\]
Since the solution is Type II, it follows that $\hat t_k \to -\infty$ and $R_{\max}(\hat t_k) \leq (1 + k^{-1}) R(\hat p_k, \hat t_k)$. As in \cite{Bre20}, we can extract a convergent subsequence to obtain a complete eternal solution of the Ricci flow with scalar curvature bounded by $1$ at each point in spacetime. Moreover, there exists a spacetime point in the limit where the scalar curvature is equal to $1$. The limit is weakly PIC2 and uniformly PIC. Since the solution attains equality in Hamilton's Harnack equality, it follows from Proposition 14 in \cite{Bre09} (which generalizes the main result of Hamilton in \cite{Ham93b}) that the limit is a steady gradient Ricci soliton. This soliton must be asymptotically cylindrically in the sense used in \cite{Bre14}. See Proposition \ref{soliton_cylindrical} in Appendix A. Hence, by Theorem 1.2 in \cite{Bre14}, the limit must be the rotationally symmetric Bryant soliton.  This completes the proof. 
\end{proof}

\begin{corollary}
There exists a sequence $\hat \varepsilon_k \to 0$ with the following properties. For each $t \in [\hat t_k - \hat \varepsilon_k^{-2} r_{\max}(\hat t_k)^{2}, \hat t_k ]$, we have $(1 - \hat \varepsilon_k) R_{\max}(\hat t_k) \leq R(\hat p_k, t) \leq R_{\max}(t) \leq R_{\max}(\hat t_k)$. Moreover, for each $ t \in [\hat t_k - \hat \varepsilon_k^{-2} r_{\max}(\hat t_k)^{2}, \hat t_k ]$, the solution $(M, g(t))$ is $\hat \varepsilon_k$-symmetric.  
\end{corollary}
\begin{proof}
The Harnack inequality \cite{Bre09} implies $R_{\max}(t) \leq R_{\max}(\hat t_k)$. The remaining statements follow from Proposition \ref{bryant_by_harnack} and the structure result Theorem \ref{structure_thm} in Appendix A. 
\end{proof}

\textbf{From now on, we assume that the ancient solution $(M, g(t))$ is not rotationally symmetric.} The proof of Theorem \ref{main} will be by contradiction. By Theorem \ref{structure_thm} and Corollary \ref{structure_cor}, we can find a sequence of positive real numbers $\varepsilon_k$ with the following properties: 
\begin{enumerate}
\item[$\bullet$] $\varepsilon_k \to 0$.
\item[$\bullet$] $\varepsilon_k \geq 2 \hat \varepsilon_k$. 
\item[$\bullet$] If a spacetime point $(x, t)$ satisfies $R(x, t) \leq \hat \varepsilon_k R_{\max}(t)$, then $(x, t)$ lies at the center of an evolving $\varepsilon_k^2$-neck. 
\end{enumerate}
For each $k$ large, we define 
\[
t_k = \inf \{t \in [\hat t_k, 0] : \text{the flow is not $\varepsilon_k$-symmetric at time $t$}\}. 
\]
For abbreviation, let $r_k = r_{\max}(t_k) = R_{\max}(t_k)^{-\frac{1}{2}}$. Note that by the Harnack inequality, $R_{\max}(t) \leq r_k^{-2}$ for all $t \leq t_k$. 

\begin{lemma}\label{epsilon_k_symmetric}
If $t \in [\hat t_k - \hat \varepsilon_k^{-2} r_{\max}(\hat t_k)^{2}, t_k)$, then $(M, g(t))$ is $\varepsilon_k$-symmetric. In particular, if $(x, t) \in M \times [\hat t_k - \hat \varepsilon_k^{-2}r_{\max}(\hat t_k)^{2}, t_k)$ is a spacetime point satisfying $\lambda_1(x, t) < \frac{1}{2} \theta R(x, t)$, then the point $(x, t)$ is $\varepsilon_k$-symmetric. 
\end{lemma}
\begin{proof}
By the corollary above and the property $\hat \varepsilon_k \leq \frac{1}{2} \varepsilon_k$,  $(M, g(t))$ is $\varepsilon_k$-symmetric for every $t \in [\hat t_k - \hat \varepsilon_k^{-2}r_{\max}(\hat t_k)^2, \hat t_k]$. On the other hand, by definition of $t_k$, $(M, g(t))$ is $\varepsilon_k$-symmetric for every $t \in [\hat t_k, \hat t_k)$. Consequently, for any $t \in [\hat t_k - \hat \varepsilon_k^{-2}r_{\max}(\hat t_k)^{2}, t_k)$, Definition \ref{cap_symmetry} implies that any point $x \in M$ satisfying $\lambda_1(x, t) < \frac{1}{2} \theta R(x, t)$ must be $\varepsilon_k$-symmetric. 
\end{proof}

\begin{lemma}
The sequence $t_k$ satisfies $\lim_{k \to \infty} t_k = -\infty$. 
\end{lemma}
\begin{proof}
As in the proof of Lemma 9.10 in \cite{Bre20}, if $\limsup_{k \to \infty} t_k > -\infty$, then $(M, g(t))$ must be rotationally symmetric for all $t < \limsup_{k \to \infty} t_k$ and therefore rotationally symmetric for all $t$, contradictory to our assumption.
\end{proof}

In the next step, we show rescalings of the solution at time $t_k$ converge to the Bryant soliton. This time, we use the classification of rotationally symmetric ancient solutions in \cite{LZ20} to deduce the result. 

\begin{proposition}
There exists a sequence of points $p_k \in M$ with the following properties. If we define rescaled solutions of the Ricci flow by $g_k(t) = r_k^{-2}g(t_k + r_k^{2}t)$, then the solutions $(M, g_k(t), p_k)$ (subsequentially) converge to the Bryant soliton in the pointed Cheeger-Gromov sense. Moreover, the points $p_k$ converge to the tip of the Bryant soliton, and $r_k^2 R(p_k, t_k) \to 1$ as $k \to \infty$. 
\end{proposition}
\begin{proof}
The proof begins precisely as in the proof of Proposition 9.11 in \cite{Bre20}. Not every point in $M$ can lie at the center of a neck at time $t_k$. Thus, we can find a sequence of points $q_k \in M$ such that $\liminf_{k \to \infty} \frac{\lambda_1(q_k, t_k)}{R(q_k, t_k)} > 0$. By Corollary \ref{structure_cor}, $R(q_k, t_k)$ is comparable to $R_{\max}(t_k)$.  Hence, $\liminf_{k \to \infty} r_k^2R(q_k, t_k) > 0$ and therefore $\liminf_{k \to \infty} r_k^2 \lambda_1(q_k, t_k) > 0$. Now if we rescale the solution by $r_k^{-1}$, then after passing to a subsequence, the rescaled solutions $(M, g_k(t), p_k)$ converge as $k \to \infty$ to a noncompact ancient $\kappa$-solution $(M^{\infty}, g^{\infty}(s))$. Since $\liminf_{k \to \infty} r_k^2 \lambda_1(q_k, t_k) > 0$, the limit does not split a line and the strong maximum principle implies the limit must be strictly PIC2.  From here, the proof in \cite{Bre20} shows that solution $(M^{\infty}, g^{\infty}(s))$ must be rotationally symmetric. Therefore, by the main result of \cite{LZ20} (see Corollary \ref{li_zhang_pic}), the limit must be the Bryant soliton. Let $p_{\infty} \in M^{\infty}$ denote the tip of the limiting soliton. The sequence of points $p_k$ is any sequence of points in $M$ converging to $p_{\infty}$. As in \cite{Bre20}, $R_{g^{\infty}(0)}(p_{\infty}) = 1$, and thus   $r_k^2 R(p_k, t_k) \to 1$ as $k \to \infty$. 
\end{proof}

We will use the approximate soliton potential $f$ as a proxy for the distance function.

\begin{corollary}\label{soliton_approximation}
There exists a sequence of positive real numbers $\delta_k \to 0$ such that $\delta_k \geq 2 \varepsilon_k$ for each $k$ and the following statements hold when $k$ is sufficiently large: 
\begin{enumerate}
\item[$\bullet$] For each $t \in [t_k - \delta_k^{-1}r_k^2, t_k]$, we have $\frac{1}{n}(1 - \delta_k) g \leq r_k^2\mathrm{Ric} \leq (1 + \delta_k)\frac{1}{n} g$ at the point $(p_k, t)$. 
\item[$\bullet$] The scalar curvature satisfies 
\[
\frac{1}{2K}(r_k^{-1} d_{g(t)}(p_k, x) + 1)^{-1} \leq r_k^2 R(x, t) \leq 2K( r_k^{-1} d_{g(t)}(p_k, x) + 1)^{-1}
\]
 for all points $(x, t) \in B_{g(t_k)}(p_k, \delta_k^{-1} r_k) \times  [t_k - \delta_k^{-1}r_k^2, t_k]$. 
\item[$\bullet$] There exists a nonnegative function $ f: B_{g(t_k)}(p_k, \delta_k^{-1} r_k) \times  [t_k - \delta_k^{-1}r_k^2, t_k] \to \mathbb{R}$ (depending upon $k$) such that 
\[
|\mathrm{Ric} - D^2 f| \leq \delta_k r_k^{-2}, \quad |\Delta f + |\nabla f|^2 - r_k^{-2}| \leq \delta_k r_k^{-2}, \quad\text{and} \quad  |\frac{\partial}{\partial t} f + |\nabla f|^2 | \leq \delta_k r_k^{-2}.
\]
\item[$\bullet$] The function $f$ satisfies 
\[
\frac{1}{2K}(r_k^{-1} d_{g(t)}(p_k, x) + 1) \leq f(x, t) + 1 \leq 2K (r_k^{-1} d_{g(t)}(p_k, x) + 1)
\]
 for all points $(x, t) \in B_{g(t_k)}(p_k, \delta_k^{-1} r_k) \times  [t_k - \delta_k^{-1}r_k^2, t_k]$.
\end{enumerate}
Here, $K := K(n) \geq 10$ is a universal constant. 
\end{corollary}
\begin{proof}
The proof is analogous to the proof of Corollary 9.12 in \cite{Bre20}. 
\end{proof}

As an immediate corollary of the corollary above, we obtain:

\begin{corollary}\label{curvature_bounds}
For each $t \in [t_k - \delta_k^{-1} r_k^2, t_k]$, we have $(1 - \delta_k^{-1})r_k^{-2} \leq R(p_k, t) \leq R_{\max}(t) \leq r_k^{-2}$.
\end{corollary}

\begin{lemma}\label{distance_estimate}
The time derivative of the distance function satisfies the estimate $0 \leq - \frac{d}{dt} d_{g(t)}(p_k, x) \leq 2 n r_k^{-1}$ for all $(x, t) \in M \times [t_k -\delta_k^{-1} r_k^2, t_k]$. 
\end{lemma}
\begin{proof}
Let $x \in M$ and let $\ell = d_{g(t)}(p_k, x)$. Let $\gamma : [0, \ell] \to M$ be a minimizing geodesic between $p_k$ and $x$ and let $X(s) = \gamma'(s)$. Then
\[
\frac{d}{dt} d_{g(t)}(p_k, x) = - \int_0^\ell \mathrm{Ric}(X, X) \, ds. 
\]
Now $0 \leq \mathrm{Ric}(X, X) \leq R_{\max}(t) \leq r_k^{-2}$, by the Harnack inequality. If $\ell > 2r_k$, then Lemma 8.3(b) in \cite{Per02} implies 
\[
0 \leq - \frac{d}{dt} d_{g(t)}(p_k, x) \leq (2(n-1) + \frac{4}{3})r_k^{-1} \leq 2n r_k^{-1}.
\]
Otherwise, since $\mathrm{Ric}(X,X) \leq r_k^{-2}$, we get
\[
0 \leq - \frac{d}{dt} d_{g(t)}(p_k, x) \leq r_k^{-2}\ell \leq 2r_k^{-1}.
\]
\end{proof}

Using Theorem \ref{structure_thm} and Corollary \ref{curvature_bounds} , we can choose a large positive real number $\Lambda$ depending upon $n$, $L$, and $\varepsilon_1$ with the following properties:
\begin{enumerate}
\item[$\bullet$] $L \sqrt{\frac{4n^3K}{\Lambda}} \leq 10^{-6}$.
\item[$\bullet$] if $(\bar x, \bar t) \in M \times [t_k - \delta_k^{-1} r_k^2, t_k]$ is a spacetime point satisfying $d_{g(\bar t)}(p_k, \bar x) \geq \Lambda r_k$, then $\lambda_1(x, t) < \frac{1}{2} \theta R(x, t)$ for all points $(x, t) \in B_{g(\bar t)}(\bar x, L \,r_{\mathrm{neck}}(\bar x, \bar t)) \times [\bar t - L \,r_{\mathrm{neck}}(\bar x, \bar t)^2, \bar t]$. 
\end{enumerate}
The idea is to first choose $\varepsilon_2$ small enough depending upon $\varepsilon_1$ so that the conclusion of the second condition holds if $(x, t)$ lies on an evolving $\varepsilon_2$-neck. Then we choose $\Lambda$ large enough so that every point outside $B_{g(\bar t)}(p_k, \Lambda r_k)$ lies at the center of an $\varepsilon_2$-neck. 

The next two results show that points sufficiently far away have improved symmetry. 

\begin{lemma}\label{symmetry_outside_lambda}
If $k$ is sufficiently large, then the following holds. If $(\bar x, \bar t) \in M \times [t_k - \delta_k^{-1} r_k^2, t_k]$ satisfies $d_{g(t_k)}(p_k, \bar x) \geq \Lambda r_k$, then $(\bar x, \bar t)$ is $\frac{\varepsilon_k}{2}$-symmetric. 
\end{lemma}
\begin{proof}
The proof is the same as the proof of Lemma 9.15 in \cite{Bre20}, except for a small difference in how we have defined the scale of a neck in higher dimensions. Let us briefly verify that the proof works here. 

If $R(\bar x, \bar t) \leq \hat \varepsilon_k R_{\max}(\bar t)$, then as in \cite{Bre20}, our choice of $\varepsilon_k$ ensure the point $(\bar x, \bar t)$ if $\frac{\varepsilon_k}{2}$-symmetric. On the other hand, suppose $R(\bar x, \bar t) \geq \hat \varepsilon_k R_{\max}(\bar t)$. For $k$ large, Corollary \ref{curvature_bounds} implies $R(\bar x, \bar t) \geq \frac{1}{2} \hat \varepsilon_k r_k^{-2}$. This means $r_{\mathrm{neck}}(\bar x, \bar t) \leq 2(n-1)(n-2) \hat \varepsilon_k^{-1}  r_k^2$. The Harnack inequality implies $r_k^{2} \leq r_{\max}(\hat t_k)^{2}$. Recall $\delta_k \geq 2 \varepsilon_k \geq 4 \hat{\varepsilon}_k$. Thus, 
\begin{align*}
\bar t - L \,r_{\mathrm{neck}}(\bar x, \bar t)^2 &\geq t_k - \delta_k^{-1} r_k^2 - 2(n-1)(n-2) L \hat \varepsilon_k^{-1} r_k^2 \\
& \geq \hat  t_k - \delta_k^{-1}r_{\max}(\hat t_k)^{2}  - 2(n-1)(n-2) L \hat \varepsilon_k^{-1} r_{\max}(\hat t_k)^{2} \\
& \geq \hat t_k - \hat \varepsilon_k^{-2} r_{\max}(\hat t_k)^{2},
\end{align*}
for $k$ sufficiently large. Consequently, by definition of $\Lambda$ and Lemma \ref{epsilon_k_symmetric} every point in $B_{g(\bar t)}(\bar x, L \,r_{\mathrm{neck}}(\bar x, \bar t)) \times [\bar t - L \,r_{\mathrm{neck}}(\bar x, \bar t)^2, \bar t)$ is $\varepsilon_k$-symmetric. Hence, by the Neck Improvement Theorem, the point $(\bar x, \bar t)$ is $\frac{\varepsilon_k}{2}$-symmetric. 
\end{proof}

\begin{proposition}\label{symmetry_at_infinity}
If $k$ is sufficiently large, then the following holds. If $(\bar x, \bar t) \in M \times [t_k - 2^{-j} \delta_k^{-1} r_k^2, t_k]$ satisfies $2^{\frac{j}{400}} \Lambda r_k \leq d_{g(\bar t)}(p_k, \bar x) \leq (400n^3KL)^{-j} \delta_k^{-1} r_k$, then $(\bar x, \bar t)$ is $2^{-j-1}\varepsilon_k$-symmetric. 
\end{proposition}
\begin{proof}
The proof is the same as the proof of Lemma 9.16 in \cite{Bre20}, except for a small difference in how we have defined the scale of the neck in higher dimensions. Let us verify that the proof works here. 

The previous lemma shows the result hold for $j = 0$. Assume $j \geq 1$ and that the assertion holds for $j -1$. Suppose $(\bar x, \bar t) \in M \times [t_k - 2^{-j} \delta_k^{-1} r_k^2, t_k]$ such that $2^{\frac{j}{400}} \Lambda r_k \leq d_{g(\bar t)}(p_k, \bar x) \leq (400n^3KL)^{-j} \delta_k^{-1} r_k$. Note that since $\lambda_1(\bar x, \bar t) < \frac{1}{2} \theta R(\bar x, \bar t)$, the point $(\bar x, \bar t)$ lies at the center of an $\varepsilon_1$-neck. 

By Corollary \ref{soliton_approximation}, 
\[
r_{\mathrm{neck}}(\bar x, \bar t)^2\leq 4K(n-1)(n-2) r_k d_{g(\bar t)}(p_k, \bar x) \leq 4K n^2 r_k d_{g(\bar t)}(p_k, \bar x). 
\]
Therefore,
\begin{align*}
\bar t - L \, r_{\mathrm{neck}}(\bar x, \bar t)^2 &\geq \bar t - 4KL n^2 r_k d_{g(\bar t)}(p_k, \bar x) \\
& \geq \bar t - 4KL n^2 (400n^3KL)^{-j} \delta_k^{-1} r_k^2 \\
& \geq \bar t - 2^{-j} \delta_k^{-1} r_k^2 \\
& \geq t_k - 2^{-j +1} \delta_k^{-1} r_k^2. 
\end{align*}
On the other hand, $r_{\mathrm{neck}}(\bar x, \bar t)^2 \leq 4K n^2 r_k d_{g(\bar t)}(p_k, \bar x) \leq \frac{4K n^2}{\Lambda} d_{g(\bar t)}(p_k, \bar x)^2$. Since $L \sqrt{\frac{4n^3K}{\Lambda}} \leq 10^{-6}$, we obtain 
\[
r_{\mathrm{neck}}(\bar x, \bar t) \leq 10^{-6} L^{-1} d_{g(\bar t)}(p_k, \bar x). 
\]
Consequently, if $x \in B_{g(\bar t)}(\bar x, L \, r_{\mathrm{neck}}(\bar x, \bar t))$, then 
\begin{align*}
d_{g(\bar t)}(p_k, x) &\geq d_{g(\bar t)}(p_k, \bar x) - L r_{\mathrm{neck}}(\bar x, \bar t) \\
& \geq (1 - 10^{-6}) d_{g(\bar t)}(p_k, \bar x)\\
& \geq (1 - 10^{-6}) 2^{\frac{j}{400}} \Lambda r_k \\
& \geq 2^{\frac{j-1}{400}} \Lambda r_k. 
\end{align*}
Now $r_k \leq r_{\max}(\bar t) \leq r_{\mathrm{neck}}(\bar x, \bar t)$. Putting this together with $r_{\mathrm{neck}}(\bar x, \bar t)^2 \leq 4K n^2 r_k d_{g(\bar t)}(p_k, \bar x)$, for any $x \in B_{g(\bar t)}(\bar x, L \, r_{\mathrm{neck}}(\bar x, \bar t))$ we obtain
\begin{align*}
d_{g(\bar t)}(p_k, x) + 2nL  \, r_{\mathrm{neck}}(\bar x, \bar t)^2r_k^{-1} &\leq d_{g(\bar t)}(p_k, \bar x) + L\, r_{\mathrm{neck}}(\bar x, \bar t) + 2n L \,r_{\mathrm{neck}}(\bar x, \bar t)^2 r_k^{-1} \\
& \leq d_{g(\bar t)}(p_k, \bar x) + (2n + 1) L\, r_{\mathrm{neck}}(\bar x, \bar t)^2 r_k^{-1}  \\
& \leq 400n^3K L \, d_{g(\bar t)}(p_k, \bar x) \\
& \leq (400n^3K L)^{-j +1} \delta_k^{-1} r_k. 
\end{align*}
Since by Lemma 5.13
\[
d_{g(\bar t)}(p_k, x) \leq d_{g(t)}(p_k, x) \leq d_{g(\bar t)}(p_k, x) + 2n L \, r_{\mathrm{neck}}(\bar x, \bar t)^2 r_k^{-1}
\]
we conclude 
\[
2^{\frac{j-1}{400}} \Lambda r_k \leq d_{g(\bar t)}(p_k, x) \leq (400n^3K L)^{-j +1} \delta_k^{-1} r_k
\]
for all $(x, t) \in B_{g(\bar t)}(\bar x, L \, r_{\mathrm{neck}}(\bar x, \bar t)) \times [\bar t - L \,  r_{\mathrm{neck}}(\bar x, \bar t)^2, \bar t]$. It follows by the induction hypothesis and the Neck Improvement Theorem that the point $(\bar x, \bar t)$ is $2^{-j -1}\varepsilon_k$-symmetric. 
\end{proof}

Next we state existence and uniqueness lemmas, analogous to Lemma \ref{big_existence} and Lemma \ref{big_uniqueness} used in the proof of the Neck Improvement Theorem. 

\begin{lemma}\label{big_existence_2}
If $j$ is sufficiently large and $k$ is sufficiently large depending upon $j$, then the following holds. Given any $\bar t \in [t_k - 2^{\frac{j}{100}} r_k^2, t_k]$, there exist a collection of time-independent vector fields $\mathcal U = \{ U^{(a)} : 1 \leq a \leq {n \choose 2} \}$ on $B_{g(\bar t)}(p_k, 2^{\frac{j}{400}} \Lambda r_k)$ with the following properties: 
\begin{enumerate}
\item[$\bullet$] In $B_{g(\bar t)}(p_k, 2^{\frac{j}{400} }\Lambda r_k) \times [\bar t - r_k^2, \bar t]$, we have the estimate
\[
|\mathcal L_{U^{(a)}}(g(t))| + r_k |D(\mathcal L_{U^{(a)}}(g(t)))| \leq C(r_k^{-1} d_{g(t)}(p_k, x) + 1)^{-100} \varepsilon_k.
\] 
\item[$\bullet$] If $t \in [\bar t - r_k^2, \bar t]$ and $\nu$ denotes the unit normal to Hamilton's CMC foliation of $(M, g(t))$, then in $B_{g(\bar t)}(p_k, 2^{\frac{j}{400}} \Lambda r_k) \setminus B_{g(\bar t)}(p_k, 2\Lambda r_k)$, we have the estimate 
\[
r_k^{-1} |\langle U^{(a)}, \nu \rangle | \leq C (r_k^{-1} d_{g(t)}(p_k, x) + 1)^{-100} \varepsilon_k.
\]
\item[$\bullet$] If $t \in [\bar t - r_k^2, \bar t]$ and $\Sigma$ denotes the leaf of Hamilton's CMC foliation passing through $x \in B_{g(\bar t)}(p_k, 2^{\frac{j}{400}} \Lambda r_k) \setminus B_{g(\bar t)}(p_k, 2\Lambda r_k)$ at time $t$, then 
\[
\sum_{a, b = 1}^{{n \choose 2}} \bigg| \delta_{ab} - \mathrm{area}_{g(t)}(\Sigma)^{-\frac{n+1}{n-1}} \int_\Sigma \langle U^{(a)}, U^{(b)} \rangle_{g(t)} \, d\mu_{g(t)} \bigg| \leq C(r_k^{-1} d_{g(t)}(p_k, x) + 1)^{-100} \varepsilon_k. 
\]
\end{enumerate}
Moreover, on the ball $B_{g(\bar t)}(p_k, 2^{\frac{j}{400}}\Lambda r_k)$, the vector fields $U^{(a)}$ are close in the $C^{2}$-norm to a standard family of rotational vector fields on the Bryant soliton.
\end{lemma}
\begin{proof}
The proof is analogous to the proof of Lemma 9.17 in \cite{Bre20}. 
\end{proof}

\begin{lemma}\label{big_uniqueness_2}
If $j$ is sufficiently large and $k$ is sufficiently large depending upon $j$, then the following holds. Consider a time $\bar t \in [t_k - 2^{\frac{j}{100}} r_k^2, t_k]$. Suppose that $\mathcal U = \{ U^{(a)} : 1 \leq a \leq {n \choose 2} \}$ is a family of time-independent vector fields defined on $B_{g(\bar t)}(p_k, 2^{\frac{j}{400}} \Lambda r_k)$ with the following properties: 
\begin{enumerate}
\item[$\bullet$] In $B_{g(\bar t)}(p_k, 2^{\frac{j}{400} }\Lambda r_k)$, we have the estimate 
\[
|\mathcal L_{U^{(a)}}(g(t))| + r_k |D(\mathcal L_{U^{(a)}}(g(t)))| \leq C(r_k^{-1} d_{g(\bar t)}(p_k, x) + 1)^{-100} \varepsilon_k.
\]
\item[$\bullet$] If $\nu$ denotes the unit normal to Hamilton's CMC foliation of $(M, g(\bar t))$, then in $B_{g(\bar t)}(p_k, 2^{\frac{j}{400}} \Lambda r_k) \setminus B_{g(\bar t)}(p_k, 4\Lambda r_k)$, we have the estimate 
\[
r_k^{-1} |\langle U^{(a)}, \nu \rangle | \leq C (r_k^{-1} d_{g(\bar t)}(p_k, x) + 1)^{-100} \varepsilon_k.
\]
\item[$\bullet$] If $\Sigma$ denotes the leaf of Hamilton's CMC foliation passing through $x \in B_{g(\bar t)}(p_k, 2^{\frac{j}{400}} \Lambda r_k) \setminus B_{g(\bar t)}(p_k, 4\Lambda r_k)$ at time $\bar t$, then 
\[
\sum_{a, b = 1}^{{n \choose 2}} \bigg| \delta_{ab} - \mathrm{area}_{g(\bar t)}(\Sigma)^{-\frac{n+1}{n-1}} \int_\Sigma \langle U^{(a)}, U^{(b)} \rangle_{g(\bar t)} \, d\mu_{g(\bar t)} \bigg| \leq C(r_k^{-1} d_{g(\bar t)}(p_k, x) + 1)^{-100} \varepsilon_k. 
\] 
\end{enumerate}
Moreover, suppose that $\tilde {\mathcal U }= \{\tilde U^{(a)} : 1 \leq a \leq {n \choose 2} \}$ is a second family of time-independent vector fields defined on $B_{g(\bar t)}(p_k, 2^{\frac{j}{400}} \Lambda r_k)$ satisfying the same three properties above (with $U^{(a)}$ replaced by $\tilde U^{(a)}$). 
Then there exists an ${n \choose 2} \times {n \choose 2}$ matrix $\omega \in O({n \choose 2})$ such that 
\[
r_k^{-1} \sum_{a = 1}^{{n \choose 2}} \bigg| \sum_{b =1}^{{n \choose 2}} \omega_{ab} U^{(b)} - \tilde U^{(a)}\bigg| \leq C(r_k^{-1} d_{g(\bar t)}(p_k, x) + 1)^{-20} \varepsilon_k
\]
on $B_{g(\bar t)}(p_k, 2^{\frac{j-1}{400}} \Lambda r_k)$. 
\end{lemma}
\begin{proof}
The proof is analogous to the proof of Lemma 9.18 in \cite{Bre20}. 
\end{proof}

In the following, we define 
\[
\Omega^{(j,k)} := \{(x, t) \in B_{g(t_k)}(p_{k}, \delta_k^{-1} r_k) \times [t_k - 2^{\frac{j}{400}} r_k^2, t_k] : f(x, t) \leq 2^{\frac{j}{400}}\}, 
\]
where $f : B_{g(t_k)}(p_{k}, \delta_k^{-1} r_k) \times [t_k - \delta_k^{-1} r_k^2, t_k]\to [0, \infty)$ is the function in Corollary \ref{soliton_approximation}. The following three results are the main results of this section. Their proofs resemble various steps in the proof of the Neck Improvement Theorem. 

\begin{proposition}\label{main_prop}
Let $j$ be a large positive integer. If $k$ is sufficiently large (depending upon $j$), then we can find a family of time-independent vector fields $\mathcal W := \{W^{(a)} : 1 \leq a \leq {n \choose 2}\}$ defined on $B_{g(t_k)}(p_{k}, 4 \Lambda r_k)$  with the property that
\[ 
\sum_{l = 1}^{40} r_k^l |D^l(\mathcal L_{W^{(a)}}(g(t)))| \leq C 2^{-\frac{j}{400}}\varepsilon_k.
\]
in $B_{g(t_k)}(p_{k}, 4 \Lambda r_k) \times [t_k - 1000n^3K\Lambda r_k^2, t_k]$. Here, the constant $C$ is independent of $j$ and $k$. Finally, on the set $B_{g(t_k)}(p_{k}, 4 \Lambda r_k) \times [t_k - 1000n^3K \Lambda r_k^2, t_k]$, the family of vector fields $\mathcal W$ is close to a standard family of rotational vector fields on the Bryant soliton in the $C^{80}$-norm 
\end{proposition}

\begin{proof}
We will assume throughout this proof that $j$ is large, and $k$ is sufficiently large depending upon $j$. In this way, the region $\Omega^{(j, k)}$ is, after rescaling by $r_k$, as close as we like to the corresponding subset of the Bryant soliton. By Corollary \ref{soliton_approximation}, the function $f : \Omega^{(j,k)} \to [0 ,\infty)$ satisfies 
\[
R + |\nabla f|^2 \leq (1 + 4\delta_k)r_k^{-2} \leq 2 r_k^{-2}
\]
and 
\[
(\frac{\partial}{\partial t} - \Delta) f \leq - (1 - 2 \delta_k) r_k^{-2} \leq -\frac{1}{2} r_k^{-2}. 
\]
Moreover, 
\[
\frac{1}{2K}(r_k^{-1} d_{g(t)}(p_k, x) + 1) \leq f(x, t) + 1 \leq 2K (r_k^{-1} d_{g(t)}(p_k, x) + 1)
\]
for a constant $K$ depending only upon $n$. In particular, if $f(x, \bar {t}) \leq 2^{\frac{j}{400}}$, then $d_{g({\bar t})}(p_{k}, x) < 2^{\frac{j}{400}}\Lambda r_k$. 

\textit{Step 1:} Using Lemma \ref{big_existence_2} and Lemma \ref{big_uniqueness_2}, we can construct a family of time-dependent vector fields of vector fields $\mathcal U := \{ U^{(a)} : 1 \leq a \leq {n \choose 2}\}$ defined on $\Omega^{(j, k)}$, with the following two properties:
\begin{enumerate}
\item[$\bullet$] $r_k |\frac{\partial}{\partial t} U^{(a)}| \leq C(f + 100)^{-10}\varepsilon_k$ on $\Omega^{(j,k)}$.
\item[$\bullet$] $|\mathcal L_{U^{(a)}}(g)| + r_k |D( \mathcal L_{U^{(a)}}(g))| \leq C(f+100)^{-100}\varepsilon_k$ on $\Omega^{(j,k)}$. 
\end{enumerate}
Here $C$ is a positive constant that is independent of $j$ and $k$. Moreover, we can find a standard family of rotational vector fields $\mathcal U_{\mathrm{Bry}} := \{ U^{(a)}_{\mathrm{Bry}} : 1 \leq a \leq {n \choose 2}\}$ on the Bryant soliton, such that the family $\mathcal U$ is close to this standard family on $\Omega^{(j,k)}$ in $C^2$-norm. In other words, as $k$ tends to infinity, the family $\mathcal U$ converges to the family $\mathcal U_{\mathrm{Bry}}$ of rotational vector fields on the Bryant soliton in $C^2$. Note that $r_k |\Delta U^{(a)} + \mathrm{Ric}(U^{(a)})| \leq C r_k |D( \mathcal L_{U^{(a)}}(g))| \leq C (f+100)^{-100}\varepsilon_k$ on $\Omega^{(j,k)}$. \\

\textit{Step 2:}  Let $V^{(a)}$ be the solution of the PDE $\frac{\partial }{\partial t} V^{(a)} = \Delta V^{(a)} + \mathrm{Ric}(V^{(a)})$ on $\Omega^{(j,k)}$ with the Dirichlet boundary condition $V^{(a)} = U^{(a)}$ on the parabolic boundary of $\Omega^{(j,k)}$. Using the estimate 
\[
r_k \big|\frac{\partial}{\partial t} U^{(a)} - \Delta U^{(a)} - \mathrm{Ric}(U^{(a)})\big| \leq C(f + 100)^{-10}\varepsilon_k,
\]
we obtain
\[
r_k \big|\frac{\partial}{\partial t} (V^{(a)} - U^{(a)}) - \Delta (V^{(a)} -U^{(a)}) - \mathrm{Ric}(V^{(a)} - U^{(a)})\big| \leq C(f + 100)^{-10}\varepsilon_k
\]
in $\Omega^{(j,k)}$, where $C$ is a positive constant that is independent of $j$ and $k$. Proposition \ref{norm_pde} gives 
\[
r_k \big(\frac{\partial}{\partial t}- \Delta\big)|V^{(a)} - U^{(a)}| \leq C(f + 100)^{-10}\varepsilon_k
\]
in $\Omega^{(j,k)}$, where $C$ is a positive constant that is independent of $j$ and $k$. By a straightforward calculation, using the estimates $(\frac{\partial}{\partial t} - \Delta) f \leq -\frac{1}{2} r_k^{-2}$, $|\nabla f|^2 \leq 2 r_k^{-2}$, we obtain
\[
\big(\frac{\partial}{\partial t}- \Delta\big)(f + 100)^{-8} \geq (f + 100)^{-9} r_k^{-2}
\]
in $\Omega^{(j,k)}$. Thus, by the maximum principle, we obtain
\[
r_k^{-1}|V^{(a)} - U^{(a)}| \leq C(f + 100)^{-8}\varepsilon_k
\]
in $\Omega^{(j,k)}$, where $C$ is a positive constant that is independent of $j$ and $k$. The distance between $\Omega^{(j-1, k)}$ and the parabolic boundary $\mathcal P\Omega^{(j, k)}$ is bounded below by $c\,r_k$ for a universal constant $c$, so by standard interior estimates for linear parabolic equations, we obtain 
\[
|D(V^{(a)} - U^{(a)})| \leq C(f + 100)^{-8}\varepsilon_k
\]
in $\Omega^{(j-1,k)}$, where $C$ is a positive constant that is independent of $j$ and $k$. In particular, on the set $\Omega^{(j-1, k)}$, the vector field $V^{(a)}$ is close in $C^1$-norm to the vector field $U_{\mathrm{Bry}}^{(a)}$. Consequently, on the set $B_{g(t_k)}(p_{k}, 8 \Lambda r_k) \times [t_k - 2000 n^3K \Lambda r_k^2, t_k]$, the family of vector fields $\mathcal V := \{V^{(a)} : 1 \leq a \leq {n \choose 2}\}$ is close to the family $\mathcal U_{\mathrm{Bry}}$ in $C^{100}$-norm. \\

\textit{Step 3:} We define $h^{(a)}(t) := \mathcal L_{V^{(a)}(t)}(g(t))$. By Theorem \ref{pde_for_lie_derivative}, we have 
\[
\frac{\partial}{\partial t}h^{(a)}(t) =  \Delta_{L, g(t)} h^{(a)}(t). 
\]
The estimate for $D(V^{(a)} - U^{(a)})$ in Step 2 implies 
\[
|\mathcal L_{V^{(a)}}(g) | \leq |\mathcal L_{U^{(a)}}(g)| + C|D(V^{(a)} - U^{(a)})| \leq C(f + 100)^{-8}\varepsilon_k
\]
in $\Omega^{(j-1,k)}$, where $C$ is a positive constant that is independent of $j$ and $k$. 

Since our solution is close to the Bryant soliton, we have the estimate $r_k^2\mathrm{Ric} > 2c f^{-2} g$ in $\Omega^{(j, k)}$ for some small universal constant $c$ if $k$ is sufficiently large depending upon $j$. This is because on the Bryant soliton the smallest eigenvalue of the Ricci tensor falls off like the square of the distance to the tip (see Chapter 4 of \cite{CLN06} for details). Hence, $\mathrm{Ric} \geq 2\rho r_k^{-2} g$ in $\Omega^{(j, k)}$ with $\rho := c \,2^{-\frac{j}{200}}$. Evidently, we also have $\mathrm{Ric} \leq C r_k^{-2} g$ in $\Omega^{(j, k)}$. Consider the weighted norm 
\[
\bar \lambda^{(a)} := \inf \{ \lambda > 0 : -\lambda(r_k^2\mathrm{Ric} - \rho g) \leq h^{(a)} \leq \lambda(r_k^2 \mathrm{Ric} - \rho g)\}. 
\]
Suppose $\lambda_1 \leq \cdots \leq \lambda_n$ denote the eigenvalues of $h^{(a)}$ corresponding to orthonormal eigenvectors $e_1, \dots, e_n$ with respect to $g(t)$ in $\Omega^{(j, k)}$. So $|h^{(a)}|^2 = \sum_{i =1}^n \lambda_i^2$. Then the inequalities
\[
|\lambda_i| \leq \bar \lambda^{(a)} (r_k^2\mathrm{Ric}(e_i, e_i) - \rho) \leq C \bar \lambda^{(a)}
\]
and 
\[
 -\rho^{-1}|h^{(a)}|(r_k^2 \mathrm{Ric} - \rho g) \leq -|h^{(a)}| g \leq h^{(a)} \leq |h^{(a)}| g \leq \rho^{-1}|h^{(a)}|(r_k^2 \mathrm{Ric} - \rho g)
\]
imply that $\rho \bar \lambda^{(a)} \leq |h^{(a)}| \leq C \bar \lambda^{(a)}$ in $\Omega^{(j,k)}$.  In other words, the weighted norm is comparable to the usual norm in $\Omega^{(j,k)}$. Now we apply Lemma \ref{Anderson_Chow_replacement} to $r_k^{-2}h^{(a)}(t)$ in the region $\Omega^{(j-1,k)}$ taking $r_k^{-2} g$ for our metric and $r_k^{-2} \rho = cr_k^{-2} 2^{-\frac{j}{200}}$ as our lower bound for the Ricci tensor. Thus the function
\[
\psi^{(a)} := \exp(-2\rho r_k^{-2}(t_k - t)) \; \bar \lambda^{(a)}
\]
satisfies
\[
\sup_{\Omega^{(j-1, k)}} \psi^{(a)}\leq \sup_{\mathcal P \Omega^{(j-1,k)} }\psi^{(a)}, 
\]
where $\mathcal P \Omega^{(j-1,k)}$ denotes the parabolic boundary of $\Omega^{(j-1,k)}$. The parabolic boundary of $\Omega^{(j-1,k)}$ consists of points $(x, t)$ where $t = t_k - 2^{\frac{j-1}{100}}r_k^2$ and points $(x, t)$ where $f(x, t) = 2^{\frac{j-1}{400}}$. If $t = t_k - 2^{\frac{j-1}{100}}r_k^2$, then using that $\bar \lambda^{(a)} \leq \rho^{-1} |h^{(a)}| \leq C 2^{\frac{j}{200}} |h^{(a)}|$ and $|h^{(a)}| \leq C(f + 100)^{-8} \varepsilon_k \leq C \varepsilon_k$, we obtain 
\[
\psi^{(a)}(x, t) \leq C \exp(-2^{\frac{j}{200}}) 2^{\frac{j}{200}} \varepsilon_k \leq C 2^{-\frac{j}{400}} \varepsilon_k.
\]
On the other hand, if $(x, t)$ is a point where $f(x,t) = 2^{\frac{j-1}{400}}$, then using the bound $\exp(-2\rho r_k^{-2}(t_k - t)) \leq 1$, we have 
\[
\psi^{(a)}(x, t) \leq C 2^{\frac{j}{200}} (2^{\frac{j-1}{400}} + 100)^{-8} \varepsilon_k \leq C2^{-\frac{j}{400}} \varepsilon_k. 
\]
Finally, we observe that $\psi^{(a)} \geq C^{-1} |h^{(a)}|$ in the region $B_{g(t_k)}(p_k, 8 \Lambda r_k) \times [t_k - 2000n^3 K \Lambda r_k^2, t_k]$. Putting all of these estimates together, we conclude
\[
|h^{(a)}| \leq C2^{-\frac{j}{400}}\varepsilon_k
\]
in the region $B_{g(t_k)}(p_k, 8 \Lambda r_k) \times [t_k - 2000n^3 K \Lambda r_k^2, t_k]$ assuming $j$ is large and $k$ is sufficiently large depending upon $j$. Here $C$ is a constant that does not depend on $j$ or $k$. By standard interior estimates for linear parabolic equations, we obtain 
\[
\sum_{l = 0}^{100} r_k^{l}\big|D^lh^{(a)}\big| \leq C 2^{-\frac{j}{400}}\varepsilon_k
\]
in the region $B_{g(t_k)}(p_k, 6 \Lambda r_k) \times [t_k - 1000n^3 K \Lambda r_k^2, t_k]$. As we have seen before, since 
\[
\frac{\partial}{\partial t} V^{(a)} = \Delta V^{(a)} + \mathrm{Ric}(V^{(a)}) = \mathrm{div} h^{(a)} - \frac{1}{2} \nabla (\mathrm{tr} h^{(a)})
\]
this implies 
\[
\sum_{l = 0}^{80} r_k^{l+1}\big|D^l(\frac{\partial}{\partial t} V^{(a)})\big| \leq C 2^{-\frac{j}{400}}\varepsilon_k,
\]
in the region $B_{g(t_k)}(p_k, 6 \Lambda r_k) \times [t_k - 1000n^3 K\Lambda r_k^2, t_k]$.\\

\textit{Step 4:} Finally, we define a family of time-independent vector fields $\mathcal W := \{W^{(a)} : 1 \leq a \leq {n \choose 2}\}$ by setting $W^{(a)} = V^{(a)}$ at time $t_k$. Because the family $\mathcal U_{\mathrm{Bry}}$ is time-independent, the family $\mathcal W$ is close to $\mathcal U_{\mathrm{Bry}}$ in the region $B_{g(t_k)}(p_k, 6 \Lambda r_k) \times [t_k - 1000n^3 K \Lambda  r_k^2, t_k]$. Integrating the time derivative estimates in Step 3, we obtain
\[
\sum_{l = 0}^{40} r_k^{l-1}\big|D^l(W^{(a)} - V^{(a)})\big| \leq C 2^{-\frac{j}{400}}\varepsilon_k. 
\]
Hence together with the estimates for $h^{(a)}$, we get
\[
\sum_{l = 0}^{40} r_k^{l}\big|D^l(\mathcal L_{W^{(a)}}(g))\big| \leq C 2^{-\frac{j}{400}}\varepsilon_k
\]
in the region $B_{g(t_k)}(p_k, 6 \Lambda r_k) \times [t_k - 1000n^3 K \Lambda  r_k^2, t_k]$. This completes the proof.  
\end{proof}

For each $k$ sufficiently large, let us choose a compact domain $D_k \subset M$ with the following three properties:
\begin{enumerate}
\item[$\bullet$] There exists a point $\bar x \in \partial D_k$ such that $\lambda_1(\bar x, t_k) = \frac{2}{3} \theta R(\bar x, t_k)$. 
\item[$\bullet$] For each $x \in D_k$, we have $\lambda_1(x, t_k) \geq \frac{2}{3} \theta R(x, t_k)$. 
\item[$\bullet$] $\partial D_k$ is a leaf of Hamilton's CMC foliation of $(M, g(t_k))$. 
\end{enumerate}
Note that $D_k \subset \{x \in M : \lambda_1(x, t_k) > \frac{1}{2} \theta R(x, t_k)\} \subset B_{g(t_k)}(p_{k}, \Lambda r_k)$ in view of how we have chosen $\Lambda$. In particular, if $\bar x \in M \setminus D_k$, then the point $(\bar x, t_k)$ lies at the center of an evolving $\varepsilon_1$-neck by Proposition \ref{neck_detection}. Recall that for such a point, we define the curvature scale $r_{\mathrm{neck}}(\bar x, t_k)$ by the identity $R(\bar x, t_k) = (n-1)(n-2) r_{\mathrm{neck}}(\bar x, t_k)^{-2}$. 

\begin{proposition}\label{main_prop_2}
Let $j$ be a large positive integer. If $k$ is sufficiently large (depending upon $j$), then the family of vector fields $\mathcal W$ constructed in Proposition \ref{main_prop} have the following property. If $\bar x \in B_{g(t_k)}(p_k, \Lambda r_k) \setminus D_k$, then 
\[
\sum_{l = 1}^{10} r_k^{l-1} |D^l(\langle W^{(a)}, \nu \rangle)| \leq C 2^{-\frac{j}{400}}\varepsilon_k
\]
on the parabolic neighborhood $B_{g(t_k)}(\bar x, 600 r_{\mathrm{neck}}(\bar x, t_k)) \times [t_k - 200 r_{\mathrm{neck}}(\bar x, t_k)^2, t_k]$. Here,  $\nu$ denotes the unit normal to the leaf of Hamilton's CMC foliation and $C$ is a constant independent of $j$ and $k$. 
\end{proposition}
\begin{proof}
Consider a point $\bar x \in B_{g(t_k)}(p_k, \Lambda r_k) \setminus D_k$. By Corollary \ref{soliton_approximation}, $r_k^2R(\bar x, t_k) \geq \frac{1}{4K \Lambda}$. Recalling that by assumption $\sqrt{\frac{4n^3K}{\Lambda}} \leq 10^{-6}$, this implies 
\[
r_{\mathrm{neck}}(\bar x, t_k)^2 = (n-1)(n-2) R(\bar x, t_k)^{-1} \leq 4n^3 K \Lambda r_k^2 \leq 10^{-12} \Lambda^2 r_k^2. 
\]
Thus 
\[
B_{g(t_k)}(\bar x, 1000 r_{\mathrm{neck}}(\bar x, t_k)) \times [t_k - 200 r_{\mathrm{neck}}(\bar x, t_k)^2, t_k] \subset B_{g(t_k)}(p_{k}, 4 \Lambda r_k) \times [t_k - 1000n^3K\Lambda r_k^2, t_k], 
\]
and in particular, the estimates satisfied by the family of vector fields $\mathcal W$ in Proposition \ref{main_prop} hold on the parabolic neighborhood $B_{g(t_k)}(\bar x, 1000 r_{\mathrm{neck}}(\bar x, t_k)) \times [t_k - 200 r_{\mathrm{neck}}(\bar x, t_k)^2, t_k]$. 

Now we estimate the quantity $\langle W^{(a)}, \nu \rangle$, where $\nu$ denotes the normal to the CMC foliation. To that end, fix a time $t \in  [t_k - 200 r_{\mathrm{neck}}(\bar x, t_k)^2, t_k]$ and let $\Sigma_s$ denote Hamilton's CMC foliation at time $t$. We will only consider those leaves of the foliaiton which are contained in the $B_{g(t_k)}(\bar x, 800 r_{\mathrm{neck}}(\bar x, t_k))$. As we did before, let us define a function $F^{(a)} : \Sigma_s \to \mathbb{R}$ by $F^{(a)} := \langle W^{(a)}, \nu \rangle$. Recall that the quantity 
\[
\Delta_{\Sigma} F^{(a)} + (|A|^2 + \mathrm{Ric}(\nu, \nu)) F^{(a)} =: H^{(a)}
\]
can be expressed in terms of $\mathcal L_{W^{(a)}}(g)$ and first derivatives of $\mathcal L_{W^{(a)}}(g)$. Consequently, the estimates of the previous proposition imply 
\[
\sum_{l =0}^{20} r_k^{l +1} \big|D^lH^{(a)}| \leq C 2^{-\frac{j}{400}}\varepsilon_k.
\]
Now as observed in \cite{Bre20}, since $\Sigma_s \subset B_{g(t_k)}(\bar x, 800 r_{\mathrm{neck}}(\bar x, t_k)) \subset  B_{g(t_k)}(p_k, 4 \Lambda r_k)$ the eigenvalues of the Jacobi operator $\Delta_{\Sigma} + (|A|^2 + \mathrm{Ric}(\nu, \nu))$ lie outside the interval $[-cr_k^{-2}, cr_k^{-2}]$, for some small universal constant $c$. This follows from approximation by the exact Bryant solution, where this estimate can be verified. Consequently, we have a universal bound for the norm of the inverse of the Jacobi operator. As in the proof of the Neck Improvement Theorem, using the estimates for $H^{(a)}$, we thus get
\[
\sum_{l =0}^{10} r_k^{l-1} |D^l F^{(a)}| \leq C 2^{-\frac{j}{400}}\varepsilon_k.
\]
Since $t \in [t_k  - 200 r_{\mathrm{neck}}(\bar x, t_k)^2, t_k]$ is arbitrary, we conclude
\[
\sum_{l =0}^{10} r_k^{l-1} |D^l (\langle W^{(a)}, \nu \rangle) | \leq C 2^{-\frac{j}{400}}\varepsilon_k.
\]
on the parabolic neighborhood $B_{g(t_k)}(\bar x, 600 r_{\mathrm{neck}}(\bar x, t_k)) \times [t_k - 200 r_{\mathrm{neck}}(\bar x, t_k)^2, t_k]$ for a constant $C$ independent of $j$ and $k$. 
\end{proof}

\begin{proposition} \label{main_prop_3}
Let $j$ be a large integer. If $k$ is sufficiently large (depending upon $j$), then the family of vector field $\mathcal W := \{ W^{(a)} : 1 \leq a \leq {n \choose 2}\}$ constructed in Proposition \ref{main_prop} have the following property. For each such $\bar x \in B_{g(t_k)}(p_{k}, \Lambda r_k) \setminus D_k$, we can find a symmetric ${n\choose 2} \times {n \choose 2}$ matrix $Q_{ab}$ such that the estimate
\[
\Big| Q_{ab} - \mathrm{area}_{g(t)}(\Sigma)^{-\frac{n+1}{n-1}} \int_{\Sigma} \langle W^{(a)}, W^{(b)} \rangle_{g(t)} \, d\mu_{g(t)} \Big| \leq C 2^{-\frac{j}{400} }\varepsilon_k. 
\]
holds whenever $t \in [t_k - 200 r_{\mathrm{neck}}(\bar x, t_k)^2, t_k]$ and $\Sigma \subset B_{g(t_k)}(\bar x, 200 r_{\mathrm{neck}}(\bar x, t_k))$ is a leaf of Hamilton's CMC foliation of $(M, g(t))$. The matrix $Q_{ab}$ is independent of $t$ and $\Sigma$. Moreover, the eigenvalues of the matrix $Q_{ab}$ lie in the interval $[\frac{1}{C}, C]$. Here the constant $C$ is independent of $j$ and $k$. 
\end{proposition}
\begin{proof}
The estimates established in Proposition \ref{main_prop} and Proposition \ref{main_prop_2} hold on the parabolic neighborhood $B_{g(t_k)}(\bar x, 600 r_{\mathrm{neck}}(\bar x, t_k)) \times [t_k - 200 r_{\mathrm{neck}}(\bar x, t_k)^2, t_k]$. We can prove the proposition in a manner analogous to Steps 11 - 14 of the Neck Improvement Theorem. In this argument, we use that the family of vector fields $\mathcal W$ is close to a standard family of rotational vector fields on the Bryant soliton. In particular, this condition ensures that the eigenvalues of the matrix $Q_{ab}$ are universally bounded from above and below. 
\end{proof}

\begin{corollary}\label{final_cor}
If $k$ is sufficiently large, then $(\bar x, t_k)$ is $\frac{\varepsilon_k}{2}$-symmetric for all $\bar x \in M \setminus D_k$.  
\end{corollary}
\begin{proof}
We have already shown in Lemma \ref{symmetry_outside_lambda} that if $\bar x \in M \setminus B_{g(t_k)}(p_{k}, \Lambda r_k)$, then $(\bar x, t_k)$ is $\frac{\varepsilon_k}{2}$-symmetric. On the other hand, if $\bar x \in B_{g(t_k)}(p_{k}, \Lambda r_k) \setminus D_k$, then Proposition \ref{main_prop}, Proposition \ref{main_prop_2}, and Proposition \ref{main_prop_3} imply the point $(\bar x, t_k)$ is $\frac{\varepsilon_k}{2}$-symmetric, provided that we first choose $j$ sufficiently large and then assume $k$ is sufficiently large depending upon $j$. 
\end{proof}

We can now complete the proof of the main result of this paper, Theorem \ref{main}. By Corollary \ref{final_cor}, we know that $(x, t_k)$ is $\frac{\varepsilon_k}{2}$-symmetric for all $x \in M \setminus D_k$. On the other hand, by Proposition \ref{main_prop}, Proposition \ref{main_prop_2}, and Proposition \ref{main_prop_3}, there are vector fields defined on the cap region $D_k$ satisfying the requirements of Definition \ref{cap_symmetry} with $\varepsilon = \frac{\varepsilon_k}{2}$, provided we take a fixed $j$ sufficiently large and then assume $k$ is sufficiently large depending upon this $j$. Thus $(M, g(t_k))$ is $\frac{\varepsilon_k}{2}$-symmetric if $k$ is sufficiently large. By Lemma 5.5, we can find a time $\tilde t_k > t_k$ such that the flow is $\varepsilon_k$-symmetric at time $t$ for all $t \in [t_k, \tilde t_k]$, assuming $k$ is sufficiently large. This contradicts the definition of $t_k$ and completes the proof of Theorem \ref{main}.


\appendix 

\section{Auxiliary Results for Ancient $\kappa$-Solutions}

\begin{theorem} \label{soliton_classification}
Suppose $n \geq 4$ and $(\bar M, \bar g, \bar f)$ is an $n$-dimensional complete gradient shrinking soliton with curvature tensor that is strictly PIC and weakly PIC2. Then $(\bar M, \bar g)$ is isometric to a quotient of either a round sphere $S^n$ or a cylinder $S^{n-1} \times \mathbb{R}$. 
\end{theorem}
\begin{proof}
For $n = 4$, this result was proven in \cite{LNW18}. Suppose $n \geq 5$. By pulling back $g$ and $f$ to the universal cover, we may assume $(\bar M, \bar g)$ is simply connected and we must show that either $M = S^n$ or $M = S^{n-1} \times \mathbb{R}$. 

If $(\bar M, \bar g)$ is strictly PIC2, then it also has positive sectional curvature. By the result of Munteanu and Wang \cite{MW17}, this implies $(\bar M, \bar g)$ is compact. Then the main result of \cite{Bre08} implies $(\bar M, \bar g)$ must have constant curvature. It follows that $M = S^n$. 

If $(\bar M, \bar g)$ is only weakly PIC2, but not strictly PIC2, then Proposition 6.6 implies $(\bar M, \bar g) = (X, g_X) \times \mathbb{R}$. In this case, it is easy to see that the triple $(X, g_X, f_X)$ is a complete gradient shrinking soliton with $f_X = \bar f|_{X\times\{0\}}$. Moreover, because $(\bar M, \bar g)$ is strictly PIC and weakly PIC2, it follows (by definition) that $(X, g_X)$ is strictly PIC1 and weakly PIC2. Consequently, $(X, g_X)$ has positive Ricci curvature and nonnegative sectional curvature. Again by \cite{MW17}, $(X, g_X)$ must be compact and then by \cite{Bre08}, $(X, g_X)$ must have constant sectional curvature. It follows that $M = S^{n-1} \times \mathbb{R}$. This completes the proof. 
\end{proof}

The structure results for ancient $\kappa$-solutions in dimension four were developed in \cite{CZ06} under slightly different positivity assumptions (following Hamilton's work \cite{Ham97}) than uniformly PIC and weakly PIC2. In order for the subsequent structure results of this section to apply in dimension $n = 4$, we need the following proposition. 

\begin{proposition}
Suppose that $(M, g(t))$ is a four-dimensional ancient $\kappa$-solution (in the sense of Definition \ref{kappa_solution}) with uniformly positive isotropic curvature. Then $(M, g(t))$ satisfies the restricted isotropic curvature pinching condition of \cite{CZ06}. That is, in the notation of \cite{Ham97}, we have the inequalities $c_3 \leq \Lambda c_1$, $a_3 \leq \Lambda a_1$, and $b_3^2 \leq a_1c_1$, for some positive constant $\Lambda$. Consequently, $(M, g(t))$ is an ``ancient $\kappa$-solution with restricted isotropic curvature pinching'' in the sense of \cite{CZ06}. 
\end{proposition}
\begin{proof}
Since by assumption $(M, g(t))$ has bounded curvature, the proposition will follow from the pinching estimates for ancient solutions of the Ricci flow developed in \cite{BHS11}. Recall that in Section 2, we noted uniformly positive isotropic curvature implies there exists a positive constant $\alpha$ such that $\min\{a_1 + a_2, c_1 + c_2 \} \geq \alpha \max\{a_3, b_3, c_3 \}$. Moreover, because our solution is ancient and has bounded curvature, the evolution equations for $a_1$ and $c_1$ imply that $\min\{a_1, c_1 \} \geq 0$ by the maximum principle. 

For $\Phi > 0$, let $C_0$ denote the cone of curvature operators satisfying 
\begin{align*}
(b_2 + b_3)^2 &\leq \Phi(a_1 + a_2)(c_1 + c_2),\\
a_2 + a_3 &\leq (\Phi +1)(a_1 + a_2), \\
c_2 + c_3 &\leq (\Phi + 1)(c_1 + c_2).
\end{align*}
By Theorem 1.3 and Theorem 1.4 in \cite{Ham97}, the cone $C_0$ is transversally invariant under the Hamilton ODE $\frac{d}{dt} R = Q(R)$. Moreover, if $\Phi$ is sufficiently large, then the curvature operator of $(M, g(t))$ is contained in $C_0$ for all $t$. For for $s \geq 0$, consider the cone of curvature operators $C(s) = C_0 \cap \{s(a_1 + a_2 + a_3) \leq a_1\} \cap \{s(c_1 + c_2 + c_3) \leq c_1\}$. Now
\[
\frac{d}{dt} a_1 \geq a_1^2 + b_1^2 + 2a_2a_3 \geq 2a_2a_3 \geq a_3(a_1 + a_2) \geq \frac{1}{(\Phi + 1)}a_3(a_2 + a_3) \geq \frac{2}{9(\Phi + 1)}(a_1 + a_2 + a_3)^2. 
\]
On the other hand,
\begin{align*}
\frac{d}{dt} (a_1 + a_2 + a_3) &= (a_1 + a_2 + a_3)^2 + (b_1^2 + b_2^2 + b_3^2) \\
& \leq (a_1 + a_2 + a_3)^2 + 3 b_3^2  \\
& \leq (a_1 + a_2 + a_3)^2 + 3 \Phi (a_1 + a_2)(c_1 + c_2) \\
& \leq \frac{4\Phi + 3}{3}(a_1 + a_2 + a_3)^2,
\end{align*}
where we have used $a_1 + a_2 + a_3 = c_1 + c_2 + c_3$. Identical inequalities hold for the eigenvalues $c_i$. Thus, as long as $0 \leq s \leq s_0 := (3(\Phi +1)(4\Phi + 3))^{-1}$, then $C(s)$ is transversally invariant under the Hamilton ODE. Since $(M, g(t))$ is an ancient solution with bounded curvature and the curvature operator of $(M, g(t))$ is contained in $C(0)$ (at every point in spacetime), the pinching result of \cite{BHS11} (see Theorem 9), implies that the curvature operator of $(M, g(t))$ is contained in $C(s_0)$. In particular, $a_1 \geq s_0 a_3$ and $c_1 \geq s_0 c_3$. 

It remains to verify the condition $b_3^2 \leq a_1c_1$. We argue in a similar fashion. Let $\tilde C(s) = \{s b_3^2 \leq a_1c_1 \} \cap \{ (b_2 + b_3)^2 \leq q(s)(a_1 + a_2)(c_1 + c_2) \}$, where $q(s) = \frac{s+1}{2s}$ and $0 < s < 1$. Each cone in this family of cones is invariant under the Hamilton ODE. The cone $\{(b_2 + b_3)^2 \leq q(s) (a_1 + a_2)(c_1 + c_2) \}$ is transversally invariant under the Hamilton ODE as long as $q(s) \neq 1$. Indeed, $q(s) > 1$ when $0 < s < 1$. Thus our goal is to show transversal invariance of the  inequality $s b_3^2 \leq a_1c_1$ under the assumption $(b_2 + b_3)^2 \leq q(s)(a_1 + a_2)(c_1 + c_2)$. To that end, we recall that 
\begin{align*}
\frac{d}{dt} \ln(b_3^2) &\leq \frac{4b_1b_2}{b_3} + 2a_3 + 2c_3, \\
\frac{d}{dt} \ln(a_1) &\geq \frac{a_1^2 + b_1^2 +2 a_2a_3}{a_1}, \\
\frac{d}{dt} \ln(c_1) & \geq \frac{c_1^2 + b_1^2 +2 c_2c_3}{c_1}. 
\end{align*}
It follows then that 
\[
\frac{d}{dt} \ln\Big(\frac{b_3^2}{a_1c_1}\Big) \leq  - \frac{(a_1 - b_1)^2 + 2(a_2 - a_1)b_1}{a_1} - \frac{(c_1 - b_1)^2 + 2(c_2 - c_1)b_1}{c_1} - \frac{4b_1(b_3 - b_2)}{b_3}.
\]
The right hand side can only vanish if $a_1 = a_2 = c_1 = c_2 = b_1$ and $b_2 = b_3$. Suppose we have equality in $s b_3^2 = a_1c_1$ and the right hand side vanishes. Then $s b_3^2 = b_1^2$.  On the other hand, the inequality  $(b_2 + b_3)^2 \leq q(s)(a_1 + a_2)(c_1 + c_2)$ implies $ b_3^2 \leq q(s)b_1^2 = q(s)sb_3^2$. Since $q(s)s < 1$ when $0 < s < 1$, we have must $b_3 =0$, which in turns implies the curvature tensor must vanish. It follows that the cone $\tilde C(s)$ is transversally invariant under the Hamilton ODE for $0 < s < 1$. Finally, since the curvature tensor of $(M, g(t))$ satisfies 
\[
a_1c_1 \geq s_0^2 a_3c_3 \geq \frac{1}{4}s_0^2 (a_1 + a_2)(c_1 + c_2) \geq \frac{1}{4\Phi}s_0^2 b_3^2, 
\]
the curvature tensor of $(M, g(t))$ is contained in $\tilde C(s)$ is $s > 0$ is sufficiently small. Once more by \cite{BHS11}, we obtain the inequality $b_3^2 \leq a_1c_1$ and this completes the proof. 
\end{proof}

The classification of asymptotic solitons can be used to give a proof of the universal noncollapsing of ancient $\kappa$-solutions in higher dimensions (see \cite{Naf19}). This is the same manner in which Perelman proved universal noncollapsing in dimension three. For our purposes, we quote the noncollapsing results Theorem 6.19 in \cite{Bre19} for dimension $n \geq 5$ and Theorem 3.5 in \cite{CZ06} for dimension $n=4$. 

\begin{theorem} \label{universal_noncollapsing}
For every $n \geq 4$ and $\alpha > 0$, there exists a constant $\kappa_0 := \kappa_0(n, \alpha)$ with the following property. If $(M, g(t))$ is an $n$-dimensional, noncompact ancient $\kappa$-solution (for some $\kappa > 0$), which $\alpha$-uniformly PIC, and strictly PIC2, then $(M, g(t))$ is a $\kappa_0$-noncollapsed for all $t$. 
\end{theorem}

We have made use of the following extension \cite{Bre09} of Hamilton's Harnack inequality \cite{Ham93a} to PIC2 solutions of the Ricci flow. 

\begin{theorem}\label{harnack_inequality}
Assume that $(M, g(t))$, $t \in (-T, 0]$ is a solution to the Ricci flow that is complete with bounded curvature and is weakly PIC2. Then
\[
\frac{\partial}{\partial t} R + 2 \langle \nabla R, v \rangle + 2\mathrm{Ric}(v,v) + \frac{1}{t-T}R \geq 0
\]
for every tangent vector $v$. In particular, the product $(t-T)R(x, t)$ is nondecreasing in $t$ for each $x \in M$. Moreover, if the solution is ancient $(T = - \infty)$, then $R(x, t)$ is nondecreasing in $t$ for each $x \in M$. 
\end{theorem}

We have also used the following extension of Perelman's pointwise derivative estimates \cite{Per03} to higher dimensions. See Corollary 6.14 in \cite{Bre19} for $n \geq 5$ and Proposition 3.3 in \cite{CZ06} for $n = 4$. 

\begin{theorem}\label{pointwise_derivative_estimates}
For every $n \geq 4$ and $\alpha > 0$, there exists a constant $\eta :=  \eta(n, \alpha) < \infty$ with the following property. If $(M, g(t))$ is an $n$-dimensional, ancient $\kappa$-solution, which is $\alpha$-uniformly PIC. Then $|DR| \leq \eta R^{\frac{3}{2}}$ and $|\frac{\partial }{\partial t} R| \leq \eta R^2$ at each point in space-time.
\end{theorem}

The following two results are consequences of the first author's work on ancient $\kappa$-solutions for $n \geq 5$ and Chen and Zhu's work in dimension four. See Section 6 in \cite{Bre19} and in particular Theorem 6.18 for $n \geq 5$. See Section 3 in \cite{CZ06} and in particular Proposition 3.4 for $n = 4$.

\begin{theorem}\label{structure_thm}
Suppose $n \geq 4$ and let $(M, g(t))$ be an $n$-dimensional ancient $\kappa$-solution which is noncompact, $\alpha$-uniformly PIC, and strictly PIC2. Given any $\varepsilon > 0$, we can find a domain $\Omega_t \subset M$ with the following properties:
\begin{enumerate}
\item[$\bullet$] For each $x \in M \setminus \Omega_t$, the point $(x, t)$ lies at the center of an evolving $\varepsilon$-neck. 
\item[$\bullet$] The boundary $\partial \Omega_t$ is a leaf of Hamilton's CMC foliation at time $t$. 
\item[$\bullet$] $\sup_{x \in \Omega_t} R(x, t) \leq C(n,\alpha, \varepsilon) \inf_{x \in \Omega_t} R(x, t)$. 
\item[$\bullet$] $\mathrm{diam}_{g(t)}(\Omega_t)^2 \, \sup_{x \in \Omega_t} R(x, t) \leq C(n,\alpha, \varepsilon)$. 
\end{enumerate}
\end{theorem}

\begin{corollary}\label{structure_cor}
Suppose $n \geq 4$ and let $(M, g(t))$ be an $n$-dimensional ancient $\kappa$-solution which is noncompact, $\alpha$-uniformly PIC, and strictly PIC2. Let $\varepsilon$ be a small positive real number and let $\Omega_t \subset M$ be the domain in Theorem \ref{structure_thm}. Then $\sup_{x \in M} R(x,t) \leq C(n, \alpha, \varepsilon) \inf_{x \in \Omega_t} R(x, t)$ and $\mathrm{diam}_{g(t)}(\Omega_t)^2 \sup_{x \in M} R(x, t) \leq C(n, \alpha, \varepsilon)$. 
\end{corollary}
\begin{proof}
This follows exactly as Corollary A.3 follows from Theorem A.2 in \cite{Bre20}. 
\end{proof}

The following result is sometimes called the Neck Detection Lemma.

\begin{lemma}\label{neck_detection_lemma}
Suppose $n \geq 4$ and let $(M, g(t))$ be an $n$-dimensional ancient $\kappa$-solution which is noncompact, $\alpha$-uniformly PIC, and strictly PIC2. Given any $\varepsilon > 0$, there exists a constant $\theta := (n, \alpha, \varepsilon)> 0$ with the following property. If $(x,t)$ is a spacetime point satisfying $\lambda_1(x, t) \leq \theta R(x,t)$, then $(x, t)$ lies at the center of an evolving $\varepsilon$-neck. 
\end{lemma}
\begin{proof}
We have the tools to follow the standard contradiction argument. If the claim is not true, then we can find a sequence of counterexamples $(M^{(j)}, g_j(t), x_j, 0)$ satisfying $\lambda_1(x_j, 0) \leq \frac{1}{j} R(x_j, 0)$, but $(x_j, 0)$ does not lie at the center of an evolving $\varepsilon$-neck. By Corollary 6.15 in \cite{Bre19} or Corollary 3.7 in \cite{CZ06} if $n = 4$ (compactness of ancient $\kappa$-solutions), after rescaling each solution so that $R(x_j, 0) = 1$, we can extract a subsequence that converges smoothly to another ancient $\kappa$-solution, $(M^{(\infty)}, g_{\infty}(t), x_{\infty}, 0)$ satisfying $\lambda_1(x_{\infty}, 0) = 0$. The limit is $\alpha$-uniformly PIC and weakly PIC2. By Corollary 6.7 in \cite{Bre19}, either the limit is strictly PIC2 or its universal cover is isometric to a family of shrinking cylinders. If the limit were strictly PIC2, we would contradict $\lambda_1(x_{\infty}, 0) = 0$, and so we conclude the limit flow $(M^{(\infty)}, g_{\infty}(t))$ is a noncompact quotient of a family of shrinking round cylinders $(S^{n-1} \times \mathbb{R})/\Gamma$. A similar argument works if $n = 4$ using Lemma 3.2 in \cite{CZ06} (note that the ``restricted isotropic curvature'' pinching condition implies the curvature operator is nonnegative). 

It remains to show $\Gamma$ must be trivial. Since for each $j$, $(M^{(j)}, g_j(t))$ is a complete, noncompact Riemannian manifold with strictly positive sectional curvature, the soul theorem implies $M^{(j)}$ is diffeomorphic to $\mathbb{R}^n$. On the other hand, since the curvature tensor of $(M^{(j)}, g_j(t))$ lies within the interior of the PIC2 cone, Theorem A.2 in \cite{Bre18} implies that for $j$ sufficiently large, the fundamental group of the quotient of $(S^{n-1} \times \{0\})/\Gamma$ injects into fundamental group of $M^{(j)}$. This implies the group of isometries $\Gamma$ must be trivial, and thus $(M^{(\infty)}, g_{\infty}(t))$ is a family of round shrinking cylinders. For $j$ large, this implies $(x_j, 0)$ lies at the center of an evolving $\varepsilon$-neck, a contradiction to our assumption, which completes the proof of the lemma. 
\end{proof}

We used in our proof of the main theorem that the ancient solution $(M, g(t))$ is a Type II ancient solution. In dimension three, this is a consequence of the Neck Stability Theorem of Kleiner and Lott \cite{KL17}. Their argument works in higher dimensions as well. As in \cite{Bre20}, we provide a brief proof here for the reader. 

\begin{proposition}\label{type_ii}
Suppose $n \geq 4$ and let $(M, g(t))$ be an $n$-dimensional ancient $\kappa$-solution which is noncompact, $\alpha$-uniformly PIC, and strictly PIC2. Then $(M, g(t))$ is a Type II ancient solution; i.e. $\limsup_{t \to -\infty} (-t) R_{\max}(t)  = \infty$. 
\end{proposition}
\begin{proof}
Suppose for sake of contradiction that $\limsup_{t \to -\infty} (-t)R_{\max}(t) < \infty$. By the proof of Theorem A.4 in \cite{Bre20} (the argument there works independent of the dimension), there exists a point $q \in M$ such that $\sup_{t \leq 0} (-t) R(q, t) \leq C(n)$. Let $\ell(x, t)$ denote the reduced distance of $(x, t)$ from $(q, 0)$. It follows from the definition that $\ell(q, t) \leq C$. By work of Perelman, there exists a sequence of times $t_i \to -\infty$ such that the sequence of flows $(M, g_i(t), q)$, where $g_i(t) = (-t_i)^{-1} g((-t_i)t)$, for $t \in [-2, -1]$, converge in the pointed Cheeger-Gromov sense to a non-flat gradient shrinking Ricci soliton $(\hat M, \hat g(t))$ for $t \in [-2, -1]$. The limit is $\alpha$-uniformly PIC, weakly PIC2, and noncompact. It follows from the classification of such solitons that $(\hat M, \hat g(t))$ must be a family of shrinking round cylinders or a noncompact quotient of the family of cylinders by a discrete group of standard isometries. Arguing as in the proof of the previous lemma, we conclude $(\hat M, \hat g(t))$ must be a family of shrinking round cylinders. 

By convergence to the family of shirinking round cylinders the quantity $(-t_i) R(q, t_i)$ converges to a constant, which we may assume is $1$.  Moreover, for any $\varepsilon > 0$, the point $q$ lies at the center of an $\varepsilon$-neck at time $t_i$ if $i$ is sufficiently large. Since $   R(q, t_i) \geq \frac{1}{2}(-t_i)^{-1}$ if $i$ is sufficiently large, it follows that $R_{\max}(t_i)^{-\frac{1}{2}} \leq 2(-t_i)^{\frac{1}{2}}$ if $i$ sufficiently large. On the other hand, by assumption there exists a positive constant $\Lambda$ such that $R_{\max}(t) \leq \Lambda (-t)^{-1}$. Since $(M, g(t))$ has positive sectional curvature, this implies $\mathrm{Ric}_{g(t)} \leq \Lambda (-t)^{-1}g(t)$. Using the bound for the Ricci curvature, Perelman's distance distortion estimate (Lemma 8.3(b) in \cite{Per02}) implies for any points $x, y \in M$ and $t \leq s$
\[
d_{g(t)}(x, y) \leq d_{g(s)}(x, y) + C\Lambda (-t)^{\frac{1}{2}}. 
\]
In particular, fixing some $j$ large, and considering $i \geq j$ (so $t_i \leq t_j$), the inequality above implies
\[
d_{g_i(-1)}(x, y) = (-t_i)^{-\frac{1}{2}} d_{g(t_i)}(x, y) \leq (-t_i)^{-\frac{1}{2}} d_{g(t_j)}(x, y) + C\Lambda \leq d_{g_j(-1)}(x, y) + C \Lambda. 
\]
For $i$ sufficiently large, let $\Sigma_i$ denote the constant mean curvature sphere passing through $q$ at time $t_i$. Let $D_i$ denote the bounded connected component of $M \setminus \Sigma_i$. The estimate above implies that the diameter of the region $D_i$ with respect to the metric $g_i(-1)$ remains bounded as $i \to \infty$. In particular, the distance between $q$ and a point $p_i \in D_i$ of maximal distance is bounded. This however gives a contradiction, since for $i$ very large the length of the neck centered at $q$ at time $t_i$ becomes arbitrarily long with respect to the metric $g_i(-1)$ as $i \to \infty$. This completes the proof. 
\end{proof}

Finally, we verify that a uniformly PIC and strictly PIC2 steady soliton is asymptotically cylindrical in sense of \cite{Bre14}. 

\begin{proposition}\label{soliton_cylindrical}
Suppose $(M, g)$ is a complete steady gradient Ricci soliton of dimension $n \geq 4$ which is $\kappa$-noncollapsed, uniformly PIC and strictly PIC2. Then $(M, g)$ is asymptotically cylindrical in the sense of \cite{Bre14} and consequently the $n$-dimensional Bryant soliton. 
\end{proposition}
\begin{proof}
We adapt the argument used on p.212 of \cite{Bre14}. We may write $\mathrm{Ric} = D^2 f$ for some real-valued function $f$. Let $X = \nabla f$ and $\Phi_t$ denote the flow of the vector $-X$ for $t \in \mathbb{R}$. Recall that the corresponding solution of the Ricci flow is given $g(t) = \Phi_t^\ast (g)$. In particular, a time derivative correspond to a directional derivative by $-X$. After scaling, we may assume that $R + |\nabla f|^2 = 1$. We recall from \cite{Bre14} that $(M, g)$ is asymptotically cylindrical if:
\begin{enumerate}
\item[(i)] There exists a point $p$ and a positive constant $\Lambda$ such that the scalar curvature satisfies $\Lambda^{-1} d_g(p, x)^{-1} \leq R(x) \leq \Lambda d_g(p, x)^{-1}$. 
\item[(ii)] Let $x_m$ be any arbitrary sequence of points tending to infinity and let $r_m$ be a sequence of real numbers such that $r_m R(x_m) \to \frac{n-1}{2}$ as $m \to \infty$. Consider the rescaled metrics
\[
\hat g_m(t) = r_m^{-1} \Phi^\ast_{r_m t}(g) = r_m^{-1} g(r_m t). 
\]
Then we require that some subsequence of the sequence of flows $(M, \hat g_m(t), x_m)$ converges in the Cheeger-Gromov sense to a family of shrinking round cylinders $(S^{n-1} \times \mathbb{R}, \bar g(t))$ for $t \in (0, 1)$ where $\bar g(t) = 2(n-2)(1 - t) g_{S^{n-1}} + dz \otimes dz.$. 
\end{enumerate}
Consider a sequence of points $x_m$ tending to infinity. From our pointwise derivative estimates, we obtain $|\langle X, \nabla R \rangle| \leq O(1) R^2$. Suppose that $d_g(p, x_m)^2 R(x_m) \to \infty$. Combining Perelman's splitting result (Proposition 6.8 in \cite{Bre19}) together with compactness for ancient $\kappa$-solutions (Corollary 6.15 in \cite{Bre19} for $n \geq 5$ and Corollary 3.7 in \cite{CZ06} for $n = 4$), after passing to a subsequence, the sequence of flows $(M, R(x_m)g(t), x_m)$ converges to an ancient $\kappa$-solution which must split a line. As we have previously argued, the pinching result of \cite{BHS11} and together with the fact that $M$ is diffeomorphic to $\mathbb{R}^n$, implies the limit is a family of shrinking round cylinders. In particular, up to a constant factor, the sequence $(M, R(x_m)g, x_m)$ converges to the standard $S^{n-1} \times \mathbb{R}$ in the Cheeger-Gromov sense. This implies $|\nabla R| \leq o(1) R^{\frac{3}{2}}$, $|\Delta R| \leq o(1) R^2$, and $|\mathrm{Ric}|^2 = (\frac{1}{n} + o(1))R^2$ at $x_m$ as $m \to \infty$. From this we deduce that $|\langle X, \nabla R \rangle + \frac{2}{n} R^2| \leq o(1) R^2$ at $x_m$ as $m \to \infty$. 

This argument above shows that $R \to 0$ at infinity. Indeed, if there exist a sequence $x_m$ tending to infinity such that $R(x_m)$ is bounded away from zero, then $d(p, x_m)^2 R(x_m) \to \infty$. From the argument above, we have $|\nabla R| \leq o(1) R^{\frac{3}{2}}$ and $|\langle X, \nabla R \rangle + \frac{2}{n} R^2| \leq o(1) R^2$ at $x_m$ as $m \to \infty$. Since $|X| \leq 1$, together these estimates imply $R(x_m)$ tends to zero, a contradiction to our assumption. 

Now that we know $R \to 0$ at infinity, there exists a point $p \in M$ where $R$ attains its maximum. By the trace second Bianchi identity, $\nabla R =  -2 \mathrm{Ric}(\nabla f)$, the point $p$ must be a critical point of $f$. Using the second variation formula for geodesics and the fact that $\mathrm{Ric} = D^2 f$, convexity of the potential $f$ implies $f(x) \geq a d(x, p) - b$ for suitable constants $a$ and $b$. On the other hand, the estimate $|\nabla f|^2 \leq 1$ implies $f(x) \leq a d(x, p) + b$. Integrating the inequality $|\langle X, \nabla R \rangle| \leq O(1) R^2$, we obtain $R(x) \geq \Lambda^{-1} d(p, x)^{-1}$ for a suitable constant $\Lambda$. Consequently, $d(p,x)^2 R(x) \to \infty$ as $x \to \infty$ and we have the estimate $|\langle X, \nabla R \rangle + \frac{2}{n} R^2| \leq o(1) R^2$. Integrating this gives the upper bound $R(x) \leq \Lambda d(p, x)^{-1}$. Applying the convergence argument above once more shows that $(M, g)$ is asymptotically cylindrical (and therefore the $n$-dimensional Bryant soliton). 
\end{proof}

\section{Replacement for Anderson-Chow Estimate}

Recall that the Lichnerowicz Laplacian $\Delta_{L, g}$ is defined on $(0,2)$-tensors by
\[
\Delta_{L, g} h_{ik} = \Delta h_{ik} + 2 R_{ijkl} h^{jl} - \mathrm{Ric}_{il} h^l_k - \mathrm{Ric}_{kl} h^l_i.
\]
Recall that if $g(t)$ is a solution of the Ricci flow, then the covariant time derivative $D_t$ acts on time-dependent $(0, 2)$-tensors by 
\[
D_t h_{ik} = \frac{\partial}{\partial t} h_{ik} + \mathrm{Ric}_{il} h^l_k + \mathrm{Ric}_{kl} h^i_k. 
\]
In particular, $D_t g = 0$. 

\begin{proposition}\label{Anderson_Chow_replacement}
Suppose $n \geq 4$ and let $(M, g(t))$, $t \in [0, T]$ be an $n$-dimensional, complete solution to the Ricci flow with nonnegative sectional curvature. Suppose $\Omega \subset M \times (0, T)$ is a compact domain of spacetime. Let $\rho$ be a positive real number and assume $\mathrm{Ric} > \rho g$ everywhere in $\Omega$. Let $h(t)$ be a one-parameter family of symmetric (0,2)-tensors satisfying the parabolic Lichnerowicz equation $\frac{\partial}{\partial t} h(t) = \Delta_{L, g(t)} h(t)$ in $\Omega$.  Define $\psi : \Omega \to \mathbb{R}$ by
\[
\psi := e^{2\rho t} \inf \{ \lambda > 0 : -\lambda(\mathrm{Ric} - \rho g) \leq h \leq \lambda(\mathrm{Ric} - \rho g)\}.
\]
Then 
\[
\sup_{\Omega} \psi \leq \sup_{\mathcal P \Omega }\psi, 
\]
where $\mathcal P \Omega$ denotes the parabolic boundary of $\Omega$. 
\end{proposition}

\begin{proof}
Let
\[
\Lambda = \sup_{ \Omega} \psi. 
\]
Suppose for sake of contradiction that $\sup_{\mathcal P \Omega} \psi < \Lambda$. Let
\[
\bar t := \inf \{ t : \psi(x, t) = \Lambda \text{ for some } (x, t) \in \Omega \}. 
\]
There exists $\bar x$ such that $\psi(\bar x, \bar t) = \Lambda$. Since $\psi < \Lambda$ on $\mathcal P \Omega$, we have $(\bar x, \bar t) \in \Omega \setminus \mathcal P \Omega$ and if $(x, t) \in \Omega$ with $t < \bar t$, then $\psi(x, t) < \Lambda$. Because $(\bar x, \bar t)$ is an interior point, we can find a smooth open set $U$ containing $\bar x$ and a small positive number $\delta > 0$ such that $D = U \times [\bar t- \delta, \bar t] \subset \Omega$. After a suitable translation in time, we may assume $\bar t - \delta = 0$. 

Consider $(0, 2)$-tensors $T^{(\pm)}(t) = e^{-2\rho t} \Lambda (\mathrm{Ric}_{g(t)}- \rho g(t)) \pm h(t)$ in the parabolic domain $D$. Since 
\[
-e^{-2\rho t} \psi (\mathrm{Ric} - \rho g) \leq h \leq e^{-2\rho t} \psi(\mathrm{Ric} - \rho g)
\]
and $0 \leq \psi \leq \Lambda$, the tensors $T^{(\pm)}$ are weakly positive definite in $D$. Moreover, by our choice of $\bar t$, both tensors $T^{(\pm)}$ are strictly positive definite in $U$ at time $t = 0$. Finally, by definition of $\Lambda$, there exists a vector $v \in T_{\bar x} M$ such that $|v|_{g(\bar t)} = 1$ and either $T^{(+)}_{ij}v_iv_j = 0$ or $T^{(-)}_{ij}v_iv_j = 0$. Without loss of generality, suppose $T^{(-)}_{ij}v_iv_j = 0$ and let us write $T= T^{(-)}$. Now since $T$ is initially strictly positive definite, we can find a nonnegative function $f : \overline U \to \mathbb{R}$ with the properties that $f > 0$ in $U$, $f = 0$ on $\partial U$, and $T(0) - f g(0)$ is weakly positive definite on $U$. Now let $F : U \times [0, \bar t]$ be a solution of the linear heat equation with respect to the evolving metric $g(t)$ with initial condition $F(\cdot, 0) = f$ and Dirichlet boundary condition $F = 0$ on $\partial U \times [0, \bar t]$.  By the maximum principle, $F > 0$ in $U \times [0, \bar t]$. Then $(D_t - \Delta )(F g) = 0$. 
For ease of notation, let $(R\ast T)_{ik} = R_{ijkl} T^{jl}$. Since we assume $(M, g(t))$ has nonnegative sectional curvature and $T$ is weakly positive definite, we have $R \ast T \geq 0$. Now we compute
\begin{align*}
\Big(D_t - \Delta \Big) T &= - 2 \rho e^{-2\rho t} \Lambda (\mathrm{Ric} - \rho g) + 2 \rho e^{-2\rho t} \Lambda \mathrm{Ric}  + 2 e^{-2\rho t} R \ast T\\
& = 2 \rho^2 e^{-2\rho t} \Lambda g +  2 e^{-2\rho t} R \ast T\geq 0.
\end{align*}
Thus, by the maximum principle for tensors $T - Fg$ is weakly positive definite in $U \times [0, \bar t]$. Since $F(\bar x, \bar t) > 0$, this implies $T$ is strictly positive definite at $(\bar x, \bar t)$, contradicting the fact that $T_{ij}v_iv_j = 0$. This completes the proof of the proposition. 
\end{proof}

\section{Interior Estimates for Linear Parabolic Equations}\label{parabolic_equations}

In this section, for the convenience of the reader we state some results from linear parabolic equations used in the proof of the main theorem. Unless otherwise noted, in this section $C$ always denotes a universal constant, which may depend upon the dimension. 

In our analysis of the parabolic Lichnerowicz equation, we used the representation formula for a solution of the one-dimensional heat equation on $[-L, L]$ with Dirichlet boundary conditions. By the method of reflection, one can compute that, for $t \geq 0$, the fundamental solution for the one-dimensional heat equation on $[-L, L]$ with Dirichlet boundary conditions is given by 
\[
S(z, t; w) := \frac{1}{\sqrt{4\pi t}} \bigg(\sum_{k \in \mathbb{Z}} \exp\Big( -\frac{|z - w + 4kL|^2}{4t}\Big) - \sum_{k \in \mathbb{Z}} \exp\Big( -\frac{|z + w + 4kL- 2L|^2}{4t}\Big) \bigg).
\]
Thus, if $u$ is a solution of the one-dimensional heat equation for $z \in [-L,L]$ and $t \in [-L, 0]$, then
\[
u(z, t) = \int_{-L}^{L} S(z, t+ L; w) \,u(w, -L) \, dw  + \int_{-L}^{t} \frac{\partial S}{\partial w}(z, t - s; L) \, u(L , s)  -  \frac{\partial S}{\partial w}(z, t - s; -L) \, u(-L , s)\, ds.
\]
Suppose now that $(z, t) \in  [-4000, 4000] \times [-4000, 0]$ and $L$ is large positive constant. From our expression for $S$, we can readily obtain estimates
\[
\int_{-L}^{L} |S(z, t+ L; w)| \, dw \leq  C, \quad \text{and} \quad \Big|\frac{\partial S}{\partial w} (z, t-s; \pm L) \Big| \leq C L (t-s)^{-\frac{3}{2}} e^{-\frac{L^2}{100(t-s)}}. 
\]
Combining these estimates with the representation formula gives us the following method of estimating a solution by its boundary and initial data.  
\begin{proposition}
Suppose $u$ is a solution of the one-dimensional heat equation in the rectangle $[-L, L] \times [-L, 0]$. If $L$ is a large positive constant and $(z, t) \in [-4000, 4000] \times [-4000, 0]$, then 
\[
|u(z, t)| \leq C\sup_{z \in [-L, L]} \big|u(z, -L)\big| + C L \int_{-L}^t e^{-\frac{L^2}{100(t-s)}}(t-s)^{-\frac{3}{2}} \big(|u(-L, s)| + |u(L, s)|\big)\, ds. 
\]
\end{proposition}

In the course of this paper, we have used standard interior estimates for tensors satisfying linear parabolic equations coupled to the Ricci flow. In particular, we often required an estimate that only uses an $C^0$ bound for the inhomogeneous term $F$. This estimate is not an immediate consequence of standard interior parabolic Schauder estimates, which bound the solution by the $C^{\alpha}$-norm of the inhomogeneous term, but we can deduce it using Duhamel's principle. For the convenience of the reader, we sketch the proof of a result here that is sufficient for our applications. In particular, the following result applies to vector fields that satisfy $\frac{\partial}{\partial t} V = \Delta V + \mathrm{Ric}(V) + F$ and $(0,2)$-tensors that satisfy $\frac{\partial}{\partial t} h = \Delta_{L, g(t)} h + F$, as we encountered in the proofs of our results. For notational convenience, let $P(x_0, t_0, r) := B_{g(t_0)}(x_0, r) \times [t_0 - r^2, t_0]$ denote the parabolic cylinder of radius $r$ centered at $(x_0, t_0)$. The result is modeled on Shi's local interior estimates for the curvature tensor. 

\begin{proposition}
Let $U$ be an open neighborhood of a point $x_0$ in an $n$-dimensional manifold $M$. Let $r > 0$. Suppose that $g(t)$ is a family of Riemannian metrics on $U$ evolving by Ricci flow for $t_0 - r^2 \leq t \leq t_0$ and suppose that the closed metric ball $\overline{B}_{g(t_0)}(x_0, r)$ is a compact subset of $U$. Let $R = R_{g(t)}$ denote the Riemann curvature tensor of $g(t)$. Let $H$ and $F$ be smooth time-dependent tensor fields defined on $P(x_0, t_0, r)$. Suppose that $|R| \leq r^{-2}$ and 
\[
\frac{\partial }{\partial t} H = \Delta H + R \ast H +  F
\]
both hold on $P(x_0, t_0, r)$, where $R \ast H$ represents a bilinear expression in the curvature tensor and $H$. Then there exist a positive constant $C$, depending only upon the dimension $n$ and the specific bilinear expression $R \ast H$, such that 
\[
|DH|(x_0, t_0) \leq C r^{-1}\sup_{P(x_0, t_0, r)} |H| +  Cr \sup_{P(x_0, t_0, r)} |F|.
\]
\end{proposition}
\begin{proof}
Throughout the course of the proof, $C$ denotes a constant which may change from line to line, but only depends upon the dimension. Without loss of generality, we may assume $t_0 = r^2$. After replacing $r$ by $2r$, Shi's interior derivative estimates gives us the bound 
\[
\sup_{P(x_0, r^2, r)} |DR| \leq Cr^{-3}. 
\]
By passing to a local cover and pulling back the local solution of the Ricci flow, we may assume the exponential map at $x_0$ at time $r^2$ is injective on the ball of radius $r$. As in Section 13 of \cite{Ham97}, we can construct a suitable, smooth, time-independent function of the radius, $\varphi$, which is compactly supported in $B_{g(r^2)}(x_0, r)$ and has the properties $\varphi(x_0) = 1$, 
\[
0 \leq \varphi \leq C, \quad |D\varphi|_{g(r^2)} \leq C r^{-1}, \quad |D^2 \varphi|_{g(r^2)} \leq C r^{-2}. 
\]
Since 
\[
\Big|\frac{\partial}{\partial t} |D\varphi|^2\Big| \leq 2 |\mathrm{Ric}| |D\varphi|^2 \leq Cr^{-2} |D\varphi|^2,
\] 
integrating over $t \in [0, r^2]$, we conclude $|D\varphi| \leq C r^{-1}$ in $P(x_0, r^2, r)$. Similarly, $|D^2 \varphi| \leq C r^{-2}$ holds in $P(x_0, r^2, r)$. Let 
\[
\Lambda:= \sup_{P(x_0, t_0, r)} |H| 
\]
We consider the homogeneous and inhomogeneous cases separately. \\

\textit{Step 1:} Suppose that $F \equiv 0$. From the evolution equation for $H$, we obtain
\[
\frac{\partial}{\partial t} DH = \Delta DH + DR \ast H + R \ast DH. 
\]
We also have
\begin{align*}
\frac{\partial}{\partial t} |H|^2 &\leq \Delta |H|^2 - 2|DH|^2 + C|R| |H|^2 \\
&\leq \Delta |H|^2 - 2|DH|^2 + Cr^{-2} \Lambda^2. 
\end{align*}
Similarly,
\begin{align*}
\frac{\partial}{\partial t} |DH|^2 &\leq \Delta |DH|^2 - 2 |D^2 H|^2 + C|DR||H||DH| + C|R| |DH|^2 \\
& \leq  \Delta |DH|^2 - 2 |D^2 H|^2 + Cr^{-4}\Lambda^2  + Cr^{-2} |DH|^2. 
\end{align*}
This implies 
\begin{align*}
\frac{\partial}{\partial t} \big(\varphi^2 |DH|^2\big) &\leq \Delta \big(\varphi^2 |DH|^2\big) + C\varphi |D\varphi| |DH| |D^2 H| + C\varphi |D^2 \varphi| |DH|^2 + C |D\varphi|^2 |DH|^2 \\
& \qquad - 2\varphi^2 |D^2 H|^2 + C\varphi^2 r^{-4}\Lambda^2  + C\varphi^2 r^{-2} |DH|^2.
\end{align*}
By Young's inequality, we can bound
\[
2\varphi |D\varphi| |DH| |D^2 H| \leq \varepsilon \varphi^2 |D^2 H|^2 + \varepsilon^{-1} |D\varphi|^2 |DH|^2,
\]
where $\varepsilon>0$ can be chosen arbitrarily. Putting this together with the estimates for $\varphi$ and its derivatives, we obtain
\[
\frac{\partial}{\partial t} \big(\varphi^2 |DH|^2\big) \leq \Delta \big(\varphi^2 |DH|^2\big) + C r^{-4}\Lambda^2  + C r^{-2} |DH|^2. 
\]
Now, in a similar fashion to \cite{Bre10}, we consider the function 
\[
G = t \varphi^2 |DH|^2 + B|H|^2
\]
for a large positive constant $B$ to be determined momentarily. We compute 
\[
\frac{\partial}{\partial t} G \leq \Delta G + \big(\varphi^2 +Ctr^{-2} -2B\big)|DH|^2 + Cr^{-2}  (tr^{-2} + B) \Lambda^2. 
\]
Since $t \leq r^2$, if we take $B$ suitably large depending upon the constant $C$ and the estimate for $\varphi$, we conclude that 
\[
\frac{\partial}{\partial t} G \leq \Delta G + Cr^{-2} \Lambda^2. 
\]
It follows from the maximum principle that $G \leq C \Lambda^2$ on $P(x_0, r^2, r)$, and this gives 
\[
|DH|(x_0, r^2) \leq C r^{-1} \Lambda. 
\] 

\textit{Step 2:} If $F \not\equiv 0$, then we use Duhamel's principle together with the interior estimates derived in the homogeneous case. To fix the particular bilinear expression $R \ast H$, let us write $LH$ in place of $\Delta H + R \ast H$. 

Let $\tilde U \subset M$ be a smooth domain such that $\overline{B}_{g(r^2)}(x_0, \frac{r}{2}) \subset \tilde U \subset B_{g(r^2)}(x_0, r)$. For each $s \in [0, r^2)$, let $\tilde H^{(s)}$ be the solution to the problem 
\[
\begin{cases} 
\frac{\partial}{\partial t} \tilde H^{(s)}(x, t) = L \tilde H^{(s)}(x, t) & (x, t) \in U \times [s, r^2], \\
 \tilde H^{(s)}(x, s) = F(x, s) & x \in U,\\
 \tilde H^{(s)}(x, s) =  0 & x \in \partial U \times [s, r^2]. 
\end{cases} 
\]
Since 
\[
\frac{\partial}{\partial t} |\tilde H^{(s)}|^2 \leq \Delta |\tilde H^{(s)}|^2 + C r^{-2} |\tilde H^{(s)}|^2, 
\]
and $s \leq t \leq r^2$, by the maximum principle, we obtain 
\[
\sup_{U \times [s, r^2]}|\tilde H^{(s)}| \leq C \sup_{P(x_0, r^2, r)} |F|.
\]
For $(x, t) \in U \times [0, r^2]$, define 
\[
\tilde H (x, t) = \int_0^t \tilde H^{(s)}(x, t) \, ds.  
\]
Then $\frac{\partial}{\partial t} \tilde H(x, t) = F(x, t) + \int_0^t \frac{\partial}{\partial t} \tilde H^{(s)} (x, t) \, ds $ and consequently 
\[
\frac{\partial}{\partial t} \tilde H = L \tilde H + F
\]
in $\tilde U \times [0, r^2]$. Observe that $\tilde H$ vanishes on the parabolic boundary of $\tilde U \times [0, r^2]$.  

Let $\hat H(x, t) = H(x, t) - \tilde H(x, t)$. It follows that $\hat H$ solves $\frac{\partial}{\partial t} \hat H= L \hat H$ in $U \times [0, r^2]$. Moreover, 
\[
\sup_{U \times [0, r^2]} |\hat H| \leq \sup_{P(x_0, r^2, r)}|H| + C r^2 \sup_{P(x_0, r^2, r)} |F|.
\]
Applying the interior estimate of Step 1 to $\tilde H^{(s)}$, we get for $t \in (s, r^2)$
\[
|D\tilde H^{(s)}|(x_0, t) \leq C(t -s)^{-\frac{1}{2}} \sup_{U \times [s, r^2]} |\tilde H^{(s)}| \leq C (t - s)^{-\frac{1}{2}} \sup_{P(x_0, r^2, r)} |F|. 
\]
Integrating over $0 \leq s \leq r^2$, we get 
\[
|D\tilde H|(x_0, r^2) \leq C r  \sup_{P(x_0, r^2, r)}|F|. 
\]
Applying the interior estimate to $\hat H$, we obtain 
\[
|D\hat H|(x_0, r^2) \leq C r^{-1}  \sup_{U \times [0, r^2]} |\hat H| \leq  C r^{-1}  \sup_{P(x_0, r^2, r)}|H|  + C r \sup_{P(x_0, r^2, r)}|F|. 
\]
Combining the estimates for $|D\tilde H|(x_0, r^2)$ and $|D\hat H|(x_0, r^2)$ completes the proof of the proposition.
\end{proof}

The proof of the following corollary is similar. 
\begin{corollary}
Under the assumptions of Proposition C.2, for every positive integer $k \geq 2$, there exists a constant $C_k$, depending only upon the dimension and the specific bilinear expression $R \ast H$, such that 
\[
|D^k H|(x_0, t_0) \leq C_k r^{-k} \sup_{P(x_0, t_0, r)} |H| + C_k \sum_{l = 0}^{k -1} r^{l-k+2}\sup_{P(x_0, t_0, r)} |D^{l} F|.
\]
\end{corollary}

\end{document}